\documentclass[leqno]{amsart}
\usepackage{amssymb}
\usepackage{amsmath}
\usepackage{amsthm} 
\usepackage{mathtools}
\usepackage{color}
\usepackage[svgnames]{xcolor}
\usepackage[utf8]{inputenc}
\usepackage[T1]{fontenc}
\usepackage{graphicx}
\usepackage{placeins}
\usepackage{vmargin}
\usepackage{setspace}
\usepackage{stmaryrd}

\numberwithin{equation}{section}
\numberwithin{figure}{section}

\newtheorem{theorem}{Theorem}[section]
\newtheorem{proposition}[theorem]{Proposition}
\newtheorem{lemma}[theorem]{Lemma}

\newtheorem{corollary}[theorem]{Corollary}

\theoremstyle{definition}
\newtheorem{definition}[theorem]{Definition}
\newtheorem{property}[theorem]{Property}
\newtheorem{assumptions}{Assumptions}
\newtheorem{assumption}[assumptions]{Assumption}
\theoremstyle{remark}
\newtheorem{remark}[theorem]{Remark}
\newtheorem{claim}[theorem]{Claim}

\newcommand{\norm}[1]{\left\lVert #1\right\rVert}

\newcommand{\eps}{\varepsilon}

\newcommand{\vp}{\varphi}

\newcommand{\N}{\mathbb{N}}
\newcommand{\R}{\mathbb{R}}
\newcommand{\C}{\mathbb C}

\newcommand{\FFF}{\mathcal{F}}
 
\newcommand{\JJJ}{\mathcal{J}}
\newcommand{\NNN}{\mathcal{N}}
\newcommand{\PPP}{\mathcal{P}}
\newcommand{\RRR}{\mathcal{R}}

\newcommand{\tvp}{\tilde{\vp}}

\newcommand{\tvarphi}{\tilde{\varphi}}
\newcommand{\tw}{\tilde{w}}

\newcommand{\ent}[1]{\left\lfloor #1 \right\rfloor}
\newcommand{\supent}[1]{\left\lceil #1 \right\rceil}
\renewcommand{\leq}{\leqslant}
\renewcommand{\geq}{\geqslant}

\DeclareMathAlphabet{\mathpzc}{OT1}{pzc}{m}{it}
\renewcommand{\Re}{\mathcal R\!\mathpzc{e}}
\renewcommand{\Im}{\mathcal I\!\mathpzc{m}}

\begin{document}

\title[Scattering theory for NLS]{Profile decomposition and scattering for general nonlinear Schr\"odinger equations}

\author[Thomas Duyckaerts]{Thomas Duyckaerts}
\address[Thomas Duyckaerts]{LAGA (UMR 7539),
\newline\indent
Institut Galil\'ee, Universit\'e Sorbonne Paris Nord,
  \newline\indent
  99 avenue Jean-Baptiste Cl\'ement,
  \newline\indent
  93430 Villetaneuse,
  France
  \newline\indent
  and
  Département Math\'ematiques et Application,
  \newline\indent
  \'Ecole Normale Sup\'erieure
  \newline\indent
  45 rue d'Ulm
  \newline\indent
  75005 Paris,
  France}
\email[Thomas Duyckaerts]{duyckaer@math.univ-paris13.fr}

\author[Phan Van Tin]{Phan Van Tin}
\address[Phan Van Tin]{LAGA (UMR 7539),
\newline\indent
Institut Galil\'ee, Universit\'e Sorbonne Paris Nord,
\newline\indent
  99 avenue Jean-Baptiste Cl\'ement,
  \newline\indent
  93430 Villetaneuse,
  France}
\email[Phan Van Tin]{vantin.phan@math.univ-paris13.fr}


\date{\today}

\begin{abstract}
We consider a Schr\"odinger equation with a nonlinearity which is a general perturbation of a power nonlinearity.
We construct a profile decomposition adapted to this nonlinearity. We also prove global existence and scattering in a general defocusing setting, assuming that the critical Sobolev norm is bounded in the energy-supercritical case. This generalizes several previous works on double-power nonlinearities.
\end{abstract}

\maketitle

\tableofcontents
\section{Introduction}
\subsection{General setting}
This article concerns the nonlinear Schr\"odinger (NLS) equation
\begin{equation}
 \label{NLS}
 i\partial_tu+\Delta u=g(u),
\end{equation} 
in space dimension $d\geq 1$,
where $g(u)$ is a $L^2$-supercritical nonlinearity of the form 
\begin{equation}
\label{def_g_intro}
g(u)=\iota_0|u|^{p_0}u+g_1(u), \iota_0\in \{\pm 1\} 
\end{equation} 
and $g_1:\C\to \C$ is a $L^2$-supercritical lower-order term, i.e. $g_1$ is $C^{k_0}$ and
\begin{equation*}
 \exists p_1,p_2,\; \frac{4}{d}<p_1\leq p_2<p_0,\quad \forall u\in \C,\; \forall k\in \llbracket 0,k_0\rrbracket,\quad 
 |D^kg_1(u)|\lesssim |u|^{p_2-k}+|u|^{p_1-k},
\end{equation*} 
for some $k_0$ that we will be specified later. 

The model case for $g$ is a sum of $k+1$ powers, $k\geq 1$
\begin{equation}
 \label{multi_power}
 g(u)=\sum_{j=0}^k \iota_j|u|^{p_j}u, \quad \iota_j\in \R\setminus\{0\},\quad  \frac{4}{d}<p_{k}<\ldots<p_1<p_0.
\end{equation}

With the regularity assumption: 
\begin{equation}
\label{regul}
\forall j\in \{0,\ldots,k\},\; p_j\text{ is an even integers, or }\supent{s_0}< p_j,
\end{equation}
where $s_0=\frac d2-\frac{2}{p_0}$
(see Subsection \ref{sub:notations} for the notation $\supent{s_0}$ and other notations that will be used in this introduction). The case of a double-power, energy-subcritical (i.e. $s_0\leq 1$) nonlinearity was studied in many work. Our goal is give a general setting for the study of \eqref{NLS} which includes also the supercritical case $s_0>1$ and more general lower order nonlinearity $g_1$.

We are interested in the global well-posedness and scattering for solutions of \eqref{NLS}. Neglecting the lower order term in $g$, we obtain the usual (NLS) equation with a single power nonlinearity
\begin{equation}
 \label{NLSh}
 i\partial_tu+\Delta u=\iota_0 |u|^{p_0}u,
\end{equation}

The equation \eqref{NLSh} is invariant by scaling: if $u$ is a solution of \eqref{NLSh} and $\lambda>0$, then $u_{\lambda}(t,x)=\lambda^{2/p_0}u(\lambda^2t,\lambda x)$ is also a solution of \eqref{NLSh}. The critical Sobolev exponent $s_0$ for \eqref{NLSh} is the unique $s_0$ such that $\|u\|_{\dot{H}^{s_0}}=\|u_{\lambda}(0)\|_{\dot{H}^{s_0}}$ for all $\lambda>0$. The equation \eqref{NLSh} is well-posed in $\dot{H}^{s_0}$ (with additional technical conditions in high dimensions ensure a minimal regularity of the nonlinearity) see \cite{CaWe90}.

\subsection{Well-posedness and profile decomposition}

In Section \ref{S:preliminary} we prove that \eqref{NLS} is locally well-posed in the inhomogenous space $H^{s}$, for any $s\geq s_0$, assuming that $g\in C^{\supent{s}+1}$ and $L^2$-supercritical (see Assumption \ref{Assum:NL} p.\pageref{Assum:NL} for the precise assumptions).  We also develop a full stability/long time perturbation theory for \eqref{NLS}.
The existence and uniqueness of solutions yields for all $u_0\in H^{s}$ a maximal interval of existence $I_{\max}(u_0)=(T_{-}(u_0),T_{+}(u_0))$.
Assuming $g(u)=G'(|u|^2)u$ for some $C^1$ function $G$, with $G(0)=0$ we also have conservation of the mass:
\begin{equation*}
M(u(t))=\int_{\R^d} |u(t,x)|^2dx,
\end{equation*} 
and, if $s\geq 1$, of the energy
\begin{equation*}
E(u(t))=\int_{\R^d}|\nabla u(t,x)|^2dx+\int_{\R^d}G(|u(t,x)|^2)dx.
\end{equation*}
and the momentum
\begin{equation*}
P(u(t))=\Im \int_{\R^d}\nabla u(t,x)\,\overline{u}(t,x)dx
\end{equation*}
of a solution. 

Our first main result is the construction of a profile decomposition adapted to bounded sequences of $H^{s_0}$ solutions of equation \eqref{NLS}, which builds up on the stability theory developed in Section \ref{S:preliminary}. This amounts to expressing such a sequence as a sum of three distinct types of objects: a dispersive behaving as a solution of the linear Schr\"odinger equation, concentrating (nonlinear) profiles that are solutions of \eqref{NLSh} rescaled with a scaling parameter going to $0$, and nonconcentrating profiles, that are solutions of \eqref{NLS}. We refer to Section \ref{S:profile}, and in particular Subsection \ref{sub:NL} and  Theorem \ref{T:NLapprox} for the detailed statements.

This profile decomposition is valid in the  general setting described above, and generalizes various previous constructions on double power nonlinearities (see Subsection \ref{sub:previous} below for references).

\subsection{Global well-posedness}

Solutions of \eqref{NLS} are not always global. Indeed, in the case of a double power nonlinearity, if $\iota_0=-1$ (the higher-order nonlinearity is focusing), a standard convexity argument (see \cite{Zakharov72,Gl77}) shows that any solution with negative energy and finite variance blows up in finite time (at least in the case of a double-power nonlinearity, \eqref{multi_power} with $k=2$). More surprisingly, Merle, Rapha\"el, Rodnianski and Szeftel \cite{MeRaRoSz22a} have constructed solutions of the homogeneous equation \eqref{NLSh} with a defocusing, energy-supercritical nonlinearity $s_0>1$, $\iota_0=1$ that blow up in finite time. It is known however in the defocusing case $\iota_0=1$, for many values of $p_0$,
that solution of \eqref{NLSh} that remains bounded in the critical Sobolev space are global and scatter. We will prove that this property implies that solutions of \eqref{NLS}, with $g$ of the form \eqref{def_g_intro} satisfying the same boundedness condition are global. We will thus consider:
\begin{property}
\label{Proper:bnd}
Let $A_0\in (0,\infty]$.
For any solution $u$ of \eqref{NLSh} with initial data in $\dot{H}^{s_0}$, if 
\begin{equation}
  \label{bound_Hs0}
  \limsup_{t\to T_{+}(u)} \|u(t)\|_{\dot{H}^{s_0}}<A_0.
\end{equation} 
Then $T_{+}(u)=+\infty$ and $u$ scatters for positive times in $\dot{H}^{s_0}$, i.e. there exists $v_0\in \dot{H}^{s_0}$ such that
$$\lim_{t\to\infty} \left\|u(t)-e^{it\Delta}v_0\right\|_{\dot{H}^{s_0}}=0.$$
\end{property} 
It is conjectured that Property \ref{Proper:bnd} always holds for $A_0=\infty$ in the defocusing case $\iota_0=1$. In the defocusing energy-critical case $d\geq 3$, $p_0=\frac{4}{d-2}$, Property \ref{Proper:bnd} is unconditional (the bound \eqref{bound_Hs0} is given by conservation of the energy), and was  proved in \cite{CoKeStTaTa08}, \cite{RyVi07} and \cite{Visan07},

The study of Property \ref{Proper:bnd} in the defocusing case for other critical exponents was initiated in \cite{KeMe10} where it was proved when $d=3$, $p_0=2$ (thus $s_0=1/2$), for radial solutions. It was later proved in many other cases:  see \cite{KillipVisan10}, \cite{MiXuZh14}, \cite{LuZheng17}, \cite{DoMiMuZh17} for supercritical nonlinearities in dimension $d\geq 4$ (with technical restriction if $d\geq 7$), and \cite{XieFang13}, \cite{Murphy14}, \cite{Murphy15}, \cite{GaoZhao19}, \cite{YuXueying21} and \cite{JaVaLu22} for several energy-subcritical nonlinearities.

In the focusing case $\iota_0=-1$, Property \ref{Proper:bnd} is only known to hold when $A_0$ is small, from the small data theory for equation \eqref{NLSh}. In the focusing energy critical case, $p_0=\frac{4}{d-2}$, it follows from \cite{KeMe06}, \cite{KiVi10} and \cite{Dodson19a} that Property \ref{Proper:bnd} holds with $A_0=\|\nabla W\|^2_{L^2}$ in dimension $d\geq 4$, and $d=3$ in the radial case, where $W$ is the ground state of equation \eqref{NLSh}.
This is optimal, since the existence of $W$ shows that Property \ref{Proper:bnd} does not hold for $A_0>\|\nabla W\|^2_{L^2}$.

When $\iota_0=-1$ and $s_0\in (0,1)$, there exists standing wave solutions of \eqref{NLSh}, so that Property \ref{Proper:bnd} does not hold for large $A_0$. When $\iota_0=-1$, $s_0>1$, travelling wave solutions in $\dot{H}^{s_0}$ do not exist, and the validity of Property \ref{Proper:bnd} for large $A_0$ is an open question. Let us also mention that the analogue of this property was proved for radial focusing nonlinear wave equation in the energy supercritical and subcritical settings (see e.g. \cite{DuKeMe12c} for supercritical $p_0$ in space dimension $3$).

Our result on global well-posedness is as follows:
\begin{theorem}
 \label{T:GWP} 
 Let $\iota_0$, $s_0$, $g$ such that Assumption \ref{Assum:profile}  holds, and such that Property \ref{Proper:bnd} is valid for some $A_0\in (0,\infty]$. Let $u$ be a solution of \eqref{NLS_g}, with initial data in $H^{s_0}$, such that \eqref{bound_Hs0} holds. Then $T_{+}(u)=+\infty$.
\end{theorem}
We refer to Definition \ref{D:NNN} and Section \ref{S:profile}, p.~\pageref{Assum:profile} for the details of Assumption \ref{Assum:profile}. Let us mention that a multi-power non-linearity as in \eqref{multi_power} with the additional technical assumption \eqref{regul} satisfies this assumption.

If Property \ref{Proper:bnd} holds only for radial functions, then the conclusion of Theorem \ref{T:GWP} is also valid when restricted to radial solutions of \eqref{NLS}. 

Theorem \ref{T:GWP} is new in the energy supercritical case $s_0>1$. In the energy-subcritical and energy-critical cases $0<s_0\leq 1$, it was proved in \cite{TaViZh07} for a double power non-linearity with the stronger assumptions $\iota_0=1$, $u_0\in H^1$, without assuming Property \ref{Proper:bnd}.

The proof of Theorem \ref{T:GWP} uses the profile decomposition mentioned above. The Property \ref{Proper:bnd} is used to deal with the concentrating profiles.

In the defocusing energy-critical case, by conservation of the energy and the scattering result for the energy-critical Schr\"odinger defocusing equation,  the global well-posedness is unconditional:
\begin{corollary}
\label{Cr:bnd} 
Assume $\iota_0=1$. Let $d\in \{3,4,5\}$, and $g$ such that Assumption \ref{Assum:profile} holds with $p_0=\frac{4}{d-2}$. Let $u$ be a solution of \eqref{NLS_g}, with initial data in $H^{1}$. Then $u$ is global.
\end{corollary}
Corollary \ref{Cr:bnd} was proved in the case of a double-power nonlinearity in \cite{Zhang06} ($d=3$) and 
in \cite{TaViZh07} (for general $d\geq 4$). Corollary \ref{Cr:bnd} generalizes these results to more general perturbations of the energy-critical nonlinearity, in low dimension. The restriction on the dimension is due to the regularity assumption $g\in C^2$ in Assumption \ref{Assum:profile}. This restriction could be weakened using a refined well-posedness/stability theory as in \cite{BuCzLiPaZh13}.

\subsection{Scattering}

Our next goal is to give sufficient conditions for scattering of solutions of \eqref{NLS}. We recall that a solution of \eqref{NLS} with initial data in $H^{s_0}$ is said to scatter (in $H^{s_0}$, forward in time) when $T_{+}(u)=+\infty$ and there exists $v_0\in H^{s_0}$ such that
$$\lim_{t\to\infty} \left\|e^{it\Delta}v_0-u(t)\right\|_{H^{s_0}}=0.$$
We will prove scattering for a general defocusing nonlinearity defined as follows:
\begin{definition}
\label{D:defocusing}
The nonlinearity $g$ is defocusing when it is of the form $g(u)=G'(|u|^2)u$ for some $C^1$ function $G$ such that for almost all $a>0$, $aG'(a)-G(a)>0$.
\end{definition}
Note that any power nonlinearity with a positive coefficient is defocusing in the sense of Definition \eqref{D:defocusing}.
For a multi-power nonlinearity $g$ as in \eqref{multi_power} the assumption that $g$ is defocusing is equivalent to
\begin{equation}
\label{general_defoc}
\forall s>0,\quad \sum_{j=0}^k\iota_j\frac{p_j}{p_j+2}s^{\frac{p_j}{2}-1}>0.
\end{equation}
Note that \eqref{general_defoc} holds when all the $\iota_j$ are positive. Also, \eqref{general_defoc} implies $\iota_0>0$, $\iota_k>0$.

For our scattering results, we distinguish between the energy-supercritical case and the energy-subcritical case:
\begin{theorem}
\label{T:defoc_defoc}
Let $d\geq 3$ and assume $s_0>1$. Let $g$ be a defocusing nonlinearity that satisfies Assumption \ref{Assum:profile}. Assume that Property \ref{Proper:bnd} holds for some $A_0\in (0,\infty]$. Let $u$ be a solution of \eqref{NLS} with initial data $u_0\in H^{s_0}$ and that satisfies \eqref{bound_Hs0}. Then $u$ is global and scatter forward in time.
\end{theorem}

When $s_0\leq 1$, we must further assume that the initial data has finite energy.
\begin{theorem}
\label{T:defoc_defoc'}
Assume $0<s_0\leq 1$. Let $g$ be a defocusing nonlinearity that satisfies Assumption \ref{Assum:profile}. Let $u$ be a solution of \eqref{NLS} with initial data $u_0\in H^1$. Then $u$ is global and scatter in both time directions.
\end{theorem}
Note that the assumption that $g$ is defocusing together with the fact that $G(a)/a$ goes to $0$ as $a$ goes to $0$ (which is a consequence of Assumption \ref{Assum:profile}) implies that $G\geq 0$. Thus the
assumptions of Theorem \ref{T:defoc_defoc'} and the conservation of mass and energy imply that any $H^1$ solution of $u$ is bounded in $H^1$.

Theorems \ref{T:defoc_defoc} and Theorem \ref{T:defoc_defoc'} show that scattering holds for a multi-power nonlinearity satisfying \eqref{multi_power}, \eqref{regul} and \eqref{general_defoc}. 

In the case of a double-power nonlinearity the condition \eqref{general_defoc} is equivalent to $\iota_0>0$, $\iota_1>0$. In this case, Theorem \ref{T:defoc_defoc} is new. Theorem \ref{T:defoc_defoc'} is proved in \cite[Theorem 1.3]{TaViZh07} and \cite{Zhang06}).%We do not know if Theorem \ref{T:defoc_defoc} holds in this case with the weaker assumption $u_0\in H^{s_0}$.

For a double-power nonlinearity with ($\iota_0>0$, $\iota_1<0$) or ($\iota_0<1$, $0<s_0<s_1<1$), there are solitary wave solutions and thus it is impossible to prove an analog of Theorems \ref{T:defoc_defoc} and \ref{T:defoc_defoc'}. This is an open question for other double power nonlinearities.

%\thomas{Thomas: I re-formulated your small mass result, assuming $\iota_0>0$. In the case where $\iota_0<0$, we should be able to prove by rescaling that $\Phi(u)$ is not positive for small masses, so that it seems to be an equivalent formulation.}

When only the main order term of the nonlinearity is defocusing, i.e. when $\iota_0>0$, scattering holds for initial data with small mass. We again distinguish between $s_0>1$ and $s_0\in (0,1]$
\begin{theorem}
\label{T:scatt_intro2}
Let $d\geq 3$ and assume $s_0>1$. Let $g$ be a
nonlinearity that satisfies Assumption \ref{Assum:profile} with $\iota_0>0$. Assume that Property \ref{Proper:bnd} holds for some $A_0\in (0,\infty]$. There exists $m_c>0$ such that any solution of \eqref{NLS} with initial data in $H^{s_0}$ such that $M(u_0)<m_c$ and that satisfies \eqref{bound_Hs0} is global and scatters forward in time.
\end{theorem}
\begin{theorem}
\label{T:scatt_intro2'}
Let $g$ be
nonlinearity that satisfies Assumption \ref{Assum:profile} with $\iota_0>0$ and $s_0\in (0,1]$. There exists $m_c>0$ such that any solution of \eqref{NLS} with initial data in $H^{1}$ and such that $M(u_0)<m_c$ is global and scatters in both time directions.
\end{theorem}
\begin{remark}
In Theorem \ref{T:defoc_defoc} and Theorem \ref{T:scatt_intro2}, if the analog of \eqref{bound_Hs0} as $t$ goes to $T_{-}(u)$ is valid then the conclusions hold backward in time. 
\end{remark}

Theorem \ref{T:scatt_intro2'} is new, even for double power nonlinearities. Theorem \ref{T:scatt_intro2'} generalizes \cite[Theorem 1.3]{TaViZh07} which concerns a double power nonlinearity with $\iota_0>0$ and $\iota_1<0$. In particular the case where $d=3$, $p_0=4$ (thus $s_0=1$), $p_1=2$ scattering was proved to hold for a larger set of initial data in \cite{KiOhPoVi17}, \cite{KillipMurphyVisan21}. In a subsequent work, we will use the material of this article, together with \cite{LeNo20} to obtain an improvement of Theorems \ref{T:scatt_intro2} and \ref{T:scatt_intro2'} in the same spirit.

The proofs of Theorem \ref{T:defoc_defoc} and Theorem \ref{T:scatt_intro2} follow the by now classical rigidity-compactness roadmap (see \cite{KeMe06}), using the profile decomposition constructed in Section \ref{S:profile}. This provides, in a contradiction argument, a global critical solution $u_c$ of \eqref{NLS} such that there exists $x(t)$ such that
$$K=\left\{u_c(t,\cdot+x(t)),\; t\in \R\right\}$$
has compact closure in $H^{s_0}$.

To exclude this critical element and obtain a contradiction, we use the virial identity
\begin{equation}
 \label{virial_ID}
 \frac{d}{dt} \Im \int x\cdot\nabla u\,\overline{u}=2\Phi(u),
\end{equation} 
where
  $$\Phi(u)=\int |\nabla u|^2+\iota_1\frac{dp_1}{2(p_1+2)}|u|^{p_1+2}+\iota_0\frac{dp_0}{2(p_0+2)}|u|^{p_0+2},$$
and the center of mass identity: 
\begin{equation}
\label{momentum}
 \frac{d}{dt}\int x|u(t)|^2=2 P(u).
\end{equation}
which are valid for solutions of $u$ with enough decay at infinity.

Using a localized version of \eqref{momentum}, one can prove:
\begin{equation}
\label{control_x(t)}
\lim_{t\rightarrow +\infty}\frac{1}{t}\left|x(t)-X(t)\right|=0\text{ where }X(t)=2\frac{P(u_c)}{M(u_c)}t.
\end{equation}
When the momentum of $u_c$ is zero, this allows to control the growth of $x(t)$. A localized version of \eqref{virial_ID} using the relative compactness $K$, gives a contradiction if $\inf_{t\in \R} \Phi(u_c(t))-\frac{|P(u_c)|^2}{M(u_c)}>0$. In the defocusing/defocusing case (as in Theorem \ref{T:defoc_defoc}), this property is true whenever $u_c$ is not identically zero. In the defocusing/focusing case, we show, using the quite complete study of the elliptic problem in \cite{LeNo20}, and some of the ideas in \cite{KiOhPoVi17}, that $\Phi$ is positive in the region described in the assumptions of Theorem \ref{T:scatt_intro2}, yielding again the desired contradiction.

One must adapt this argument when the momentum of $u_c$ is not zero. The standard strategy, going back to \cite{DuHoRo08} is to use the Galilean transformation to reduce to the case to a critical solution with zero momentum. However the effect of the Galilean transform on the Sobolev norm $\dot{H}^{s_0}$ of the solution is not explicit, and thus the strategy breaks down in the case $s_0>1$, where our proof relies on an induction-type argument on this norm. To tackle this difficulty, we observe that \eqref{virial_ID}, \eqref{momentum} and the conservation of momentum imply:
$$\frac{d}{dt}\Im \int (x-X(t))\nabla u\, \overline{u}= 2\Phi(u(t)) - \frac{2|P(u)|^2}{M(u)},$$
which we localize with a time-dependant localization around $X(t)$. This gives again a contradiction using an improved Cauchy-Schwarz inequality going back to \cite{Banica04} to prove that $\Phi(u)  - \frac{|P(u)|^2}{M(u)}$ is still positive for the solutions that we consider.

\subsection{Previous works}
\label{sub:previous}
To our knowledge, the profile decomposition for a general nonlinearity of the form \eqref{def_g_intro} was not considered before.

Dynamics of nonlinear Schr\"odinger equations with general nonlinearities were treated in very few works. In \cite{SofferLiu23}, \cite{SofferWu23}, the authors prove that any radial, bounded, global solutions of a nonlinear Schr\"odinger equation with a general energy-subcritical or energy-critical nonlinearity is asymptotically the sum of a radiation term, solution of the linear Schr\"odinger equation, and a localized term. We refer to \cite{SoWu23} for results in a nonradial setting.

Let us mention a few works on NLS equation with a double-power nonlinearity. The study of this type of equation was initiated in \cite{Zhang06}, in dimension $3$, where the author investigated the global well-posedness, scattering and blow-up phenomena in the case $p_0=4$. This includes in particular a scattering result for small mass, in the spirit of Theorem \ref{T:scatt_intro2}. Similar results were obtained in \cite{TaViZh07}, in general dimension $d$ in the energy-critical and subcritical setting $s_0\leq 1$. %As already mentioned, Theorem \ref{T:scatt_intro2} was proved in \cite{KiOhPoVi17} in the particular case $d=3$, $p_0=4$, $p_1=2$. The scattering was extended to a slightly larger region in \cite{KillipVisan21}.

The problem with $p_1=\frac{4}{d}$, $\iota_1=-1$, $p_0<\frac{4}{d-2}$ was considered in \cite{Cheng20}, \cite{Murphy21}, where the author investigated scattering below or at the mass of the ground-state for the mass-critical homogeneous equation. See also \cite{CaCh21} which considers the 
case $(p_0,p_1,\iota_0,\iota_1)=(4,2,1,-1)$ in space dimensions $1$, $2$ and $3$.

The problem with a focusing dominant nonlinearity $\iota_0=-1$ was considered in many works, always in the energy-critical or subcritical cases. Let us mention in particular
\cite{AkIbKiNa} where a nine-set theorem in the spirit of \cite{NaSch11}, \cite{NaSc12} was proved. 
We also refer to \cite{MiXuZh13}, \cite{Xie14},\cite{HuangZhang14}, \cite{XuYang16}, \cite{ChMiZh16}, \cite{MiaoZhaoZheng17}, \cite{Luo} for scattering result or scattering blow-up dichotomy in the case $\iota_0=-1$.

Let us mention that in several of the preceding works, the authors construct and use a profile decomposition adapted to NLS equation with a particular double-power nonlinearity. The profile decomposition in Section \ref{S:profile} generalises these profiles decompositions to the large class of non-linearities described above.

\subsection{Outline}
We conclude this introduction by an outline of the article. Section \ref{S:preliminary} is devoted to some preliminaries, well-posedness and perturbation theory for the NLS equation \eqref{NLS} with a general nonlinearity $g$. In Section \ref{S:profile}, we construct a profile decomposition for sequences that are bounded in $H^{s_0}$ adapted to \eqref{NLS} with a nonlinearity of the form \eqref{def_g_intro}. This profile decomposition is based on the linear profile decomposition of Shao \cite{Shao09} (which relies ultimatly on the result by  Merle-Vega \cite{MeVe98})  and the long-time perturbation result proved in Subsection \ref{sub:Cauchy}. In Section \ref{S:GWP}, we use this profile decomposition to prove our global well-posedness result, Theorem \ref{T:GWP}. In Section \ref{S:rigidity} we prove a general rigidity result for solutions of \eqref{NLS} with a relatively compact trajectory in $H^{s_0}$. In Section \ref{S:scattering}, we prove our scattering result Theorems \ref{T:defoc_defoc} and \ref{T:scatt_intro2}, using the material of the preceding sections.

\section*{Acknowledgment}

The second author is supported by Post-doc fellowship from Labex MME-DII: SAIC/2022 No 10078.

\section{Preliminaries and Cauchy theory}
\label{S:preliminary}
This section is concerned with the Cauchy and stability theory for the equation:
\begin{equation}
 \label{NLS_g}
 i\partial_tu+\Delta u=g(u).
\end{equation} 
for a general $L^2$ supercritical nonlinearity $g(u)$ in space dimension $d\geq 1$. Our assumptions on $g$ will of course include the case of double power nonlinearities that we are interested in. We start by introducing some notations and functions spaces (see \S \ref{sub:notations}) and recalling some nonlinear estimates (see \S \ref{sub:nonlinear}).  In \S  \ref{sub:nonlinearity} we prove estimates on the nonlinearity $g(u)$ that are crucial for the well-posedness theory. In \S \ref{sub:Cauchy} we prove our main results.

\subsection{Notations and function spaces}
\label{sub:notations}

For $s\in\R$, $\supent{s}$ is the smallest integer number larger or equal $s$, and $\ent{s}$ is the largest integer smaller or equal $s$ (the integer part of $s$).
If $j$ and $k$ are integers with $j\leq k$, we denote by $\llbracket j,k \rrbracket=\{j,j+1,\ldots,k-1,k\}$.

When $A$ and $B$ are two positive quantities depending on some parameters, we denote $A\lesssim B$ when there is a constant $C>0$ such that $A\leq CB$ and $A\approx B$ when $A\lesssim B$ and $B\lesssim A$.

For each $q\geq 1$, we define $q'$ such that 
\[
\frac{1}{q}+\frac{1}{q'}=1.
\]
If $X$ is a vector space, $(u,v)\in X^2$, we will make a small abuse of notation, denoting $\|(u,v)\|_{X}=\|u\|_X+\|v\|_X$.

We fix $d\geq 1$. If $m$ is a complex valued function on $\R^d$, we define by $m(\nabla)$ the Fourier multiplier with symbol $m(\xi)$, i.e. $\widehat{m(\nabla )u}=m(\xi)\hat{u}(\xi)$, where $\hat{u}$ is the Fourier transform of $u$.
For $s\geq 0$, $p\geq 1$, we define
\[
\norm{u}_{H^{s,p}}=\norm{\langle \nabla  \rangle^s u}_{L^p},\quad \norm{u}_{\dot{H}^{s,p}}=\norm{|\nabla|^s u}_{L^p}
\]
where $\langle \xi\rangle=(1+|x|^2)^{s/2}$. 
It follows from Mikhlin multiplier theorem that
\begin{equation}
 \label{equivalence}
 \norm{u}_{H^{s,p}}\approx \|u\|_{L^p}+\left\| |\nabla |^su\right\|_{L^p}, \quad 1<p<\infty.
\end{equation} 
(see the proof in the appendix). For a multi-index $\alpha=(\alpha_1,\alpha_2,...,\alpha_d)$, denote
\[
D^{\alpha}=\partial_{x_1}^{\alpha_1}\cdot\cdot\cdot \partial_{x_d}^{\alpha_d}, \quad |\alpha|=\sum_{i=1}^d|\alpha_i|.
\]
For $s>0$ and $1<p<\infty$ and $v=s-\ent{s}$, we have (see \cite[Lemma 3.2]{AnKi21}),
\begin{equation}
\label{equivalent norm}
\sum_{|\alpha|=\ent{s}}\norm{D^{\alpha}f}_{\dot{H}^{v,p}} \approx \|f\|_{\dot{H}^{s,p}}.
\end{equation}

We recall that a pair $(q,r)$ is \emph{Strichartz-admissible} for the Schr\"odinger equation when $2\leq q\leq \infty$, $2\leq r\leq \infty$, $\frac{2}{q}+\frac{d}{r}=\frac{d}{2}$, and $(q,r,d)\neq (2,\infty,2)$. We recall that if $(q,r)$ and $(a,b)$ are Strichartz admissible, we have, 
$$ \|u\|_{L^{q}_tL^r_{x}}\lesssim \|u_0\|_{L^2}+\|f\|_{L^{a'}_tL^{b'}_x}$$
for any solution of the Schr\"odinger equation
$i\partial_tu +\Delta u=f$ with initial data $u_0$.
We denote:
$$S^0(I)=\begin{cases}
L^{\infty}\left(I,L^2(\R^d)\right)\cap L^{2}\left(I,L^{\frac{2d}{d-2}}(\R^d)\right)& \text{ if }d\geq 3\\
L^{\infty}\left(I,L^2(\R^2)\right)\cap L^{q_2}\left(I,L^{r_2}(\R^2)\right)& \text{ if }d=2 \\
L^{\infty}\left(I,L^2(\R)\right)\cap L^{4}\left(I,L^{\infty}(\R)\right)& \text{ if }d=1,
         \end{cases}
$$
where when $d=2$, $(q_2,r_2)$ is an admissible pair with $q_2>2$ close to $2$, and 
\begin{equation*}
N^0(I)= \begin{cases}
L^{1}\left(I,L^2(\R^d)\right)+ L^{2}\left(I,L^{\frac{2d}{d+2}}(\R^d)\right)& \text{ if }d\geq 3\\
L^{1}\left(I,L^2(\R^2)\right) + L^{q_2'}\left(I,L^{r_2'}(\R^2)\right)& \text{ if }d=2 \\
L^{1}\left(I,L^2(\R)\right)+ L^{4/3}\left(I,L^{1}(\R)\right)& \text{ if }d=1,
         \end{cases}
\end{equation*}
Note that the norm of $S^0(I)$ is equivalent to the supremum of all $L^q(I,L^r)$ norms, and that the norm of $N^0(I)$ is smaller than the infimum of all $L^{q'}(I,L^{r'})$ norms, where in each case, $(q,r)$ is taken over all Strichartz admissible pairs (with $q\geq q_2$ if $d=2$). We also define the following Strichartz spaces, and dual Strichartz spaces:
\begin{equation*}
W^0(I)=
L^{\frac{2(2+d)}{d}}\left(I\times\R^d\right)
 \quad Z^0(I)=\left(W^0(I)\right)'=L^{\frac{2(d+2)}{d+4}}(I\times \R^d),
\end{equation*}
so that we have $S^0(I)\subset W^0(I)$ and $Z^0(I)\subset N^0(I)$ (with continous embedding). We denote
$$ X_p(I)=L^{\frac{p(d+2)}{2}}(I\times \R^d). $$

For $s\geq 0$, we denote
\begin{gather*}
\norm{u}_{S^s(I)}=\norm{\langle\nabla\rangle^{s} u}_{S^0(I)},\quad \norm{u}_{\dot{S}^s(I)}=\norm{|\nabla |^{s} u}_{S^0(I)}
\end{gather*}
and define similarly $W^s(I)$, $\dot{W}^s(I)$,  $N^s(I)$, $\dot{N}^s(I)$, $Z^s(I)$ and $\dot{Z}^s(I)$.

\subsection{Preliminary nonlinear estimates}
\label{sub:nonlinear}
\begin{lemma}(Product Rule 1)\label{L:product_rule}
Let $s \geq 0$ and $1<r,r_1,r_0,q_1,q_0<\infty$ such that $\frac{1}{r}=\frac{1}{r_i}+\frac{1}{q_i}$, for $i=0,1$. Then,
\begin{gather}
\label{product_rule1}
\norm{|\nabla|^{s} (f\varphi)}_{L^r_x}\lesssim \norm{f}_{L^{r_1}_x}\norm{|\nabla|^{s} \varphi}_{L^{q_1}_x}+\norm{|\nabla|^{s}f}_{L^{r_0}_x}\norm{\varphi}_{L^{q_0}_x}\\
\label{product_rule2}
\norm{\left<\nabla\right>^{s} (f\varphi)}_{L^r_x}\lesssim \norm{f}_{L^{r_1}_x}\norm{\left<\nabla\right>^{s} \varphi}_{L^{q_1}_x}+\norm{\left<\nabla\right>^{s}f}_{L^{r_0}_x}\norm{\varphi}_{L^{q_0}_x}.
\end{gather}
\end{lemma}
\begin{proof}
See e.g \cite[Proposition 3.3]{ChWe91}. For \eqref{product_rule1}, see e.g \cite[Lemma 2.2]{AnKi21} for the statement and \cite{ChWe91} for proof in dimension $1$ and $s\in (0,1)$.

By \eqref{equivalence}, \eqref{product_rule1} and H\"older inequality, we have
\begin{multline*}
 \left\|\left<\nabla\right>^s (f\varphi)\right\|_{L^r}\approx\norm{f\varphi}_{L^r}+\norm{|\nabla|^s(f\varphi)}_{L^r}
\lesssim \norm{f}_{L^{r_1}}\norm{\varphi}_{L^{q_1}}+\norm{f}_{L^{r_1}}\norm{|\nabla|^s\varphi}_{L^{q_1}}+\norm{|\nabla|^sf}_{L^{r_0}}\norm{\varphi}_{L^{q_0}}\\
\lesssim\norm{f}_{r_1}\norm{\left<\nabla\right>^s \varphi}_{q_1}+\norm{\left<\nabla\right>^sf}_{r_0}\norm{\varphi}_{q_0}.
\end{multline*}
Hence \eqref{product_rule2}.
\end{proof}

\begin{lemma}[Product rule 2](see e.g \cite[Corollary 2.3]{AnKi21}).
\label{lm product rule 2}
Let $s\geq 0$, $n \in \N$ ($n\geq 1$), and, for $i,j\in \llbracket 1,n\rrbracket$, $1<r,r^i_k<\infty$, such that for all $k\in \llbracket 1,n\rrbracket$, $\frac{1}{r}=\sum_{i=1}^n\frac{1}{r^i_k}$. Then
\[
\norm{\prod_{i=1}^n f_i}_{\dot{H}^{s,r}}\lesssim \sum_{k=1}^n\left(\norm{f_k}_{\dot{H}^{s,r^k_k}}\prod_{i\neq k}\norm{f_i}_{L^{r^i_k}}\right).
\]
\end{lemma}

\begin{lemma}[Fractional chain rule](see e.g \cite[Lemma 2.4]{AnKi21}).
\label{lm fractional chain rule}
Let $G\in C^1(\C,\C)$, $s\in (0,1)$, $1<r,r_2<\infty$, and $1<r_1\leq\infty$ satisfying $\frac{1}{r}=\frac{1}{r_1}+\frac{1}{r_2}$,
\[
\norm{G(u)}_{\dot{H}^{s,r}}\lesssim \norm{G'(u)}_{L^{r_1}}\norm{u}_{\dot{H}^{s,r_2}}.
\]
\end{lemma}

\begin{lemma}[Gagliardo-Nirenberg inequality \cite{BrMi18}] 
\label{lm GN inequality}
Let $s_1\leq s_2$, $p_2> 1$, $s=\theta s_1+(1-\theta)s_2$, $\frac{1}{p}=\frac{\theta}{p_1}+\frac{1-\theta}{p_2}$. Then
\[
\norm{u}_{H^{s,p}}\lesssim \norm{u}_{H^{s_1,p_1}}^{\theta}\norm{u}_{H^{s_2,p_2}}^{1-\theta}.
\]
\end{lemma}

\begin{lemma}[Homogeneous Gagliardo-Nirenberg inequality]
\label{lm homogeneous GN inequality}
Let $s_1\leq s_2$, $p_2> 1$, $s=\theta s_1+(1-\theta)s_2$, $\frac{1}{p}=\frac{\theta}{p_1}+\frac{1-\theta}{p_2}$. Then
\[
\norm{u}_{\dot{H}^{s,p}}\lesssim \norm{u}_{\dot{H}^{s_1,p_1}}^{\theta}\norm{u}_{\dot{H}^{s_2,p_2}}^{1-\theta}.
\]
\end{lemma}

\begin{proof}
Let $\varphi\in H^s$ Define $\varphi^{\lambda}(x)=\varphi(\lambda x)$. Applying Lemma \ref{lm GN inequality} for $\varphi^{\lambda}$, we have
\begin{multline*}
\left(\lambda^{-d}\norm{\varphi}_{L^p}^p+\lambda^{ps-d}\norm{\varphi}_{\dot{H}^{s,p}}^p\right)^{\frac{1}{p}}\\
\lesssim \left(\lambda^{-d}\norm{\varphi}_{L^{p_1}}^{p_1}+\lambda^{p_1s_1-d}\norm{\varphi}_{\dot{H}^{s_1,p_1}}^{p_1}\right)^{\frac{1}{{p_1}}}\left(\lambda^{-d}\norm{\varphi}_{L^{p_2}}^{p_2}+\lambda^{p_2s_2-d}\norm{\varphi}_{\dot{H}^{s_2,p_2}}^{p_2}\right)^{\frac{1}{p_2}}.
\end{multline*}
By dividing both sides by $\lambda^{s-\frac{d}{p}}=\lambda^{\theta s_1-\frac{d\theta}{p_1}}\lambda^{(1-\theta)s_2-\frac{d(1-\theta)}{p_2}}$, we have
\[
\left(\lambda^{-ps}\norm{\varphi}_{L^p}^p+\norm{\varphi}_{\dot{H}^{s,p}}^p\right)^{\frac{1}{p}}\lesssim \left(\lambda^{-p_1s_1}\norm{\varphi}_{L^{p_1}}^{p_1}+\norm{\varphi}_{\dot{H}^{s_1,p_1}}^{p_1}\right)^{\frac{1}{{p_1}}}\left(\lambda^{-p_2s_2}\norm{\varphi}_{L^{p_2}}^{p_2}+\norm{\varphi}_{\dot{H}^{s_2,p_2}}^{p_2}\right)^{\frac{1}{p_2}}.
\]
Let $\lambda\rightarrow +\infty$, we obtain the desired result.
\end{proof}

\begin{lemma}[Leibniz rule]
\label{L:Leibniz}
Let $f \in C^k(\C,\C)$ and $\alpha=(\alpha_1,\alpha_2,\cdot\cdot\cdot,\alpha_d)\in \N^d$ such that $|\alpha|\leq k$. Then $D^{\alpha}\left(f(u)\right)$ is a linear combinations of terms of the form 
$$ (\partial_z^{h_1}\partial_{\overline{z}}^{h_2}f)(u)\prod_{k=1}^{h_1} D^{\beta_k} u \prod_{k=1}^{h_2} D^{\gamma_k} \overline{u} ,$$
where $\sum_{k=1}^{h_1}\beta_k+\sum_{k=1}^{h_2}\gamma_k=\alpha$, $1\leq h_1+h_2\leq |\alpha|$, $|\beta_k|\geq 1$, $|\gamma_k|\geq 1$ for all $k$.
\end{lemma}
\begin{proof}
The proof is by induction, using the formula:
\[
\partial_{x_i}f(u)=\partial_zf(u)\partial_{x_i} u+\partial_{\overline{z}}f(u)\overline{\partial_{x_i} u}.
\]
\end{proof}

\subsection{Local Lipschitz continuity of the nonlinearity}
\label{sub:nonlinearity}

In this subsection we consider the following general classes of nonlinearities:
\begin{definition}
\label{D:NNN}
 Let $0<p_1\leq p_0$, $s\geq 0$ be real numbers. We denote by $\NNN(s,p_0,p_1)$ the vector space of functions $g\in C^{\supent{s}+1}(\C,\C)$ such that
 \begin{gather}
  \label{Hyp_g} 
  \exists C>0,\; \forall k\in \llbracket 0,\supent{s}+1\rrbracket,\; \forall z\in \C,\quad \left|g^{(k)}(z)\right|\leq C \left(|z|^{p_0+1-k}+|z|^{p_1+1-k}\right).
 \end{gather}
\end{definition}
In the definition $|g^{(k)}(z)|$ denotes the supremum of all the derivatives (in $z$, $\overline{z}$) of order $k$ of $g$.

\begin{assumption}
\label{Assum:NL}
$d\geq 1$, $g\in \NNN(s,p_0,p_1)$ with $0\leq s$, $\frac{4}{d}\leq p_1\leq p_0$, $1<p_1$ and
$$\supent{s}\leq p_0 \text{ or }g\text{ is a polynomial in }u,\overline{u}.$$
\end{assumption}
We recall the notation $X_p(I)=L^{\frac{p(d+2)}{2}}(I\times \R^d)$. We will prove:

\begin{proposition}
\label{P:boundg}
If Assumption \ref{Assum:NL} is satisfied, we let
\begin{equation}
\label{defX}
X(I)=X_{p_0}(I)\cap X_{p_1}(I)=L^{p_0(d+2)/2}_{t,x}(I\times\R^d)\cap L^{p_1(d+2)/2}_{t,x}(I\times\R^d).
\end{equation}
Then
\begin{multline}
  \label{diff_bound}
  \left\|g(u)-g(v)\right\|_{N^s(I)}\lesssim 
  \left(\|(u,v)\|_{X(I)}^{p_1-1}+\|(u,v)\|_{X(I)}^{p_0-1}\right)
  \\
  \times \Big[ \|u-v\|_{\dot{W}^s(I)}\|(u,v)\|_{X(I)}+\|u-v\|_{X(I)}\|(u,v)\|_{W^s(I)}\Big],
 \end{multline} 
 In particular
 \begin{equation}
  \label{bound_g}
  \left\|g(u)\right\|_{N^s(I)}\lesssim \|u\|_{W^s(I)}\left(\|u\|_{X(I)}^{p_0}+\|u\|_{X(I)}^{p_1}\right).
 \end{equation} 
 \end{proposition}
Let us insist on the important fact that the first norm of $u-v$ in the second line of \eqref{diff_bound} is the norm in the \emph{homogeneous} space $\dot{W}^s$.

We start with a few lemmas.

\begin{lemma}
\label{L:boundg_onepower}
Let $p\geq 1$ be an integer, and $g(u)$ be a homogeneous polynomial of degree $p+1$ in $u$, $\overline{u}$. 
 Let $s\geq 0$, $u,v\in S^s(I)$. Then 
 \begin{multline}
 \label{diff_boundj'}
 \left\||\nabla|^s(g(u)-g(v))\right\|_{N^0(I)}\lesssim 
  \left(\|u\|_{X_p(I)}^{p-1}+\|v\|_{X_p(I)}^{p-1}\right)\\
  \times \Big[ \||\nabla|^s(u-v)\|_{W^0(I)}\left(\|u\|_{X_p(I)}+\|v\|_{X_p(I)}\right)+\|u-v\|_{X_p(I)}\left(\||\nabla|^su\|_{W^0(I)}+\||\nabla|^sv\|_{W^0(I)}\right)\Big],
 \end{multline}
 and
\begin{equation}
\label{diff_boundj''}
\left\|g(u)-g(v)\right\|_{N^0(I)}\lesssim \|u-v\|_{X_p(I)}\|(u,v)\|_{X_p(I)}^{p-1}\|(u,v)\|_{W^0(I)}.
\end{equation} 
 \end{lemma}
\begin{proof}
It is sufficient to prove \eqref{diff_boundj'} and \eqref{diff_boundj''} for $g(u)=u^{p+1-j}\overline{u}^j$, where $j\in \llbracket 0,p+1\rrbracket$.
 We have
 \begin{equation}
  \label{ID_g}
 g(u)-g(v)=\left(u^{p+1-j}-v^{p+1-j}\right)\overline{u}^j+v^{p+1-j} (\overline{u}^j-\overline{v}^j).
 \end{equation}
We treat the contribution of the first term in the right-hand side of \eqref{ID_g}. The contribution of the second term is similar. We have
\begin{equation}
\label{Id_up}
 \left(u^{p+1-j}-v^{p+1-j}\right)\overline{u}^j=(u-v)\overline{u}^j\sum_{k=0}^{p-j} u^{p-j-k}v^k.
\end{equation}
We work on the interval $I$ for all norms.
We note the following relations between the exponents defining the $W^0$, $Z^0=(W^0)'$ and $X_{p}$ norms:
\begin{equation}
 \label{relation}
\frac{d+4}{2(d+2)}=\frac{d}{2(d+2)}+p\times \frac{2 }{p(d+2)}.
\end{equation}
 By H\"older inequality, the definitions of $X_p$ and $Z_0$ and the fact that $Z_0$ is continuously embedded into $\dot{N}^s$, we obtain \eqref{diff_boundj''}.

 We next prove \eqref{diff_boundj'}. Using \eqref{Id_up}, Lemma \ref{L:product_rule} and the definitions of $X_p$ and $Z^0$, we have
\begin{multline*}
\norm{
\left(u^{p+1-j}-v^{p+1-j}\right)\overline{u}^j
}_{\dot{N}^s}
\lesssim \norm{|\nabla|^s \Big(\left(u^{p+1-j}-v^{p+1-j}\right)\overline{u}^j\Big)
}_{Z^0}
\\
\lesssim \sum_{k=0}^{p-j}\norm{|\nabla|^s (u-v)}_{W^0}\norm{u}_{{X_p}}^{p-k}\norm{v}_{{X_p}}^{k}\\
\quad+\sum_{k=0}^{p-j}\norm{u-v}_{{X_p}}\left(\norm{|\nabla|^s u}_{W^0}\norm{u}_{X_p}^{p-k-1}\norm{v}_{X_p}^k+\norm{|\nabla|^sv}_{W^0}\norm{v}_{{X_p}}^{k-1}\norm{u}_{X_p}^{p-k}\right)\\
 \lesssim \norm{u-v}_{W^s}(\norm{u}_{X_p}^{p}+\norm{v}_{X_p}^{p})+\norm{u-v}_{X_p}(\norm{u}_{W^s}+\norm{v}_{W^s})(\norm{u}_{X_p}^{p-1}+\norm{v}_{X_p}^{p-1}).
\end{multline*}
Combining with the same bound for the second term in \eqref{ID_g}, we obtain \eqref{diff_boundj'}, concluding the proof.
\end{proof}

\begin{lemma}
\label{L:NL}
 Let $s$, $p$ be real numbers such that $0<s$, $\supent{s}\leq p$, $1<p$ and $\frac{4}{d}\leq p$. Let $g\in C^{\supent{s}+1}(\C,\C)$ such that 
 \begin{equation}
  \label{hyp_g_lemma}
  \forall k\in \llbracket 0,\supent{s}+1\rrbracket,\quad \left| g^{(k)}(u)\right|\leq C |u|^{p+1-k}.
 \end{equation} 
 Then \eqref{diff_boundj'} and \eqref{diff_boundj''} hold.
\end{lemma}
\begin{proof}
We have 
\[
g(u)-g(v)=(u-v)\int_0^1 g_z (v+\theta (u-v))\,d\theta+\overline{u-v}\int_0^1 g_{\overline{z}}(v+\theta(u-v))\,d\theta,
\]
where $g_z=\frac{\partial g}{\partial z}$, $g_{\overline{z}}=\frac{\partial g}{\partial \overline{z}}$.
Since $\left|\frac{\partial g}{\partial z}\right|,\left|\frac{\partial g}{\partial \overline{z}}\right|\leq C |z|^{p}$, we have
\[
|g(u)-g(v)|\leq C|u-v|(|u|^{p}+|v|^{p}).
\]
Thus, by H\"older inequality and \eqref{relation},
\begin{equation}
\label{T20}
\norm{g(u)-g(v)}_{N^0}\lesssim \norm{g(u)-g(v)}_{Z^0}\lesssim \|u-v\|_{X_p}\|(u,v)\|^{p-1}_{X_p}\|(u,v)\|_{W^0}.
 \end{equation} 

This yields \eqref{diff_boundj''}. We are left with proving \eqref{diff_boundj'} when $s>0$ . We define $a$ by
\[
\frac{1}{a}= \frac{d+4}{2(d+2)}-\frac{2}{p(d+2)}.
\]
Note that the assumption $p\geq \frac 4d$ implies $\frac{2(d+2)}{d+4}\leq a\leq \frac{d+2}{2}$. 
By using product rule Lemma \ref{L:product_rule}, we have
\begin{align}
\label{zero term}
&\norm{g(u)-g(v)}_{\dot{N}^s}\leq \norm{|\nabla|^s(g(u)-g(v))}_{Z^0}  \\
&\lesssim \norm{|\nabla|^s (u-v)}_{W^0}\left(\norm{\int_0^1 g_z(v+\theta (u-v))\,d\theta}_{L^{\frac{d+2}{2}}_{t,x}}+\norm{\int_0^1 g_{\overline{z}}(v+\theta (u-v))\,d\theta}_{L^{\frac{d+2}{2}}_{t,x}}\right)\label{first term}\\
&\quad +\norm{u-v}_{X_p}\left(\norm{\int_0^1  |\nabla|^s g_z(v+\theta (u-v))\,d\theta}_{L^{a}_{t,x}}+\norm{\int_0^1  |\nabla|^s g_{\overline{z}}(v+\theta (u-v))\,d\theta}_{L^{a}_{t,x}}\right)\label{second term}.
\end{align}
By the assumption \eqref{hyp_g_lemma}, the term \eqref{first term} is bounded as follows
\begin{equation}
\label{T32}
\eqref{first term}\lesssim
\norm{u-v}_{\dot{W}^s} (\norm{u}^{p}_{L^{p(d+2)/2}_{t,x}}+\norm{v}^{p}_{L^{p(d+2)/2}_{t,x}})
\approx \norm{u-v}_{\dot{W}^s} \norm{(u,v)}_{X_p}^{p}.
\end{equation}
We now consider the term \eqref{second term}. We will prove the following bound
\begin{equation}
\label{T33}
\eqref{second term}\lesssim \|u-v\|_{X_p}\|(u,v)\|_{\dot{W}^s}\|(u,v)\|^{p-1}_{X_p}
\end{equation}
Combining \eqref{T32} and \eqref{T33}, we obtain the bound \eqref{diff_boundj'}. We are thus left with proving \eqref{T33}.

Note that the term \eqref{second term} is bounded by
\begin{align*}
\norm{u-v}_{X_p}(\sup_{\theta\in [0,1]}\norm{|\nabla|^s g_z(v+\theta (u-v))}_{L^{a}_{t,x}}+\sup_{\theta\in [0,1]}\norm{|\nabla|^s g_{\overline{z}}(v+\theta (u-v))}_{L^{a}_{t,x}}).
\end{align*}
Letting $f=g_z$ or $f=g_{\overline{z}}$, we will prove, for a general function $u\in \dot{W}^s\cap X_p$,
\begin{equation}
 \label{star}
 \left\| |\nabla|^s (f(u))\right\|_{L^a_{t,x}} \lesssim \|u\|_{\dot{W}^s}\|u\|_{X_p}^{p-1}.
\end{equation} 
Note that this will conclude the proof of \eqref{T33}. 

The function $f$ belongs to $C^{\supent{s}}$ and satisfies 
\begin{equation}
 \label{boundf}
|f^{(k)}(z)|\lesssim |z|^{p-k}, 
 \end{equation} 

for each $0\leq k\leq\supent{s}$. We use similar argument as in \cite[Proof of Lemma 3.3]{AnKi21}. Using the equivalence of norms \eqref{equivalent norm}, we have
\begin{equation}\label{eq345}
\norm{|\nabla|^s f(u)}_{L^{a}_{t,x}} \lesssim \sum_{|\alpha|=\ent{s}}\norm{D^{\alpha}f(u)}_{L^{a}\dot{H}^{v,a}},
\end{equation}
where $v=s-\ent{s}$. By Lemma \ref{L:Leibniz}, $D^{\alpha}\left(f(u)\right)$ is a linear combinations of terms of the form 
$$ (\partial_z^{h_1}\partial_{\overline{z}}^{h_2}f)(u)\prod_{k=1}^{h_1} D^{\beta_k} u \prod_{k=1}^{h_2} D^{\gamma_k} \overline{u} ,$$
where $\sum_{k=1}^{h_1}\beta_k+\sum_{k=1}^{h_2}\gamma_k=\alpha$, $1\leq h_1+h_2\leq |\alpha|$, $|\beta_k|\geq 1$, $|\gamma_k|\geq 1$ for all $k$. To simplify notations we will only consider terms of the form
$$f^{(h)}(u)\prod_{k=1}^{h}D^{\beta_k}u,\qquad \sum_{k=1}^h \beta_k=\alpha,\; 1\leq h\leq |\alpha|,\; |\beta_k|\geq 1,$$
where $f^{(h)}=\left(\frac{\partial }{\partial z}\right)^{h}f$. The proof is the same for the other terms. We distinguish between the cases $h<p$ and $h=p$. 

\medskip

\emph{Case $h<p$.}

Using Lemma \ref{lm product rule 2} and \eqref{boundf}, we have
\begin{align}
&\norm{f^{(h)}(u)\prod_{i=1}^h D^{\beta_i}u}_{L^{a}\dot{H}^{v,a}}\label{number0}\\
&\lesssim\norm{f^{(h)}(u)}_{L^{q_1}\dot{H}^{v,q_1}}\prod_{i=1}^h\norm{D^{\beta_i}u}_{L^{r_i}_{t,x}}\label{number1}\\
&\quad +\norm{|u|^{p-h}}_{L^{q_2}_{t,x}}\sum_{k=1}^h \prod_{i=1,i\neq k}^h\norm{D^{\beta_i}u}_{L^{r_i}_{t,x}}\norm{D^{\beta_k}u}_{L^{\rho_k}\dot{H}^{v,\rho_k}}.\label{number2} 
\end{align}
where, for $i\in \llbracket 1,h\rrbracket$,
\begin{equation}
\label{def_param1}
\frac{1}{r_i}=\frac{|\beta_i|}{s}\frac{d}{2(d+2)}+\left(1-\frac{|\beta_i|}{s}\right)\frac{2}{p(d+2)}, \; \frac{1}{\rho_i}=\frac{|\beta_i|+v}{s}\frac{d}{2(d+2)}+\left(1-\frac{|\beta_i|+v}{s}\right)\frac{2}{p(d+2)}
\end{equation}
and
\begin{equation}
\label{def_param2}
\frac{1}{q_1}=\frac{v}{s}\frac{d}{2(2+d)}+\frac{2}{p(d+2)}\frac{\ent{s}}{s}+\frac{2(p-h-1)}{p(d+2)},\qquad \frac{1}{q_2}=\frac{2(p-h)}{p(d+2)}.
\end{equation}
Since $\sum_{k=1}^{h}|\beta_k|=|\alpha|=\ent{s}$, we see that the right-hand sides of the equalities in \eqref{def_param1} and \eqref{def_param2} are positive, and thus that $q_1$, $q_2$, and, for $i=\llbracket 1,h\rrbracket$, $r_i$ and $\rho_i$ are finite and positive. Using also that $v=s-\ent{s}$, we obtain
$$ \frac{1}{a}=\frac{1}{q_1}+\sum_{i=1}^h r_i=\frac{1}{q_2}+\sum_{\substack{1\leq i\leq h\\ i\neq k}}\frac{1}{r_i}+\frac{1}{\rho_k},$$
which proves that $q_1$, $q_2$, and, for $i=\llbracket 1,h\rrbracket$, $r_i$ and $\rho_i$ are all greater than $a$, and that the assumptions of Lemma \ref{lm product rule 2} are satisfied.

By Gagliardo-Nirenberg inequality Lemma \ref{lm homogeneous GN inequality} and the definition of $r_i$, $\rho_i$, we see that 
\begin{equation*}
 \left\| D^{\beta_i}u\right\|_{L^{r_i}_x}\lesssim \||\nabla|^su\|_{L^{\frac{2(2+d)}{d}}_x}^{\frac{|\beta_i|}{s}}\|u\|_{L^{\frac{p(d+2)}{2}}_x}^{1-\frac{|\beta_i|}{s}},\quad 
 \left\| D^{\beta_i}u\right\|_{\dot{H}^{v,\rho_i}}\lesssim \||\nabla|^su\|_{L^{\frac{2(2+d)}{d}}_x}^{\frac{|\beta_i|+v}{s}}\|u\|_{L^{\frac{p(d+2)}{2}}_x}^{1-\frac{|\beta_i|+v}{s}}.
\end{equation*} 
Integrating in time and using H\"older inequality, we obtain
\begin{equation}
\label{T60}
 \left\| D^{\beta_i}u\right\|_{L^{r_i}_{t,x}}\lesssim \|u\|_{\dot{W}^{s}}^{\frac{|\beta_i|}{s}}\|u\|_{X_p}^{1-\frac{|\beta_i|}{s}},\quad \left\| D^{\beta_i}u\right\|_{L^{\rho_i}_t\dot{H}^{v,\rho_i}} \lesssim\|u\|_{\dot{W}^{s}}^{\frac{|\beta_i|+v}{s}}\|u\|_{X_p}^{1-\frac{|\beta_i|+v}{s}}
\end{equation} 
By \eqref{T60}, we obtain: 
\begin{equation*}
\text{\eqref{number2}}\lesssim \norm{u}_{X_p}^{p-1}\norm{u}_{\dot{W}^s}.
\end{equation*}
Using the first inequality in  \eqref{T60}, \eqref{number1} is estimated by 
\begin{align}\text{\eqref{number1}}
&\lesssim \norm{f^{(h)}(u)}_{L^{q_1}\dot{H}^{v,q_1}}\prod_{i=1}^h \norm{u}_{\dot{W}^{s}}^{\frac{|\beta_i|}{s}}\norm{u}_{X_p}^{1-\frac{|\beta_i|}{s}}\nonumber\\
&=\norm{f^{(h)}(u)}_{L^{q_1}\dot{H}^{v,q_1}}\norm{u}_{\dot{W}^s}^{\frac{\ent{s}}{s}}\norm{u}_{X_p}^{h-\frac{\ent{s}}{s}}\label{number3}.
\end{align} 
In \eqref{number3}, we will estimate $\norm{f^{(h)}(u)}_{L^{q_1}\dot{H}^{v,q_1}}$ by $\norm{u}_{X_p}$ and $\norm{u}_{\dot{W}^{s}}$ using fractional chain rule Lemma \ref{lm fractional chain rule} and Gagliardo-Nirenberg inequality Lemma \ref{lm homogeneous GN inequality}.   

If $s \notin \N$ then $\supent{s} \geq h+1$ (because $h \leq |\alpha|=\ent{s}$). Using Fractional chain rule Lemma \ref{lm fractional chain rule}, we have
\begin{align*}
\norm{f^{(h)}(u)}_{L^{q_1}\dot{H}^{v,r_1}}
&\lesssim \norm{f^{(h+1)}(u)}_{L^{q_{3}}_{t,x}}\norm{u}_{L^{q_{4}}\dot{H}^{v,q_{4}}}\\
&\lesssim \norm{|u|^{p-h-1}}_{L^{q_{3}}_{t,x}}\norm{u}_{L^{q_{4}}\dot{H}^{v,q_{4}}}\\
&=\norm{u}^{p-h-1}_{L^{(p-h-1)q_{3}}_{t,x}}\norm{u}_{L^{q_{4}}\dot{H}^{v,q_{4}}},
\end{align*}
where 
\begin{equation*}
\frac{1}{q_3}=\frac{2(p-h-1)}{p(d+2)},\quad 
\frac{1}{q_4}=\frac{1}{q_1}-\frac{1}{q_3}=\frac{v}{s}\frac{d}{2(2+d)}+\frac{\ent{s}}{s}\frac{2}{p(d+2)}.
\end{equation*}
For $(q_4,r_4)$ as above, using Gagliardo-Nirenberg inequality Lemma \ref{lm homogeneous GN inequality}, we have
\[
\norm{u}_{L^{q_4}\dot{H}^{v,r_4}}\lesssim \norm{u}_{\dot{W}^s}^{\frac{v}{s}}\norm{u}_{X_p}^{\frac{\ent{s}}{s}}.
\]
Thus, for $s\notin \N$, 
\begin{equation}
 \label{T80}
\eqref{number1}\lesssim \eqref{number3}\lesssim \norm{u}^{p-h-1}_{X_p}\norm{u}_{\dot{W}^s}^{\frac{v}{s}}\norm{u}_{X_p}^{\frac{\ent{s}}{s}}\norm{u}_{\dot{W}^s}^{\frac{\ent{s}}{s}}\norm{u}_{X_p}^{h-\frac{\ent{s}}{s}}=\norm{u}_{\dot{W}^s}\norm{u}_{X_p}^{p-1},
\end{equation}
yielding \eqref{star}.
If $s\in\N$ then $v=0$, $\sum_{i=1}^k|\beta_i|=s$, and $
\frac{1}{q_1}=\frac{p-h}{p(d+2)/2}$. Using \eqref{T60}, we obtain
\begin{equation*}
\eqref{number1}\lesssim \eqref{number3}
\lesssim \norm{|u|^{p-h}}_{L^{q_1}_{t,x}}\norm{u}_{\dot{W}^s}\norm{u}_{X_p}^{h-1}=\norm{u}_{X_p}^{p-1}\norm{u}_{\dot{W}^s},
\end{equation*}
which proves \eqref{star} in this case also.

\medskip

\emph{Case $h=p$.}

In this case, we have $p=|\alpha|$, and thus, since $|\alpha|=\ent{s}\leq \supent{s}\leq p$, $s$ is an integer, $v=0$ and $s=p$.

By the assumptions on $f$, we have $|f^{(h)}(u)|\lesssim 1$. Thus
$$\eqref{number0}\lesssim \norm{f^{(h)}(u)\prod_{i=1}^h D^{\beta_i}u}_{L^{a}\dot{H}^{v,a}}\lesssim \norm{\prod_{i=1}^h D^{\beta_i}u}_{L^{a}\dot{H}^{v,a}}\lesssim \prod_{i=1}^h \norm{D^{\beta_i}u}_{L^{r_i}_{t,x}},$$
where the $r_i$ are defined as above. Using \eqref{T60} and since $\sum_{i=1}^h|\beta_i|=s$, and $h=p$, we obtain \eqref{star}, which concludes the proof of the Lemma. 
\end{proof}

\begin{proof}[Proof of Proposition \ref{P:boundg}]
Proposition \ref{P:boundg} follows easily from Lemmas \ref{L:boundg_onepower} and \ref{L:NL}. We fix $g\in \NNN(s,p_0,p_1)$ with $0\leq s$, $p_0\geq q_d$ and $p_0\geq\supent{s}$. All the norms used are over the interval $I$.

The estimate \eqref{bound_g} is exactly \eqref{diff_bound} with $v=0$. To prove \eqref{diff_bound}, we decompose $g(u)$ as follows:
$$ g(u)=P(u)+\tilde{g}(u).$$
Where $P(u)$ is the Taylor expansion of $g$ at $u=0$ up to order $\supent{s}$. As a consequence $P$ is a polynomial of the form 
$$P(u)=\sum_{p_1+1\leq k_1+k_2\leq \supent{s}} a_{k_1,k_2}u^{k_1}\overline{u}^{k_2}.$$
By Lemma \ref{L:boundg_onepower}, using that by the definition of $X$ and the assumptions $p_1\geq \frac{4}{d}$ and $p_0\geq \supent{s}$ we have
$$p_1\leq q\leq \supent{s}-1 \Longrightarrow\|u\|_{X_{q}}\leq \|u\|_{X},$$
we obtain
\begin{multline}
\label{estim:polyn}
 \left\|P(u)-P(v)\right\|_{N^s}\lesssim 
  \left(\|u\|_{X}^{p_0-1}+\|v\|_{X}^{p_0-1}+\|u\|_{X}^{p_1-1}+\|v\|_{X}^{p_1-1}\right)\\
  \times \Big[ \|u-v\|_{\dot{W}^s}\left(\|u\|_{X}+\|v\|_{X}\right)+\|u-v\|_{X}\left(\|u\|_{W^s}+\|v\|_{W^s}\right)\Big].
  \end{multline}
  Next, we notice that since $g\in \NNN(s,p_0,p_1)$, the definition of $P$ and the inequalities $p_1<p_0$ and $\supent{s}\leq p_0$ imply, for $k\in \llbracket 0,\supent{s}+1\rrbracket$, $z\in \C$,
  \begin{equation*}
  |z|\geq 1\Longrightarrow  \left|\tilde{g}^{(k)}(z)\right|\leq C |z|^{p_0+1-k},\quad |z|\leq 1\Longrightarrow  \left|\tilde{g}^{(k)}(z)\right|\leq C |z|^{\supent{s}+1-k}.
  \end{equation*}

  We let $\chi$ be a smooth function such that 
\[
\chi=\left\{
\begin{matrix}
1 \quad\text{ if } |x|<1,\\
0 \quad\text{ if } |x|>2.
\end{matrix}
\right.
\] 
Set $g_0(u)=(1-\chi(u))\tilde{g}(u)$, $g_1(u)=\chi(u)\tilde{g}(u)$, so that $g_0$ satisfies the assumptions of Lemma \ref{L:NL} with $p=\max(p_0,\supent{s})=p_0$, and $g_1$ the same assumptions with $p=\supent{s}$. Combining the conclusion of this Lemma for $g_0$, $g_1$ with the estimate \eqref{estim:polyn}, we obtain the conclusion of the proposition.
\end{proof}

\subsection{Cauchy and stability theory for general nonlinearities}
\label{sub:Cauchy}
In all this subsection, we fix $g$, $p_0$, $p_1$, $s$ such that Assumption \ref{Assum:NL} holds. We assume furthermore
\begin{equation}
 \label{Assum_LWP}
s\geq s_0=\frac{d}{2}-\frac{2}{p_0}. 
 \end{equation}

Note that these assumptions are satisfied for a sum of powers:
$$ g(u)=\sum_{j=0}^{k} \lambda_j |u|^{p_j}u,$$
where $k\geq 0$, $ \frac{4}{d}\leq p_1<\ldots<p_{k-1}<p_0$, $\lambda_j\in \R$ for all $j$, provided $g\in C^{\supent{s}+1}$ (which is the case, for example, when all the $p_j$ are even integers), $s\geq \frac{d}{2}-\frac{2}{p_0}$, and $\supent{s}\leq p_0$ if $g$ is not a polynomial.

\begin{definition}
Let $0\in I$ be an interval. By definition, a solution $u$ to \eqref{NLS} on $I$, with initial data in $H^s$ ($s\geq s_0$) is a function $u\in C(I,H^s) $ such that for all $K \subset I$ compact, $u\in S^s(K)$ and $u$ satisfies the following Duhamel formula
\begin{equation}
 \label{Duhamel}
u(t)=e^{it\Delta}u_0-i\int_0^t e^{i(t-\tau)\Delta}g(u)(\tau)\,d\tau,
 \end{equation} 
for all $t\in I$.
\end{definition}
We recall from \eqref{defX} the definition of $X(I)$. 
Noting that the assumption $p_0\geq \frac{4}{d}$ implies $p_0(d+2)/2>2$ (and $p_0(d+2)/2\geq 6$ if $d=1$), we can choose $q_0$ such that $(p_0(d+2)/2,q_0)$ is an admissible pair. By Sobolev inequality and the definitions of $s_0$, $q_0$, one can check
\begin{equation}
 \label{inclusion}
 \|u\|_{L^{\frac{p_0(d+2)}{2}}_x}\lesssim \left\| |\nabla |^{s_0} u\right\|_{L^{q_0}}
 \end{equation} 
and thus
\begin{equation}
 \label{Ss_X}
 \|u\|_{L^{\frac{p_0(d+2)}{2}}(I\times \R^d)}\lesssim \||\nabla |^{s_0}u\|_{S^{0}(I)},\quad \|u\|_{X(I)}\lesssim \|u\|_{S^{s_0}(I)}.
\end{equation}

\begin{proposition}\label{P:local wellposed}
Let $u_0\in H^s$. Let $g$, $p_0$, $p_1$, $s$ such that Assumption \ref{Assum:NL}  holds and $s\geq\frac{d}{2}-\frac{2}{p_0}$. Let $0\in I$ be an interval of $\R$, and $A>0$. Assume that $\norm{u_0}_{H^s}\leq A$ and
\begin{align*}
\norm{e^{i\cdot\Delta}u_0}_{X(I)}= \delta\leq \delta_0(A) \text{ small}.
\end{align*}
Then there exists a unique solution $u$ to \eqref{NLS} such that
\begin{equation*}
u(0)=u_0,\quad
\norm{u}_{S^s(I)}\lesssim \norm{u_0}_{H^s},\quad\norm{u}_{X(I)}\leq 2\delta.
\end{equation*}
Moreover, if $u_{0,k} \rightarrow u_0$ in $H^s$ (so that, for $k$ large, $\norm{e^{it\Delta}u_{0,k}}_{X(I)}< \delta$) then the corresponding solution $u_{k} \rightarrow u$ in $C(I,H^s)$.
\end{proposition}
\begin{remark}
By Strichartz estimates, \eqref{Ss_X} and Proposition \ref{P:local wellposed}, if $\norm{u_0}_{H^s}$ is small then $u$ is global and 
\[
\norm{u}_{S^s}\lesssim \norm{u_0}_{H^s}.
\]
\end{remark}
\begin{proof}
We use similar argument as in \cite[Lemma 2.5]{KeMe06}. We work on the interval $I$ for all norms.

Consider 
\[
B=\Big\{ u:\; \norm{u}_{X} \leq 2\delta,\; \norm{u}_{S^{s}} \leq MA \Big\},
\]
for some large universal contant $M$, with the topology induced by the norm in $S^s(I)$. 
%Rewrite \eqref{NLS} in Duhamel form:
% \[
% u(t)=e^{it\Delta}u_0-i\int_0^t e^{i(t-\tau)\Delta}g(u)(\tau), d\tau.
% \]
We denote by $\Psi(u)$ the right hand side of \eqref{Duhamel}. We show that if $\delta\leq \delta_0(A)$ small enough, and $M$ large enough (independently of $A$), $\Psi$ is a contraction map on $B$. 

For $u\in B$, we have, by Strichartz estimates and Proposition \ref{P:boundg}
\begin{multline*}
\norm{\Psi(u)}_{X(I)}\leq \norm{e^{it\Delta}u_0}_{X(I)}+\norm{\int_0^te^{i(t-\tau)\Delta}g(u)(\tau)\,d\tau}_{X(I)}\\
\leq \delta+C \norm{\int_0^te^{i(t-\tau)\Delta}g(u)(\tau)\,d\tau}_{S^s(I)}\leq \delta +C\|g(u)\|_{N^s(I)}\leq \delta+C \|u\|_{S^s}\left(\|u\|_{X}^{p_1}+\|u\|_{X}^{p_0}\right)
\end{multline*}
Thus,
\[
\norm{\Psi(u)}_{X(I)}\leq \delta+CM A\delta^{p_1}\leq 2\delta,
\]
if $\delta$ is small enough (so that $CM A\delta^{p_1-1}\leq 1$). 
Similarly,
\begin{equation*}
\norm{\Psi(u)}_{S^s}\leq\norm{e^{i\cdot\Delta}u_0}_{S^s}+C\norm{g(u)}_{N^s}\leq CA+ CM A\delta^{p_1}\leq MA,
\end{equation*}
choosing $M\geq 2C$, and $\delta\leq \delta_0(A)$ small enough. Thus $\Psi(u)\in B$.

Now, let $u,v\in B$. We have, by Strichartz estimates and Proposition \ref{P:boundg},
\begin{multline*}
\norm{\Psi(u)-\Psi(v)}_{S^s} \leq C\norm{g(u)-g(v)}_{N^s}\leq \|u-v\|_{S^s}\left(\|u\|_{S^s}+\|v\|_{S^s}\right)\left(\|u\|^{p_1-1}_X+\|v\|^{p_1-1}_X\right)\\
\leq C\delta^{p_1-1}MA\|u-v\|_{S^s}
\end{multline*}
Thus, taking $\delta$ small enough, we obtain that $\Psi$ is a contraction map on $B$. By the fixed point theorem, there exists a unique $u\in B$ such that $u=\Psi(u)$. Thus, $u$ is a solution to \eqref{NLS}. Also, since $u\in B$ (and since we can take $A=\|u_0\|_{H^s}$) we have as anounced $\|u\|_{S(I)}\lesssim \|u_0\|_{H^s}$, $\|u\|_{X(I)}\leq 2\delta$.
If $\norm{u_0}_{H^s}$ is small then $
\norm{e^{i\cdot\Delta}u_0}_{X(\R)} \leq C\norm{u_0}_{H^s}$ is small. Thus, $u$ is global and $
\norm{u}_{S^s(\R)} \leq M\norm{u_0}_{H^s}.$
\end{proof}

\begin{remark}
We have $\norm{e^{i\cdot\Delta}u_0}_{X(\R)} \leq C \norm{u_0}_{H^s}<\infty$. Thus, for each $\varepsilon>0$, there exists a small interval $I$ around $0$ such that $\norm{e^{i\cdot\Delta}u_0}_{X(I)}<\varepsilon$. This implies local existence of solution.
\end{remark}

\begin{remark}
\label{R:maximal}
If $u,v$ are two solutions to \eqref{NLS} on $0\in I$ such that $u(0)=v(0)$ then $u \equiv v$. We first show this assertion when $I'\subset I$ is a small interval around $0$. Shrinking $I'$ if necessary, we assume 
$$\|u\|_{X(I')}\leq\delta_0(A),\quad \|v\|_{X(I')}\leq \delta_0(A),$$
with $\delta_0(A)$ as in Proposition \ref{P:local wellposed}, and
$$A=\max_{t\in I'}\|u(t)\|_{H^s}+\|v(t)\|_{H^s}.$$
Thus, $u,v \in B$ (where $B$ is as in the preceding proof), which shows that $u=v$ on $I'$. Repeating this argument, we deduce that $u \equiv v$ on all $I$. This allows us to define a maximal interval $I_{\max}=(T_{-},T_{+})$ with $T_{-}<0$, $T_{+}>0$.
\end{remark}

\begin{lemma}
\label{L:scatt_criterion}
Let $u$ be a maximal solution on $I_{\max}$ as in Proposition \ref{P:local wellposed}, Remark \ref{R:maximal}. Assume that
\[
u\in X([0,T_{+})).
\] 
Then $T_{+}=\infty$, and $u$ scatters in the future in in $H^s$. A similar result holds in the past, if $u\in X((T_{-},0])$.
\end{lemma}
\begin{remark}
 Lemma \ref{L:scatt_criterion} implies the blow-up criterion $T_{+}<\infty \Longrightarrow \|u\|_{X(0,T_{+})}=\infty$ (and similarly for $T_{-}$). In the case $s>\frac{d}{2}-\frac{2}{p_0}$ we can also show (see e.g. \cite{CaWe90})  that if $T_{+}(u)$ is the maximum time of existence of $u$ and $T_{+}(u)<+\infty$ then
\[
\norm{u(t)}_{H^s}\rightarrow \infty, \text{ as } t\rightarrow T_{+}(u).
\]
\end{remark}
\begin{proof}
We only work on $[0,T_{+})$. In $(T_{-},0]$, we use similar argument. We only need to prove that
\begin{equation}
\label{finite_S}
\norm{u}_{S^s(0,T_{+})}<\infty.
\end{equation}
Indeed, if \eqref{finite_S} holds and $T_{+}$ is finite, then Proposition \ref{P:local wellposed} shows that one can extend the solution $u$ beyond $T_{+}$, a contradiction with the definition of $T_{+}$. Thus $T_{+}=\infty$, and Proposition \ref{P:boundg} implies $\|g(u)\|_{N([0,\infty))}<\infty$, which implies by standard arguments scattering in $H^s$.

% Let $\alpha$, $\beta$ be as in \eqref{def_alpha_beta}. Let
% $$ \varphi(t)=\left\|\langle \nabla\rangle^s u(t)\right\|_{L^{\frac{2d}{d-2}}},\quad \rho(t)=\|u(t)\|^{p_1}_{\frac{p_1(d+2)}{2}}+\|u(t)\|^{p_0}_{\frac{p_0(d-2)}{2}},$$
% and not that the assumption on $u$ implies $\rho \in L^{\frac{d-2}{2}}(0,T_{+})$. By Strichartz estimates, and Lemma \ref{L:product_rule}, for all $0\leq T<T_{+}$.
% \begin{equation*}
%  \|\varphi\|_{L^2(0,T)}\lesssim \|u_0\|_{H^s}+\left\| \langle \nabla\rangle^s (g_j(u)+g_0(u))\right\|_{L^{\alpha'}\left(0,T,L^{\beta'}\right)}\lesssim \|u_0\|_{H^s}+\left\| \rho \varphi\right\|_{L^{\alpha'}(0,T)}.
% \end{equation*}
% This is a Gronwall type inequality. By Lemma 8.1 of \cite{Ca..} we obtain $\varphi\in L^2(0,T_{+})$, i.e. $\langle u \rangle^s\in
% L^2(0,T_{+}L^{\frac{2d}{d-2}}$.

We divide $[0,T_{+})$ into finite intervals $I_k$ such that $\norm{u}_{X(I_k)}<\varepsilon$. On each $I_k=[t_k,t_{k+1}]$, we have
\begin{align*}
\norm{u}_{S^s(I_k)}&\leq C\norm{u(t_k)}_{H^s}+C\norm{u}_{S^s(I_k)}(\norm{u}_{X(I_k)}^{p_1}+\norm{u}_{X(I_k)}^{p_0})\\
&\leq C\norm{u(t_k)}_{H^s}+C\norm{u}_{S^s(I_k)}(\varepsilon^{p_1}+\varepsilon^{p_0}).
\end{align*}
Choosing $\varepsilon$ small and using that $\|u(t_{k+1})\|_{H^s}\lesssim \|u\|_{S^s(I_k)}$, we obtain by induction on $k$ that $\norm{u}_{L^{\infty}([0,T_{k\max}),H^s)}<\infty$ then \eqref{finite_S},
concluding the proof.
\end{proof}

One can also prove the existence of wave operators for equation \eqref{NLS}:
\begin{proposition}\label{P:wave operators}
Let $u_0\in H^s$. Let $g$, $p_0$, $p_1$, $s$ such that Assumption \ref{Assum:NL}  holds and $s\geq\frac{d}{2}-\frac{2}{p_0}$. Let $v_0\in H^s$ and $v_L(t)=e^{it\Delta}v_0$. Then there exist a unique solution $u\in C^0((T_{-}(u),\infty))$ of \eqref{NLS} such that
$$\lim_{t\to \infty}\|u(t)-v_L(t)\|_{H^s}=0.$$
\end{proposition}
We omit the proof which is similar to the proof of Proposition \ref{P:local wellposed}.

\begin{theorem}(Long time perturbation theory)\label{T:long time perturbation}
Let $A>0$, $s\geq s_0$. There exists constants $\eps(A,s)\in (0,1]$, $C(A,s)>0$ with the following properties.
Let $0\in I$ be an compact interval of $\R$ and $w$ be a solution of the following equation
\[
Lw=g(w)+e,
\] 
and $u_0\in H^s$ such that 
\begin{equation*}
\norm{w_0}_{H^s}+\norm{w}_{X(I)}\leq A,\quad
\norm{e}_{N^s(I)}+
\norm{u_0-w(0)}_{H^s)}=\varepsilon\leq \eps(A,s).
\end{equation*}
There the solution $u$ of \eqref{NLS} with initial data $u_0$ is defined on $I$ and satisfies
\begin{equation}
\label{bounduw}
\norm{u-w}_{S^s(I)}\leq C(A,s) \eps,\quad
\norm{u}_{S^s(I)}\leq C(A,s).
\end{equation}
\end{theorem}
\begin{proof}

Without loss of generality, using the reversibility of the equation, we can assume $I=[0,T)$. We can also assume that $I\subset I_{\max}(u)$ (indeed, if $T_{+}(u)<T$, the proof below will show that \eqref{bounduw} holds with $I$ replaced by $[0,T_{+})$, contradicting the blow-up criterion).

Divide $I$ into $J=J(A,s)$ subintervals $I_j=[t_j,t_{j+1}]$, $j\in [0,J-1]$, with $t_0=0$, such that 
\begin{equation}
\label{cond_Ij}
\norm{w}_{X(I_j)}\leq \delta, 
\end{equation} 
where $\delta$ is a small constant (depending only on $s$), to be specified.
On each $I_j$, by Strichartz and Proposition \ref{P:boundg}, we have
\begin{align*}
\norm{w}_{S^s(I_j)}&\leq C\norm{w(t_j)}_{H^s}+\norm{g(w)}_{N^s(I_j)}+\norm{e}_{N^s(I_j)}\\
&\lesssim   \norm{w(t_j)}_{H^s} +\norm{w}_{S^s(I_j)}(\norm{w}_{X(I_j)}^{p_1}+\norm{w}_{X(I_j)}^{p_0})+\varepsilon\\
&\lesssim   \norm{w(t_j)}_{H^s} +\varepsilon+\norm{w}_{S^s(I_j)}(\delta^{p_1}+\delta^{p_0}).
\end{align*}
Thus, for $\delta$ small enough, $\norm{w}_{S^s(I_j)}\leq \norm{w(t_j)}_{H^s} +\eps$. Since $\norm{w(t_{j+1})}_{H^s}\leq \norm{w}_{S^{s}(I_j)}$, $\norm{w(0)}_{H^s}\leq A$, we obtain by a finite induction $\norm{w(t_j)}_{H^s}\leq C(A,s)$, and thus
\begin{equation}\label{eq2}
\|w\|_{S^s(I)}\leq C(A,s).
\end{equation}
Thus, we may divide $I$ into $J=J(A,s)$ subintervals, which we still denote by $I_j=[t_j,t_{j+1}]$, $j\in \llbracket 0,J-1\rrbracket$, such that for all $j$,
\begin{equation}
\label{cond_Ij_bis}
\norm{w}_{W^s(I_j)}+\norm{w}_{X(I_j)} \leq \delta.
\end{equation}
Taking $\delta$ small (independently of $A$ and $\eps$), and $\eps\leq \eps(A,s)$ small, we prove by induction 
\begin{equation}
\label{boundSsIj}
\forall j\in \llbracket 0,J-1\rrbracket,\quad \|u-w\|_{S^s(I_j)}\leq (Kj+K+1)\eps,
\end{equation} 
for some large constant $K$ independent of $A$ and $\eps$. This will imply the desired conclusion \eqref{bounduw}.
 More precisely, we will prove, for $j\in \llbracket 0,J-1\rrbracket$,
\begin{equation}
 \label{inductionk}
 \norm{u(t_j)-w(t_j)}_{H^s}\leq (Kj+1)\eps\Longrightarrow \norm{u-w}_{S^s(I_j)}<(Kj+K+1)\eps,
\end{equation} 
which will yield \eqref{boundSsIj}, since $\norm{u(t_0)-w(t_0)}_{H^s}\leq \eps$. 

Assuming $j\in \llbracket 0,J-1\rrbracket$, $\norm{u(t_j)-w(t_j)}_{H^s}\leq (Kj+1)\eps$, we argue by contradiction, assuming also that there exists $T\in [t_j,t_{j+1}]$ such that $\norm{u-w}_{S^s([t_j,T])}=(Kj+K+1)\eps$. Then by the equation satisfied by $u-w$ and Strichartz estimates, we obtain
\begin{equation}
\label{interm1}
 \|u-w\|_{S^s([t_j,T])}\leq \|u(t_j)-v(t_j)\|_{H^s}+C\|e\|_{N^s([t_j,T])}+ C\|g(u-w)\|_{N^s([t_j,T])}
 \end{equation}
%  \leq \Big(K(j+1/2)+1\Big)\eps+ \right).
We have $\|u\|_{W^s([t_j,T])}\leq \|u-w\|_{W^{s}([t_j,T])}+\|w\|_{W^s([t_j,T])}\leq 
(Kj+K+1)\eps+\delta$, and similarly $\|u\|_{X^s([t_j,T])}\leq 
C(Kj+K+1)\eps+\delta$. Thus by Proposition \ref{P:boundg}, taking $\eps$ small enough, so that $(Kj+K+1)\eps\leq 1$.
$$\|g(u-w)\|_{N^s([t_j,T])}\leq C(Kj+K+1)\Big( (Kj+K+1)^{p_1}\eps^{p_1}+\delta^{p_1}\Big).$$
Going back to \eqref{interm1}, we deduce
$$(Kj+K+1)\eps= \|u-w\|_{S^s([t_j,T])}\leq C\eps+(Kj+1)\eps+ C(Kj+K+1)\eps \big( (Kj+K+1)^{p_1}\eps^{p_1}+\delta^{p_1}\big).$$
Taking $K$ large enough, $\eps\leq \eps(s,A)$ small (so that $(KJ+K+1)\eps\leq \delta$), and $\delta$ small, we obtain.
$$ (Kj+K+1)\eps\leq (Kj+K/2+1)\eps,$$
a contradiction which concludes the proof.
% On $I_0$, apply Lemma \ref{short time}, there exists a solution $u$ of \eqref{NLS} such that
% \begin{align*}
% \norm{u-w}_{X(I_0)}&\lesssim \norm{u-w}_{S^s(I_0)}\leq C(A_1)\varepsilon,
% \end{align*}
% Thus, 
% \[
% \norm{u(t_1)-w(t_1)}_{H^S}\leq C(A_1)\varepsilon.
% \]
% Apply Lemma \ref{short time} on $I_1$ for $(C(A_1)+1)\varepsilon$ instead of $\varepsilon$, we have
% \[
% \norm{u-w}_{S^s(I_1)}\leq \varepsilon C_0(A_1),
% \]
% then
% \[
% \norm{u(t_0)-w(t_0)}_{H^s}\leq \varepsilon C_0(A_1).
% \]
% By similar argument, we have $u$ solves \eqref{NLS} on all $I$ such that on each $I_k$,
% \[
% \norm{u-w}_{S^s(I_k)}\leq \varepsilon C_k(A_1).
% \]
% Thus, summing on all $I_k$, we have
% \[
% \norm{u-w}_{S^s(I)}\leq \sum_{k} C_k(A_1) \varepsilon\leq C(A_1,A_0,\varepsilon)\varepsilon.
% \]
% Combining to \eqref{eq2}, we have
% \[
% \norm{u}_{S^s(I)}\leq C(A_1,A_0,\varepsilon).
% \]
% This completes the proof.
\end{proof}

\begin{remark}
 \label{R:homogeneous}
Assume $d\geq 1$, $g(u)=|u|^{p_0}u$, $s\geq s_0 =\frac{d}{2}-\frac{2}{p_0}$, $p_0\geq \frac{4}{d}$ and $g\in C^{\supent{s}+1}$ (i.e. $p_0$ is an even integer or $p_0>\supent{s}$). Then analogs of Proposition \ref{P:local wellposed} and Theorem \ref{T:long time perturbation} where all the spaces are replaced by homogeneous spaces hold. Precisely, in the statement of Proposition \ref{P:local wellposed}, one can replace in this case $X(I)$ by $X_{p_0}(I)=L^{\frac{p_0(d+2)}{2}}(I\times \R^d)$, $H^s$ by $\dot{H}^s$, $S^s(I)$ by $\dot{S}^s(I)$, and similarly for Proposition \ref{P:local wellposed}. The proof is the same, replacing Proposition \ref{P:boundg} by Lemma \ref{L:boundg_onepower}. See also e.g. \cite{KeMe06}.
\end{remark}
\begin{proposition}[Conservation laws]
 \label{P:conservation}
 Let $g$ such that Assumption \ref{Assum:NL} holds and $g(z)=G'(|z|^2)z$ for a $C^1$ function $G:[0,\infty)\to\R$ with $G(0)=0$. Let $u\in C^0(I,H^s(\R^d))$ be a solution of \eqref{NLS_g}, where $s\geq \frac{d}{2}-\frac{2}{p_0}$. Then the mass
 $$M(u):=\int |u(t,x)|^2dx$$
 is conserved on $I$. Furthermore, if $s\geq 1$, then the momentum:
 $$P(u):=\Im \int \nabla u(t,x) \overline{u}(t,x)dx$$
 and the energy
 $$ E(u):=\int|\nabla u(t,x)|^2dx+\int G(|u(t,x)|^2)dx.$$
 are well-defined and conserved on $I$.
 \end{proposition}
We omit the classical proof. See e.g. \cite{Ca03}.
Observe that the assumptions of Proposition \ref{P:conservation} imply
\begin{equation}
\label{bound_G}
 G(|u|^2)\lesssim |u|^{p_0+2}+|u|^{p_1+2},
\end{equation}
so that by Sobolev inequalities, $E(u)$ is well-defined if $s\geq 1$.

\section{Profile decomposition}
\label{S:profile}
\subsection{Profiles in homogeneous Sobolev spaces}
\label{sub:homog_profiles}
Let $0<s<\frac{d}{2}$. We denote by $u_L(t)=e^{it\Delta}u_0$ the solution to the linear Schr\"odinger equation on $\R\times \R^d$
\begin{equation}
 \label{LS}
 i\partial_tu_L+\Delta u_L=0,
\end{equation} 
with initial data 
\begin{equation}
 \label{IDHs}
 u_{L\restriction t=0}=u_0\in \dot{H}^{s}.
\end{equation} 
\begin{definition}
 \label{D:hom_profile}
 Let $p>\frac{4}{d}$ such that $s=\frac{d}{2}-\frac{2}{p}$.
 A linear $\dot{H}^{s}$-\emph{profile}, in short \emph{profile}, is a sequence $(\varphi_{Ln})_n$, of solutions of \eqref{LS}, of the form 
 \begin{equation}
  \label{linear_profile}
 \varphi_{Ln}(t,x)=\frac{1}{\lambda_n^{\frac{2}{p}}}\varphi_L\left( \frac{t-t_n}{\lambda_n^2},\frac{x-x_n}{\lambda_n} \right), 
\end{equation}
 where $\varphi_L$ is a fixed solution of \eqref{LS}, \eqref{IDHs} and $\Lambda_n=(\lambda_n,t_n,x_n)_n$ is a sequence in $(0,\infty)\times \R\times \R^d$ (called sequence of transformations) such that
 \begin{equation}
  \label{lim_time}
  \lim_{n\to\infty}\frac{-t_n}{\lambda_n^2}=\tau\in \R\cup \{\pm\infty\}.
 \end{equation} 
\end{definition}
\begin{definition}
\label{D:ortho_equiv}
We say that two sequence of transformations $\Lambda_n=(\lambda_n,t_n,x_n)$ and $M_n=(\mu_n,s_n,y_n)$ are \emph{orthogonal} when they satisfy
$$\lim_{n\to\infty} \frac{|t_n-s_n|}{\lambda_n^2}+\frac{|x_n-y_n|}{\lambda_n}+\left|\log\left( \frac{\lambda_n}{\mu_n} \right)\right|=\infty.$$ 
 We say that two $\dot{H}^{s}$-profiles $\varphi_{Ln}$ and $\psi_{Ln}$ are \emph{equivalent} (in $\dot{H}^{s}$) when
 $$\lim_{n\to\infty}\left\|\varphi_{Ln}(0)-\psi_{Ln}(0)\right\|_{\dot{H}^{s}}=0.$$
 We say that they are \emph{orthogonal} when one of the two profiles is identically $0$ or when the corresponding sequence of transformations are orthogonal.
\end{definition}
\begin{remark}
For a given $\dot{H}^{s}$-profile $\varphi_{Ln}$, the choice of the solution $\varphi_L$ and the sequence of transformations  $(\Lambda_n)_n$ are not unique. However the definitions of equivalent and orthogonal profiles do not depend on these choices. Also, two equivalent $\dot{H}^{s}$-profiles are orthogonal to the same $\dot{H}^{s}$-profiles. 
\end{remark}

\begin{definition}\label{D:decomposition1} Let $(u_{0,n})_n$ be a bounded sequence in $\dot{H}^{s}$, and $u_{Ln}=e^{it\Delta}u_{0,n}$. We say that the sequence $\left(\Big(\varphi_{Ln}^j\Big)_{n}\right)_{j\geq 1}$ of $\dot{H}^{s}$-profiles is a \emph{profile decomposition} of $(u_{0,n})_n$ if these profiles are pairwise orthogonal, and satisfy
\begin{equation}
 \label{weak_limit_profiles}
\lim_{J\to\infty}\limsup_{n\to\infty}\left\|w_{Ln}^J\right\|_{X_p(\R)}+\left\| |\nabla|^sw_{Ln}^J\right\|_{W^0(\R)}=0,
\end{equation} 
where 
\begin{equation}
\label{def_wLnJ} 
w_{Ln}^J=u_{Ln}-\sum_{j=1}^J \varphi_{Ln}^j.
\end{equation} 
 \end{definition}

\begin{proposition}
\label{P:decomposition1}
For any bounded sequence $(u_{0,n})_n$ in $\dot{H}^{s}$, there exists a subsequence (that we still denote by $(u_{0,n})_n$) that admits a profile decomposition $\left((\varphi_{Ln}^j)_n\right)_{j}$. Furthermore, we have the Pythagorean expansion:
\begin{equation}
 \label{Pythagorean1}
 \forall J\geq 1,\quad \|u_{0,n}\|^2_{\dot{H}^{s}}=\sum_{j=1}^J \left\|\varphi_{L}^j(0)\right\|^2_{\dot{H}^{s}}+\left\|w_{Ln}^J(0)\right\|^2_{\dot{H}^{s}}+o(1),\quad n\to\infty,
\end{equation} 
where $w_n^J$ is as in Definition \ref{D:decomposition1}.
\end{proposition}
Proposition \ref{P:decomposition1} is proved by Shao in \cite{Shao09}, using the $L^2$-critical profile decomposition of Merle and Vega \cite{MeVe98}, generalized to higher dimension by B\'egout and Vargas \cite{BeVa07}. See also \cite{Keraani01} for the energy-critical case $p=\frac{4}{d-2}$.

\begin{remark}
 \label{R:weaklimit}
  In the notations of Definition \ref{D:decomposition1}, if $(u_{0,n})$ admits a profile decomposition, then
  \begin{gather}
  \label{profile0}
  (\lambda_{n}^j)^{\frac{2}{p}}\varphi_{Ln}^k(t_{j,n},x_{j,n}+\lambda_{n}^j\cdot )\xrightharpoonup[]{n\to\infty}0 \text{ in }\dot{H}^{s},\quad j\neq k\\
  \label{profile0bis}
  (\lambda_{n}^j)^{\frac{2}{p}}w_{Ln}^J(t_{j,n},x_{j,n}+\lambda_{n}^j\cdot )\xrightharpoonup[]{n\to\infty}0\text{ in }\dot{H}^{s},\quad J\geq j\\
  \label{profileDef}
  (\lambda_{n}^j)^{\frac{2}{p}}u_{Ln}(t_{j,n},x_{j,n}+\lambda_{n}^j\cdot )\xrightharpoonup[]{n\to\infty}\varphi_L^j(0)\text{ in }\dot{H}^{s}.
  \end{gather}
  Indeed, \eqref{profile0} follows from the orthogonality of the profiles. The property \eqref{profile0bis} follows from \eqref{weak_limit_profiles} and the orthogonality of the profiles. Finally, \eqref{profileDef} follows easily from the two other properties. %The convergence \eqref{weak_limit_profiles} shows that all possible nonzero weak limits as in \eqref{profileDef} must correspond to a nonzero profile. 
  By \eqref{profile0bis},\eqref{profileDef}, one sees that the initial data of the nonzero profiles are exactly the nonzero weak limits of the form \eqref{profileDef}. This implies in particular that the profile decomposition is unique, up to reordering and equivalent profiles, if one ignore the null profiles.
 \end{remark}

Let $q=\frac{2d}{d-2s}=\frac{d}{2}p$ be the Lebesgue exponent such that the Sobolev embedding $\dot{H}^{s}\subset L^{q}$ holds. Let $(u_{0,n})_n$ be a sequence that has a profile decomposition as above. We next prove the analog of the property \eqref{weak_limit_profiles} of the remainder $w_n^J$ in the space $L^{\infty}(\R,L^q)$ and obtain a Pythagorean expansion of the $L^q$ norm of $u_{0,n}$.
Extracting subsequences, we can assume that for all $j$, the following limit exists in $\R\cup\{\pm\infty\}$:
 \begin{equation}
 \label{deftauj}
 \tau^j=\lim_{n\to\infty} \frac{-t_n^j}{(\lambda_n^j)^2}.
\end{equation}
We can also assume (translating the profiles in time if necessary), $\tau^j\in\{0,-\infty,+\infty\}$. If $\tau^j\in \{\pm\infty\}$, we have
\begin{equation}
\label{phinj0}
\lim_{n\to\infty} \|\varphi_{Ln}^j(0)\|_{L^q}=0.
\end{equation}
\begin{lemma}
\label{L:PythagoreLq}
Let $(u_{0,n})_n$ be as in Proposition \ref{P:decomposition1}. Then
\begin{gather}
\label{limLqwnJ}
\lim_{J\rightarrow +\infty}\limsup_{n\rightarrow +\infty} \norm{w^j_{Ln}}_{L^{\infty}(\R,L^q)} =0\\
 \label{expansionLq}
 \lim_{n\to\infty}\|u_{0,n}\|_{L^q}^{q}=\sum_{\substack{j\geq 1\\ \tau^j=0}}\|\varphi_{Ln}^j(0)\|^{q}_{L^q}.
\end{gather}
\end{lemma}

\begin{proof}
This follows from the elliptic profile decomposition of Patrick G\'erard \cite{Gerard98}. We first prove \eqref{limLqwnJ} by contradiction, assuming that there exists $\eta>0$ and a sequence $J_k\to\infty$ such that
 \begin{equation}
 \label{absurd_wnJ}
 \forall k,\quad
  \lim_{J\to\infty}\limsup_{n\to\infty}\|w_{Ln}^{J_k}\|_{L^{\infty}(\R,L^q)}\geq \eta.
 \end{equation}
Let
 $$ \eps_k=\limsup_{n\to\infty}\|w_{Ln}^{J_k}\|_{X_p(\R)}\xrightarrow[k\to\infty]{}0,$$
 where the convergence to $0$ follows from \eqref{weak_limit_profiles}. For all $k$, there exists $N_k$ such that
$$ n\geq N_k\Longrightarrow \|w_{Ln}^{J_k}\|_{X_p(\R)}\leq 2\eps_k. $$
By \eqref{absurd_wnJ}, we can find $n_k\geq N_k$, $t_k\in \R$ such that
$$\|w_{Ln_k}^{J_k}(t_k)\|_{L^q}\geq \frac{1}{2}\eta.$$
As a consequence, by the main result of \cite{Gerard98}, we can find $\varphi \in \dot{H}^s\setminus\{0\}$ such that
$$ w_{Ln_k}^{J_k}(t_k)\xrightharpoonup[k\to\infty]{}\varphi.$$
By Strichartz estimates, this implies that the sequence $\left(w_{Ln_k}^{J_k}(\cdot +t_k)\right)_k$ converges weakly to
$e^{i\cdot\Delta}\varphi$ in $X_p(\R)$, contradicting \eqref{weak_limit_profiles}. This concludes the proof of \eqref{limLqwnJ}. As a consequence (using also \eqref{phinj0}), we have the ``elliptic'' profile decomposition at fixed time:
$$ u_{0,n}=\sum_{\substack{1\leq j\leq J\\ \tau^j=0}} \varphi_{Ln}^j(0)+\tilde{w}_{Ln}^J(0),\quad \lim_{J\to\infty}\limsup_{n\to\infty}\|\tilde{w}_{Ln}^J(0)\|_{L^q}=0.$$
The pseudo-Pythagorean expansion \eqref{expansionLq} then follows from \cite{Gerard98}: see the expansion (1.11) there.
\end{proof}

%\thomas{I changed \eqref{limLqwnJ} to
%$$\lim_{J\rightarrow +\infty}\limsup_{n\rightarrow +\infty} \norm{w^j_{Ln}}_{L^{\infty}(\R,L^q)} =0.$$
%}
%About your proof:
%\stefan{
%I think \eqref{limLqwnJ} should be replaced by $$\lim_{J\rightarrow +\infty}\limsup_{n\rightarrow +\infty} \norm{w^j_{Ln}}_{L^{\infty}(\R,L^q)} =0.$$
%Now, we give a simple proof of  \eqref{limLqwnJ}. By \eqref{Pythagorean1}, we have $$ \lim_{J\rightarrow +\infty}\limsup_{n\rightarrow +\infty} \norm{w^j_{Ln}(0)}_{\dot{H}^s}=0.$$
%}
%\thomas{This is wrong! One can construct a sequence $(u_{0,n})$ such that
%$$\lim_{n\to\infty} \|u_{n}\|_{L^{\infty}L^q}=0$$
%But
%$$\|u_{0,n}\|_{\dot{H}^s}=1.$$
%Say $u_{0,n}=e^{i\xi_n x_1}\xi_n^{-s}\varphi$, where $\xi_n$ goes to $\infty$, and $\varphi$ has $C^{\infty}$, compactly supported Fourier transform.

%For this sequence, the profile decomposition is trivial (all the profiles are zero).
%}

\subsection{Profiles in non-homogeneous Sobolev spaces}
\label{sub:NLprofiles}
In all this subsection, we fix $s_0>0$ and $p_0>\frac{4}{d}$ with $s_0=\frac{d}{2}-\frac{2}{p_0}$. We consider a sequence $(u_{0,n})$, which is bounded in the \emph{inhomogeneous} Sobolev space $H^{s_0}(\R^d)$. We assume that $u_{0,n}$ admits a $\dot{H}^{s_0}$ profile decomposition $(\varphi_{Ln}^j)_{n,j}$, where
$$ \varphi_{Ln}^j(t,x)=\frac{1}{(\lambda_n^j)^{\frac{2}{p_0}}}\varphi^j_L\left( \frac{t-t_n^j}{\lambda_n^j},\frac{x-x_n^j}{\lambda_n^j} \right).$$

\begin{claim}
 \label{C:modulation}
 For all $j$, the sequence $(\lambda_{n}^j)_n$ is bounded. 
\end{claim}
\begin{proof}
 We have 
 \begin{equation}
  \label{weakCV1}
  (\lambda_n^j)^{\frac{2}{p_0}}u_{Ln}(t_n^j,\lambda_n^j\cdot+x_n^j)\xrightharpoonup{n\to\infty} \varphi_L^j(0) \text{ weakly in }\dot{H}^{s_0}.
 \end{equation} 
 Furthermore,
 $$\left\|(\lambda_n^j)^{\frac 2{p_0}} u_{Ln}\left( t_n^j,\lambda_n^j\cdot+x_n^j\right) \right\|_{L^2}=(\lambda_n^j)^{\frac{2}{p_0}-\frac{N}{2}}\|u_{0,n}\|_{L^2}.$$
 Assume that $\lambda_{n}^j\to\infty$ along a subsequence in $n$. As a consequence, since $\|u_{0,n}\|_{L^2}$ is bounded, we obtain, along the same subsequence, that
 $(\lambda_n^j)^{\frac 2p_0} u_{Ln}\left( t_n^j,\lambda_n^j\cdot+x_n^j\right)$ converges strongly to $0$ in $L^2$. By uniqueness of the distributional limit,
 $\varphi^j_L(0)=0$ as announced.
\end{proof}
Using the claim and extracting subsequences we obtain that for all $j$, one of the following holds:
\begin{itemize}
\item $\varphi^j_L\equiv 0$. In this case we say that $(\varphi_{Ln}^j)_n$ is a \emph{null profile}, and denote $j\in \JJJ_0$. 
\item $\varphi^j_L\not \equiv 0$ and $\lim_{n\to\infty}\lambda_{n}^j=\lambda^j\in (0,\infty)$. In this case we say that $(\varphi_{Ln}^j)_n$ is a \emph{non-concentrating profile}, and denote $j\in \JJJ_{NC}$. The weak limit \eqref{weakCV1} and the fact that the sequence $(u_{0,n})$ is bounded in $H^{s_0}$ proves that $\varphi_0^j=\varphi^j_L(0)\in H^{s_0}$. Replacing $\varphi_0^j$ by $\frac{1}{(\lambda^j)^{\frac{N}{2}}}\varphi_0^j\left( \frac{x}{\lambda^j} \right)$ and $\lambda_n^j$ by $\lambda_n^j/\lambda^j$, we see that we can assume $\lambda^j=1$. As a consequence, we can assume $\lambda_n^j=1$ for all $n$ (this will modify the profile $\varphi_{Ln}^j$ only by a term which goes to $0$ in $H^{s_0}$ as $n\to\infty$).
\item $\varphi^j_L \not\equiv 0$ and $\lim_{n\to\infty}\lambda_n^j=0$. In this case we say that $(\varphi_{Ln}^j)_n$ is a \emph{concentrating profile}, and denote $j\in\JJJ_C$.
\end{itemize}
\begin{remark}
\label{R:no_Jc}
 Assume that the sequence $(u_{0,n})$ is bounded in $H^{s_2}$ for some $s_2>s_0$. Then it is easy to see that $\JJJ_C$ is empty, i.e. that there is no concentrating profile.
\end{remark}

\begin{remark}
\label{R:Hs}
 Let $s$ such that $0<s<s_0$. Then $\left((\varphi_{Ln}^j)_n\right)_{j\in \JJJ_{NC}}$ is a $\dot{H}^s$ profile decomposition of $(u_{0,n})_n$. Indeed, by Remark \ref{R:no_Jc} the $\dot{H}^s$-profile decomposition of this sequence has no concentrating profile.
 By the preceding subsection, we have the Pythagorean expansion
 \begin{equation}
 \label{PythagoreanHs}
 \forall J\geq 1,\quad \|u_{0,n}\|^2_{\dot{H}^{s}}=\sum_{\substack{j\in \JJJ_{NC}\\1\leq j\leq J}} \left\|\varphi_{L}^j(0)\right\|^2_{\dot{H}^{s}}+\left\|\overline{w}_{Ln}^J(0)\right\|^2_{\dot{H}^{s}}+o(1),\quad n\to\infty,
\end{equation}
where $\overline{w}_{Ln}^J=u_{Ln}-\sum_{\substack{j\in \JJJ_{NC}\\1\leq j\leq J}}   \varphi_{Ln}^j$. One can also prove, as a consequence of the orthogonality of the profiles,
 \begin{equation}
 \label{PythagoreanL2}
 \forall J\geq 1,\quad \|u_{0,n}\|^2_{L^2}=\sum_{\substack{j\in \JJJ_{NC}\\1\leq j\leq J}} \left\|\varphi_{L}^j(0)\right\|^2_{L^2}+\left\|\overline{w}_{Ln}^J(0)\right\|^2_{L^2}+o(1),\quad n\to\infty.
\end{equation}

% \stefan{Moreover,we have the following Pythagorean expansion, for all $J\geq 1$, as $n\rightarrow\infty$:
% \begin{gather}
% \label{Eq:Energy expansion}
% E(u_{0,n})-\sum_{\substack{ j\in \JJJ_{NC} \\1\leq j\leq J}} E(\varphi^j_L(0))-E(\bar{w}^J_{Ln}(0))=o_n(1),\\
% \label{Eq:Virial function expansion}
% \Phi(u_{0,n})-\sum_{\substack{ j\in \JJJ_{NC} \\1\leq j\leq J}} \Phi(\varphi^j_L(0))-\Phi(\bar{w}^J_{Ln}(0))=o_n(1).
% \end{gather}
%}
\end{remark}

\begin{remark}
\label{R:Lq}
Let $2<q<q_0=\frac{2d}{d-2s_0}$, and $\tau^j$ be defined by \eqref{deftauj}. Assume as before $\tau^j\in\{0,\pm\infty\}$. Then if $\tau_j\in \{\pm\infty\}$, one has, by standard properties of the linear Schr\"odinger equation,
\begin{equation}
\label{0Lq}
\lim_{n\to\infty}\|\varphi_{Ln}^j(0)\|_{L^q}=0.
\end{equation} 
Moreover, if $j\in \JJJ_C$, then \eqref{0Lq} holds by a simple scaling argument. Finally, using Lemma \ref{L:PythagoreLq} and the same argument as in Remark \ref{R:Hs}, one obtain
$$\lim_{n\to\infty}\|u_{0,n}\|_{L^q}^{q}=\sum_{\substack{j\in \JJJ_{NC}\\ \tau^j=0}}\|\varphi_{L}^j(0)\|^{q}_{L^q}.$$ 
\end{remark}

\subsection{Nonlinear profile decomposition}
\label{sub:NL}
We now construct a nonlinear profile decomposition, based on the preceding linear profile decomposition, and adapted to the equation \eqref{NLS_g},
where $g$ satisfy the following assumptions.
\begin{assumption}
\label{Assum:profile}
\begin{equation}
 \label{model_g}
g(u)= g_0(u)+g_1(u), \quad g_0(u)=\iota_0|u|^{p_0}u, 
\quad \iota_0\in \{\pm 1\}
 \end{equation} 
and $g_1(u)$ is a nonlinearity of lower order. Precisely, we will assume
\begin{equation}
\label{Assum_Profile}
g_1\in \NNN(s_0,p_2,p_1),\quad \frac{4}{d}<p_1\leq p_2<p_0,\quad 1<p_1, 
\end{equation} 
where as usual $s_0=\frac{d}{2}-\frac{2}{p_0}$, and
\begin{equation}
 \label{Assum_Profile2}
\supent{s_0}\leq p_0, \text{ or }g \text{ is a polynomial in }u,\;\overline{u}.
 \end{equation} 
 Furthermore, the nonlinearity is of the form $g(z)=G'(|z|^2)z$, $G\in C^1([0,\infty),\R)$ with $G(0)=0$.
\end{assumption}
We note that the last condition on $g$ implies conservation of the mass, and, if $s_0\geq 1$, of the
 energy and the momentum. Note that Assumption \ref{Assum:profile} implies Assumption \ref{Assum:NL} with $s=s_0$.

We will also consider the equation \eqref{NLS_g} with $g=g_0$, that is the homogeneous nonlinear Schr\"odinger equation \eqref{NLSh}
with initial data in $\dot{H}^{s_0}$.

To each (linear) profile $\varphi_{Ln}^j$, we associate a \emph{nonlinear profile} $\varphi_{n}^j$ and a \emph{modified nonlinear profile} $\tilde{\varphi}_{n}^j$  in the following way:
\begin{itemize}
\item If $j\in \JJJ_0$, $\tilde{\varphi}^j_n$ and $\varphi_n^j$ are both equal to the constant null function.
 \item If $j\in \JJJ_{NC}$, the 
 modified  nonlinear profile and the nonlinear profile are equal, and defined by 
 $$\tvarphi_n^j(t,x)=\varphi_n^j(t,x)=\varphi^j\left(t-t^j_n,x-x_n^j\right),$$
 where  $\varphi^j$ is the unique solution of \eqref{NLS_g} such that 
 \begin{equation}
 \lim_{t\to\tau^j}\left\|\varphi^j(t)-\varphi^j_L(t)\right\|_{H^{s_0}}=0,\quad \tau^j=\lim_{n\to\infty}-t^j_n.
 \end{equation}
 This solution is given by the well-posedness theory, Proposition \ref{P:local wellposed} (if $\tau^j$ is finite) or by the existence of wave operators, Proposition \ref{P:wave operators} (if $\tau^j\in \{\pm\infty\}$).
 \item If $j\in \JJJ_{C}$, the nonlinear profile $\varphi_n^j$ is defined by 
 \begin{equation}
 \label{defphinj}
 \varphi_{n}^j(t,x)=\frac{1}{(\lambda_n^j)^{\frac{2}{p_0}}}\varphi^j\left( \frac{t-t_n^j}{(\lambda_n^j)^2},\frac{x-x_n^j}{\lambda_n^j} \right)
 \end{equation} 
 where $\varphi^j$ is the unique solution of the homogeneous equation \eqref{NLSh} such that
 $$\lim_{t\to\tau^j} \left\|\varphi^j(t)-\varphi_L^j(t)\right\|_{\dot{H}^{s_0}}=0,\quad \tau^j=\lim_{n\to\infty}\frac{-t_{j,n}}{(\lambda_n^j)^2}.$$
  By definition of $\varphi^j$, we see that $\varphi^j(\tau)$  is in $\dot{H}^{s_0}$ for all $\tau$ in the domain of existence of $\varphi^j$. However it is not necessarily in $H^{s_0}$. To tackle with this difficulty, we fix $\sigma^j$ in the maximal interval of existence of $\varphi^j$ (if $\tau^j$ is finite, we can take $\sigma^j=\tau^j$, if $\tau^j=\pm\infty$, $|\sigma^j|$ large and with the same sign than $\tau^j$). We let 
 $$ s^j_n=(\lambda_n^j)^2\sigma^j+t_n^j$$
 and denote by $\tilde{\varphi}_n^j$ the solution of \eqref{NLSh} such that 
 \begin{equation}
 \label{def_tildephij}
 \tilde{\varphi}_n^j(s^j_n)=\chi\left( x-x_n^j \right)\varphi_n^j(s_n^j)= \chi\left(x-x_n^j\right)\frac{1}{(\lambda_n^j)^{\frac{2}{p_0}}}\varphi^j\left( \sigma^j,\frac{x-x_{n}^j}{\lambda_{n}^j} \right),  
 \end{equation} 
where $\chi\in C_0^{\infty}(\R^N)$ is radially symmetric, $\chi(x)=1$ for $|x|<1$, $\chi(x)=0$ for $|x|>2$.
 \end{itemize}

\begin{lemma}
\label{L:modifprofile1}
 Let $j\in \JJJ_C$. Then
 \begin{equation}
  \label{equiv_prof1}
  \lim_{n\to\infty}\left\|\tilde{\varphi}^j_n(0)-\varphi_{n}^j(0)\right\|_{\dot{H}^{s_0}}=0.
 \end{equation}
 More precisely, let $0\in J_n$ be a sequence of interval. Let $I^j_n=\left\{\frac{t-t_n^j}{(\lambda_n^j)^2},t\in J_n\right\}.$ Assume that there
exists an interval $I$ in the domain of existence of $\varphi^j$ such that $\|\varphi^j\|_{X_{p_0}(I)}<\infty$ and for large $n$, $I^j_n\subset I$. Then for large $n$, $J_n$ is included in the domain of existence of $\tilde{\varphi}^j_n$ and
\begin{gather}
 \label{equiv_prof2}
 \sup_{t\in J_n}\left\|\varphi_n^j(t)-\tilde{\varphi}_n^j(t)\right\|_{\dot{H}^{s_0}}+\|\varphi_n^j-\tilde{\varphi}_n^j \|_{\dot{S}^{s_0}(J_n)} \underset{n \to \infty}{\longrightarrow} 0\\
 \label{equiv_prof3}
\forall s\in [0,s_0),\quad \sup_{t\in J_n}\left\|\tilde{\varphi}_n^j(t)\right\|_{H^{s}}+\|\tilde{\varphi}_n^j \|_{S^{s}(J_n)} \underset{n \to \infty}{\longrightarrow} 0
 \end{gather}
 \end{lemma}
\begin{proof}
We first prove \eqref{equiv_prof2}. Without loss of generality, we can assume $\sigma_j\in I$. Let
$$\widetilde{\Phi}_n^j(\tau,y)=(\lambda_n^j)^{2/p} \tilde{\varphi}_n^j\left(t_n^j+(\lambda_n^j)^2\tau,x_n^j+\lambda_n^jy\right).$$
Then
$$\widetilde{\Phi}_n^j(\sigma_j,y)=(\lambda_n^j)^{\frac{2}{p_0}}\tilde{\varphi}_n^j\left(s_n^j,x_n^j+\lambda_n^j y\right)=\chi(\lambda_n^jy)\varphi^j(\sigma^j,y).$$
Thus, using that $\lambda_n^j\to 0$ as $n\to\infty$, we obtain that $I$ is included in the domain of existence of $\widetilde{\Phi}_n^j$ and
$$\lim_{n\to\infty} \left\|\widetilde{\Phi}_n^j\left(\sigma^j\right)-\varphi^j(\sigma^j)\right\|_{\dot{H}^{s_0}}=0.$$
Using the long time perturbation theory for equation \eqref{NLSh} (see Theorem \ref{T:long time perturbation} and Remark \ref{R:homogeneous}),
we obtain 
$$ \sup_{\tau \in I}\left\|\varphi^j(\tau)-\widetilde{\Phi}^j_n(\tau)\right\|_{\dot{H}^{s_0}}+\left\|\varphi^j-\widetilde{\Phi}^j_n\right\|_{\dot{S}^{s_0}(I)}\underset{n\to\infty}{\longrightarrow} 0.$$
By the change of variable $\tau=\frac{t-t_{j,n}}{(\lambda_{n}^j)^2}$, $y=\frac{x-x_{j,n}}{\lambda_{n}^j}$, we obtain \eqref{equiv_prof2}. Applying \eqref{equiv_prof2}, we obtain \eqref{equiv_prof1}.

%Taking $J_n=\left[-\lambda_{n}^j\eps,\lambda_{n}^j\eps\right]$, $\eps>0$ small, we obtain \eqref{equiv_prof1}.

To prove \eqref{equiv_prof3}, we first notice that using \eqref{def_tildephij} and $\lambda_n^j\to 0$, 
$$\lim_{n\to\infty} \|\tilde{\varphi}^j_n(s_n^j)\|_{L^2}=0,
$$
By conservation of the $L^2$ norm, and interpolation with the bound of the $\dot{H}^s$ norm which follows from \eqref{equiv_prof2}
\begin{equation}
\label{limHs}
\forall s\in [0,s_0),\quad
 \lim_{n\to \infty} \sup_{t\in J_n}\left\|\tilde{\varphi}^j_n\right\|_{H^{s}}=0.
\end{equation}
It remains to prove the second limit in \eqref{equiv_prof3}. We let $s\in [0,s_0)$. To any $J\subset J_n$ such that $\left\|\tilde{\varphi_n}^j\right\|_{X(J)}\leq \eps$ (where $\eps$ is a small constant), and $a_n\in J$, we have by Strichartz estimates and Proposition \ref{P:boundg},
$$ \norm{\tilde{\varphi}_n^j}_{S^s(J)}\lesssim \norm{\tilde{\varphi}_n^j(a_n)}_{H^s}+\left\|g_0(\tilde{\varphi}_n^j)\right\|_{N^s(J)}\lesssim \norm{\tilde{\varphi}_n^j(a_n)}_{H^s} +\norm{\tilde{\varphi}_n^j}_{S^{s}(J)}\norm{\tilde{\varphi}_n^j}_{X(J)}^{p_0},$$
and thus (if $\eps$ is small enough),
$$\norm{\tilde{\varphi}_n^j}_{S^s(J)}\lesssim \norm{\tilde{\varphi}_n^j(a_n)}_{H^s}.$$
Since 
$$\limsup_{n\to\infty} \left\|\tilde{\varphi}_n^j\right\|_{S^{s_0}(J_n)}<\infty,$$
and $S^{s_0}(J)$ is continuously embedded in $X(J)$, we can divide the interval $J_n$ in $N$ subintervals $J_n^k$, $k\in \llbracket 1,N\rrbracket$ ($N$ independent of $n$), such that
$\left\|\tilde{\varphi}_n^j\right\|_{X(J_n^k)}\leq \eps$. Arguing as in the proof of the long-time perturbation theory result (Theorem \ref{T:long time perturbation}), and using \eqref{limHs}, we obtain \eqref{equiv_prof3}.
\end{proof}
When $j\in \JJJ_C$, the modified profiles $\tilde{\varphi}_n^j$ are approximate solutions of \eqref{NLS_g}:
\begin{lemma}
\label{L:modifprofile2}
Let $j\in \JJJ_C$, and $J_n$ be as in Lemma \ref{L:modifprofile1}. Let 
\begin{equation}
 \label{deftildeenj}
 \tilde{e}_n^j=i\partial_t\tvp_n^j+\Delta \tvp_n^j-g(\tvp_n^j)=-g_1(\tilde{\varphi}_n^j).
\end{equation} 
Then
$$\lim_{n\to\infty}\left\|\tilde{e}_n^j\right\|_{N^{s_0}(J_n)}=0.$$
\end{lemma}
\begin{proof}
 Since $\tilde{\varphi}_n^j$ is a solution of \eqref{NLSh} with $p=p_0$, we indeed have $\tilde{e}_n^j=-g_1(\tilde{\varphi}_n^j)$. As a consequence, by Proposition \ref{P:boundg}, using that $g_1\in \NNN(s_0,p_2,p_1)$,
\begin{multline}
\norm{\tilde{e}_n^j}_{N^{s_0}(J_n)}\lesssim \norm{\tilde{\varphi}_n^j}_{S^{s_0}(J_n)}\left(\norm{\tilde{\varphi}_n^j}_{L^{p_1(d+2)/2}(J_n\times \R^d)}^{p_1}+\norm{\tilde{\varphi}_n^j}_{L^{p_2(d+2)/2}(J_n\times \R^d)}^{p_2}\right)\\
\lesssim \norm{\tilde{\varphi}_n^j}_{S^{s_0}(J_n)}\left(\norm{\tilde{\varphi}_n^j}_{S^{s_1}(J_n)}^{p_1}+\norm{\tilde{\varphi}_n^j}_{S^{s_2}(J_n)}^{p_2}\right),
 \end{multline}
where $s_k=\frac{d}{2}-\frac{2}{p_k}$, $k\in \{1,2\}$.
 Since $s_1<s_2<s_0$, the conclusion of the Lemma follows from \eqref{equiv_prof3}.
\end{proof}
We next give the announced approximation result. For this we must also modify the linear remainder $w_{L,n}^J$: we let $\tw_{L,n}^J$ be the solution of the linear wave equation with initial data
\begin{equation}
\label{def_tilde_w}
\tilde{w}_{L,n}^J(0)=u_{0,n}-\sum_{j=1}^J\tilde{\varphi}_n^j(0). 
\end{equation} 
\begin{claim}
 \label{C:modifw}
 For all $s$ with $0<s\leq s_0$,
 $$\lim_{J\to\infty}\limsup_{n\to\infty}\left\|\tilde{w}_{L,n}^J\right\|_{X(\R)}+\left\|\tilde{w}_{L,n}^J\right\|_{\dot{W}^s(\R)}=0.$$
\end{claim}
\begin{proof}
 We have
$$\tilde{w}_{L,n}^J(0)=w_{L,n}^J(0)+\sum_{1\leq j\leq J}(\varphi_{Ln}^j(0)-\tvarphi_n^j(0)).$$
 By Lemma \ref{L:modifprofile1} and linear profile decomposition, we have
 $$ \lim_{n\to \infty} \left\|\varphi_{Ln}^j(0)-\tilde{\varphi}_n^j(0)\right\|_{\dot{H}^{s_0}}\leq \lim_{n\to \infty} \left\|\varphi_{Ln}^j(0)-\varphi_n^j(0)\right\|_{\dot{H}^{s_0}}+\lim_{n\to \infty} \left\|\varphi_{n}^j(0)-\tilde{\varphi}_n^j(0)\right\|_{\dot{H}^{s_0}}\rightarrow 0.$$ Combining with \eqref{weak_limit_profiles}, we obtain
 \begin{equation}
  \label{bnd_s0level}
 \lim_{J\to\infty}\limsup_{n\to\infty}\left\|\tilde{w}_{L,n}^J\right\|_{X_{p_0}(\R)}+\norm{|\nabla|^{s_0}\tilde{w}_{L,n}^J}_{W^0(\R)}=0.
 \end{equation} 
 By Lemma \ref{L:modifprofile1}, we also have
 $$ \tilde{w}^J_{L,n}(0)=u_{0,n}-\sum_{\substack{j\in \JJJ_{NC}\\ 1\leq j\leq J}}\varphi_n^j(0)+o_n(1),\text{ in } H^s,\; 0\leq s<s_0,$$
 which shows by the Pythagorean expansion \eqref{PythagoreanL2} that 
 $$\sup_J\limsup_{n\to\infty} \|\tilde{w}_{L,n}^J(0)\|_{L^2}<\infty.$$
 Using Strichartz estimates, we obtain
 \begin{equation}
\label{estimate_tw}
 \lim_{J\to\infty}\limsup_{n\to\infty}\norm{\tilde{w}_{L,n}^J}_{W^0(\R)}<\infty,
 \end{equation} 
 which yields, combining with \eqref{bnd_s0level}, the conclusion of the claim.
 \end{proof}
 \begin{theorem}[Approximation by profiles]
\label{T:NLapprox}
Let $(u_{0,n})_n$ be a sequence bounded in $H^{s_0}$ that admits a profile decomposition $\left((\varphi^j_{Ln})_n\right)_{j}$. Define as above the nonlinear profiles $\varphi_n^j$, the modified nonlinear profiles $\tvarphi_n^j$ and the modified remainder $\tilde{w}_{L,n}^J$. Let $I_n$ be a sequence of intervals such that $0\in I_n$, and assume that for each $j\geq 1$, for large $n$,
$0\in I_n \subset I_{\max}(\varphi_n^j),$
\begin{gather}
\label{scatt_profile1}
j\in \JJJ_C\Longrightarrow \limsup_{n\to\infty} \left\||\nabla|^{s_0}\varphi_n^j\right\|_{S^{0}(I_n)}<\infty,\\
\label{scatt_profile2}
j\in \JJJ_{NC}\Longrightarrow \limsup_{n\to\infty} \left\|\varphi_n^j\right\|_{S^{s_0}(I_n)}<\infty.
\end{gather}
Let $u_n$ be the solution of \eqref{NLS_g} with initial data $u_{0,n}$. Then for large $n$, $I_n\subset I_{\max}(u_n)$, 
$$ u_n(t)=\sum_{1\leq j\leq J} \tvarphi_n^j(t,x)+\tilde{w}_{L,n}^J(t)+r_{n}^J(t),\quad t\in I_n,$$
with 
$$\limsup_{n\to \infty}\|u_n\|_{S^{s_0}(I_n)}<\infty$$
and
\begin{equation}
 \label{lim_rnJ}
\lim_{J\to\infty}\limsup_{n\to\infty}\left\|r_n^J\right\|_{S^{s_0}(I_n)}=0.
 \end{equation} 
\end{theorem}
We will also need the fact that Pythagorean expansions of the Sobolev norms hold in the setting of the preceding Theorem:

\begin{lemma}
\label{L:Pythagorean}
With the same assumptions and notations as in Theorem \ref{T:NLapprox}, if $(t_n)_n$ is a sequence of time with $t_n\in I_n$ for all time, then for all $J\geq 1$,
$$ \|u_n(t_n)\|^2_{\dot{H}^{s_0}}=\sum_{j=1}^J\|\tilde{\varphi}^j_n(t_n)\|_{\dot{H}^{s_0}}^2+\|\tilde{w}_n^J(t_n)\|^2_{\dot{H}^{s_0}}+o_{J,n}(1),$$
where $\lim_{J\to\infty}\limsup_{n\to\infty} o_{J,n}(1)=0$. Furthermore, for all $s$ with $0\leq s<s_0$, 
$$ \|u_n(t_n)\|^2_{\dot{H}^{s}}= \sum_{\substack{1\leq j\leq J\\ j\in \JJJ_{NC}}}\|\tilde{\varphi}^j_n(t_n)\|_{\dot{H}^{s}}^2+\|\tilde{w}_n^J(t_n)\|^2_{\dot{H}^{s}}+o_{J,n}(1),$$
\end{lemma}
Before proving Theorem \ref{T:NLapprox} and Lemma \ref{L:Pythagorean}, we need two technical lemmas.
\begin{lemma}
\label{L:technical}
 Let $(q,r)$ be a Schr\"odinger admissible pair with $q,r$ finite, $\varphi$ such that $|D|^{s_0}\varphi\in L^qL^r(\R\times \R^d)$, $\psi\in C_0^{\infty}(\R\times \R^d))$. Let $\Lambda_n=(\lambda_n,t_n,x_n)$ and $M_n=(\mu_n,s_n,y_n)$ be two   sequences of transformations that are orthogonal in the sense of Definition \ref{D:ortho_equiv}. Let 
 $$ \varphi_n(t,x)=\frac{1}{\lambda_{n}^{\frac{2}{p_0}}}\varphi\left( \frac{t-t_n}{\lambda_n^2},\frac{x-x_n}{\lambda_n} \right),\quad \psi_n(t,x)=\psi\left( \frac{t-s_n}{\mu_n^2},\frac{x-y_n}{\mu_n} \right).$$
 Assume that $(\mu_n/\lambda_n)_n$ is bounded.
 Then, 
 $$ \lim_{n\to\infty} \Big\| |\nabla|^{s_0}(\varphi_n\psi_n)\Big\|_{L^qL^r}=0.$$
\end{lemma}
\begin{proof}
By density (since $(q,r)$ are finite), we can assume $\varphi\in C_0^{\infty}(\R\times \R^d)$.
 Rescaling and translating $\varphi_n\psi_n$, we can assume (since $\mu_n/\lambda_n$ is bounded) that one of the following holds:
 \begin{itemize}
  \item $\lim_{n\to\infty}\lambda_n=\lim_{n\to\infty}\mu_n=1$. In this case, the fact that the two sequences are orthogonal implies that $\varphi_n\psi_n=0$ for large $n$.
  
  \item $\lim_{n}\lambda_n=\infty$ and $\forall n$, $\mu_n=1$, $s_n=0$, $y_n=0$. In this case, we have 
  $$|\varphi_n\psi_n(t,x)|\leq \frac{1}{\lambda_n^{2/p_0}}\|\varphi\|_{\infty}|\psi(t,x)|$$
  and thus $\varphi_n\psi_n$ goes to $0$ in $L^qL^r$ as $n\to\infty$. The same argument proves that for all $\alpha\in \N,\beta \in \N^d$, $\partial_t^{\alpha}\partial_x^\beta(\varphi_n\psi_n)$ goes to $0$ in $L^qL^r$ as $n\to\infty$. Interpolating we obtain the conclusion of the lemma.
  \end{itemize}
\end{proof}
\begin{lemma}
\label{L:ortho}
 Let $(\Lambda_n^j)_{n}=(\lambda_n^j,t_n^j,x_n^j)_n$, $1\leq j\leq J$ be a family of sequences of transformations that are pairwise orthogonal. For $j\in \llbracket 1,J\rrbracket$, we let $\varphi^j$ such that $|\nabla|^{s_0}\varphi^j\in S^{0}(\R)$, and let $\varphi_n^j$ be defined by \eqref{defphinj}.
Then 
\begin{gather}
\label{ortho_X0}
\lim_{n\to\infty} \Big\|\sum_{j=1}^J \varphi_{n}^j\Big\|_{X_{p_0}(\R)}^{\frac{p_0(d+2)}{2}}=\sum_{j=1}^J \norm{\varphi^j}_{X_{p_0}(\R)}^{\frac{p_0(d+2)}{2}}\\
\label{orthoWs}
\lim_{n\to\infty} \Big\|\sum_{j=1}^J |\nabla|^{s_0}\varphi_{n}^j\Big\|_{W^{0}(\R)}^{\frac{2(d+2)}{d}}=\sum_{j=1}^J \norm{|\nabla|^{s_0}\varphi^j}_{W^0(\R)}^{\frac{2(d+2)}{d}}\\
\label{ortho_g2}
\lim_{n\to\infty}\norm{g_0\left(\sum_{j=1}^J\varphi_n^j\right)-\sum_{j=1}^J g_0(\varphi_n^j)}_{\dot{N}^{s_0}(\R)}=0.
\end{gather} 
Furthermore, assuming that $\lambda_n^j=1$ and that $\varphi^j\in S^{s_0}(\R)$ for all $n$, one has
\begin{gather}
\label{ortho_X2}
\lim_{n\to\infty} \Big\|\sum_{j=1}^J \varphi_{n}^j\Big\|_{X_{p_1}(I)}^{\frac{p_1(d+2)}{2}}-\sum_{j=1}^J \norm{\varphi_{n}^j}_{X_{p_1}(I)}^{\frac{p_1(d+2)}{2}}=0\\
\label{orthoW}
\lim_{n\to\infty} \Big\|\sum_{j=1}^J\varphi_{n}^j\Big\|_{W^{0}(\R)}^{\frac{2(d+2)}{d}}-\sum_{j=1}^J \norm{\varphi_{n}^j}_{W^0(\R)}^{\frac{2(d+2)}{d}}=0\\
\label{orthoN}
\lim_{n\to\infty}\norm{g\left(\sum_{j=1}^J\varphi_n^j\right)-\sum_{j=1}^J g(\varphi_n^j)}_{N^{s_0}(\R)}=0.
\end{gather} 
\end{lemma}
\begin{proof}
\emph{Proof of \eqref{ortho_X2}, \eqref{orthoW} and \eqref{orthoN}.}
By a density argument (and Proposition \ref{P:boundg} for the proof of \eqref{orthoN}), we can assume that $\varphi^j\in C_0^{\infty}(\R\times \R^d)$. As a consequence of the orthogonality of the sequences $\Lambda_n^j$, we deduce that for large $n$, the supports of the functions $\varphi_n^j$, $j\in \llbracket 1,J\rrbracket$ are two-by-two disjoint. The three estimates follow immediately.

\emph{Proof of \eqref{ortho_g2}.}
By a density argument and \eqref{diff_boundj'}, we can assume $\varphi^j\in C_0^{\infty}\left(\R\times \R^d\right)$. We argue by induction on $J$. We fix $J\geq 2$. Arguing by contradiction, reordering the profiles and extracting subsequences, we see that we can assume
$$\forall n,\quad \lambda_n^J=\min_{j\in \llbracket 1,J\rrbracket} \lambda_n^j.$$
It is sufficient to prove:
\begin{equation}
 \label{inductionJ}
\lim_{n\to\infty}\norm{g_0\Big(\sum_{j=1}^J\varphi_n^j\Big)-g_0\Big(\sum_{j=1}^{J-1} \varphi_n^j\Big)-g_0(\varphi_n^J)}_{\dot{N}^{s_0}(\R)}=0. 
\end{equation}
We let $\chi\in C_0^{\infty}\left(\R\times \R^d\right)$, with $\chi=1$ on the support of $\varphi^J$. We let $\chi_n(t,x)=\chi\left( \frac{t-t^J_n}{(\lambda_n^J)^2},\frac{x-x_n^J}{\lambda_n^J} \right)$. By Lemma \ref{L:technical}, and Lemma \ref{L:boundg_onepower} or Lemma \ref{L:NL} (see \eqref{diff_boundj'}),
\begin{gather*}
g_0\Big( \sum_{j=1}^J \varphi_n^j \Big)=g_0\bigg((1-\chi_n)\sum_{j=1}^{J-1} \varphi_n^j+\varphi_n^J\bigg)+o_n(1)\text{ in }\dot{N}^{s_0}\\
g_0\Big( \sum_{j=1}^{J-1} \varphi_n^j \Big)=g_0\bigg( (1-\chi_n)\sum_{j=1}^{J-1} \varphi_n^j \bigg)+o_n(1)\text{ in }\dot{N}^{s_0}.
\end{gather*}
Next, we note that, since $g_0(0)=0$ and the supports of $(1-\chi_n)\sum_{j=1}^{J-1}\varphi_n^j$ and $\varphi_n^J$ are disjoint, one has
$$ g_0\bigg((1-\chi_n)\sum_{j=1}^{J-1} \varphi_n^j+\varphi_n^J\bigg)=g_0\bigg((1-\chi_n)\sum_{j=1}^{J-1} \varphi_n^j\bigg)+g_0\big(\varphi_n^J\big).$$
Combining the preceding estimates, we obtain \eqref{inductionJ}, and hence \eqref{ortho_g2}. We omit the similar proofs of \eqref{ortho_X0} and \eqref{orthoWs}.
\end{proof}
\begin{proof}[Proof of Theorem \ref{T:NLapprox}]
 We let, for $t\in I_n$,
 \begin{equation}
 \label{def_vnJ_enJ}
 v_n^J(t)=\sum_{1\leq j\leq J} \tilde{\varphi}_n^J(t,x)+\tilde{w}_{L,n}^J(t,x),\quad e_n^J=Lv_n^J-g(v^J_n).
 \end{equation} 
We will prove that there exists a constant $C>0$ such that 
\begin{equation}
 \label{NL10} 
 \forall J\geq 1,\quad\limsup_{n\to\infty} \|v_n^J\|_{X(I_n)}+\|v_n^J\|_{W^{s_0}(I_n)}\leq C,
\end{equation} 
and that for all $\delta >0$, there exists $J_{\delta}$ such that 
\begin{equation}
 \label{NL11} 
 \forall J\geq J_{\delta},\quad\limsup_{n\to\infty} \|e_n^J\|_{N^{s_0}(I_n)}\leq \delta.
\end{equation} 
By the definitions \eqref{def_wLnJ} and \eqref{def_tilde_w} of $w_{L,n}^J$ and $\tilde{w}_{L,n}^J$, we have $v_n^J(0)=u_{0,n}$. This implies, together with \eqref{NL10} and the assumption that $(u_{0,n})_n$ is bounded in $H^{s_0}$ that there exists a constant $C>0$ independent of $J$ such that 
$$\forall J, \quad \limsup_{n\to\infty}\|v_n^J(0)\|_{H^{s_0}}\leq C.$$
Thus we see that for $J\geq J_{\delta}$, and $n$ large enough, the assumptions of Theorem \ref{T:long time perturbation} are satisfied. The conclusion of Theorem \ref{T:NLapprox} follows. We are left with proving \eqref{NL10} and \eqref{NL11}. In all the proof, we will denote by $C$ a large positive constant that may change from line to line, might depend on the sequence $(u_n)_n$, but is \emph{independent of $J$}.

\smallskip

\emph{Proof of \eqref{NL10}.} We note that it is sufficient to prove the bound in \eqref{NL10} for $J$ large.

By the Pythagorean expansion \eqref{Pythagorean1} we have, 
$$\forall J\geq 1,\quad \sum_{j=1}^J \|\varphi^j_L(0)\|_{\dot{H}^{s_0}}^2 \leq \limsup_{n\to\infty}\|u_n(0)\|^2_{\dot{H}^{s_0}}\leq C.$$
Hence $\sum_{j=1}^{\infty}\|\varphi^j_L(0)\|_{\dot{H}^{s_0}}^2 <\infty$. Letting $\eps_0$ be a small positive number, we see that there exists $J_0\geq 1$ such that 
\begin{equation}
 \label{NL21} \sum_{j=J_0}^{\infty}\norm{\varphi_L^j(0)}^2_{\dot{H}^{s_0}}\leq \eps_0.
\end{equation}
Using the small data theory for equation \eqref{NLSh} (see Proposition \ref{P:local wellposed} and Remark \ref{R:homogeneous}), we deduce that if $j\in \JJJ_C$ and $j\geq J_0$, $\varphi^j$ is global and 
\begin{equation}
 \label{NL22}
 \sum_{\substack{j\geq J_0\\ j\in \JJJ_C}}\norm{\varphi^j}^2_{\dot{S}^{s_0}(\R)}<\infty
\end{equation} 
Arguing similarly with the Pythagorean expansion of the $H^{s_0}$ norm given by \eqref{Pythagorean1}, \eqref{PythagoreanL2}, we obtain (taking a larger $J_0$ if necessary) that if $j\in \JJJ_{NC}$ and $j\geq J_0$, $\varphi^j$ is global and 
\begin{equation}
 \label{NL23}
 \sum_{\substack{j\geq J_0\\ j\in \JJJ_{NC}}}\norm{\varphi^j}^2_{S^{s_0}(\R)}<\infty.
\end{equation} 
Next, we see that the assumptions of the theorem implies that for all $1\leq j\leq J_0-1$, there exists an interval $I^j\subset I_{\max}(\varphi^j)$ such that for large $n$, $\left\{\frac{t-t_n^j}{(\lambda_n^j)^2},\quad t\in I_n\right\}\subset I^j$ and
\begin{equation*}
  \begin{cases}
   \|\varphi^j\|_{\dot{S}^{s_0}(I^j)}<\infty &\text{ if }j\in \JJJ_C\\
   \|\varphi^j\|_{S^{s_0}(I^j)}<\infty &\text{ if }j\in \JJJ_{NC}.
  \end{cases}
 \end{equation*} 
 We let $\overline{\varphi}^j(t,x)=\varphi^j(t,x)$, if $t\in I^j$, $\overline{\varphi}^j(t,x)=0$ if $t\in \R\setminus I^j$. We denote by $\overline{\varphi}^j_n$ the corresponding modulated profiles, defined similarly as in \eqref{defphinj}. By \eqref{NL22}, \eqref{NL23}, the definition of $\tilde{\varphi}^j_n$ and Lemma \ref{L:ortho}, we obtain, for $J\geq J_0$.
 \begin{multline}
  \label{NL30}
  \limsup_{n\to\infty} \left\|\sum_{j=1}^J \tvarphi_n^j\right\|^{\frac{2(2+d)}{d}}_{\dot{W}^{s_0}(I_n)}=\limsup_{n\to\infty} \left\|\sum_{j=1}^J \varphi_n^j\right\|^{\frac{2(2+d)}{d}}_{\dot{W}^{s_0}(I_n)}\leq \limsup_{n\to\infty}\left\|\sum_{j=1}^{J_0-1}\overline{\varphi}_n^j+\sum_{j=J_0}^J\varphi_n^j\right\|_{\dot{W}^{s_0}(\R)}^{\frac{2(2+d)}{d}}\\
  =\sum_{j=1}^{J_0-1} \|\overline{\varphi}^j\|_{\dot{W}^{s_0}(\R)}^{\frac{2(2+d)}{d}}+\sum_{j=J_0}^{J} \|\varphi^j\|_{\dot{W}^{s_0}(\R)}^{\frac{2(2+d)}{d}}\leq C.
 \end{multline}
A similar argument yields
$
\limsup_{n\to\infty} \left\|\sum_{j=1}^J \tvarphi_n^j\right\|_{X_{p_0}(I_n)}\leq C$. We also have, for $0\leq s<s_0$, by \eqref{equiv_prof3} in  Lemma \ref{L:modifprofile1},
\begin{equation}
 \label{NL31} 
 \lim_{n\to\infty} \Bigg\|\sum_{j=1}^J\tvarphi_n^j\Bigg\|_{W^s(I_n)}^{\frac{2(2+d)}{d}}=\lim_{n\to\infty} \Bigg\|\sum_{\substack{1\leq j\leq J\\ j\in \JJJ_{NC}}}\varphi_n^j\Bigg\|_{W^s(I_n)}^{\frac{2(2+d)}{d}},
\end{equation} 
and the same argument as above, using \eqref{NL23} and Lemma \ref{L:ortho}, yields  $\limsup_{n\to\infty} \left\|\sum_{j=1}^J \tvarphi_n^j\right\|_{W^{s}(I_n)}\leq C$. Combining the estimates above with Claim \ref{C:modifw} and \eqref{estimate_tw}, we obtain \eqref{NL10}.
 
 \smallskip
 
 \emph{Proof of \eqref{NL11}.}
 We have, by \eqref{def_vnJ_enJ},
 \begin{equation*}
  e_n^J=\overbrace{\sum_{\substack{1\leq j\leq J\\ j\in \JJJ_{NC}}} L\tvarphi_n^j-g(\tvarphi_n^j)}^0+\overbrace{\sum_{\substack{1\leq j\leq J\\ j\in \JJJ_{C}}} L\tvarphi_n^j-g_0(\tvarphi_n^j)}^0+\overbrace{L\tw_{L,n}^J}^0-g(v_n^J)+\sum_{\substack{1\leq j\leq J\\ j\in \JJJ_{C}}} g_0(\tvarphi_n^j)+\sum_{\substack{1\leq j\leq J\\ j\in \JJJ_{NC}}} g\Big(\tvarphi_n^J\Big)
 \end{equation*}
 By Proposition \ref{P:boundg}, \eqref{estimate_tw} and \eqref{NL10},
\begin{equation*}
\left\| g(v_n^J)-g\Big(\sum_{j=1}^J \tilde{\varphi}_n^j\Big)\right\|_{N^{s_0}(I_n)} \leq C \left(\|\tilde{w}_{L,n}^J\|_{X(I_n)}+\|\tilde{w}_{L,n}^J\|_{\dot{W}^{s_0}(I_n)}\right)
 \end{equation*}
By Claim \ref{C:modifw}, there exists $J_{\delta}$ such that, for all $J\geq J_{\delta}$,
 \begin{equation}
  \label{NL50}
  \forall J\geq J_{\delta},\quad \limsup_{n\to\infty} \left\| g(v_n^J)-g\Big(\sum_{j=1}^J \tilde{\varphi}_n^j\Big)\right\|_{N^{s_0}(I_n)} \leq \frac{\delta}{2}.
  \end{equation}
 Using Lemma \ref{L:boundg_onepower} or Lemma \ref{L:NL} together with Lemma \ref{L:ortho}:
 \begin{equation}
\label{NL51}
  g_0\Big(\sum_{j=1}^J \tilde{\varphi}_n^j\Big)=
  g_0\Big(\sum_{j=1}^J \varphi_n^j\Big)+ o_n(1)= \sum_{j=1}^J  g_0\Big(\varphi_n^j\Big)+o_n(1) \text{ in } \dot{N}^{s_0}(I_n),
  \end{equation}
  and similarly (using also Lemma \ref{L:modifprofile1}),
\begin{multline}
\label{NL52}
    g_0\Big(\sum_{j=1}^J \tilde{\varphi}_n^j\Big)=   g_0\bigg(\sum_{\substack{1\leq j\leq J\\ j\in \JJJ_{NC}}} \tilde{\varphi}_n^j\bigg)+o_n(1)= \sum_{\substack{1\leq j\leq J\\ j\in \JJJ_{NC}}} g_0\Big(\tilde{\varphi}_n^j\Big)+o_n(1)\\= \sum_{{1\leq j\leq J}} g_0\Big(\tilde{\varphi}_n^j\Big)+o_n(1)  \text{ in }N^{0}(I_n)
\end{multline}
By Proposition \ref{P:boundg} and Lemma \ref{L:modifprofile1}, using that $g_1\in \NNN(s_0,p_0,p_1)$, we obtain
\begin{multline*}
\quad\left\|g_1\Big(\sum_{j=1}^J \tilde{\varphi}_n^j\Big)-  g_1\bigg(\sum_{\substack{1\leq j\leq J\\ j\in \JJJ_{NC}}} \tilde{\varphi}_n^j\bigg)\right\|_{N^0(I_n)}\\
\lesssim \sum_{j\in\JJJ_C, 1\leq j\leq J}\norm{\tilde{\varphi}^j_n}_{X_{p_1}}(\sum_{1\leq j\leq J}\norm{\tilde{\varphi}^j_n}_{X_{p_1}})^{p_1-1}\sum_{1\leq j\leq J}\norm{\tilde{\varphi}^j_n}_{W^0} \underset{n\to\infty}{\longrightarrow}0.\quad
\end{multline*}
Combining with Lemma \ref{L:modifprofile2} and Lemma \ref{L:ortho}, we obtain
\begin{equation}
\label{NL53}
    g_1\Big(\sum_{j=1}^J \tilde{\varphi}_n^j\Big)=   g_1\bigg(\sum_{\substack{1\leq j\leq J\\ j\in \JJJ_{NC}}} \tilde{\varphi}_n^j\bigg)+o_n(1)= \sum_{\substack{1\leq j\leq J\\ j\in \JJJ_{NC}}} g_1\Big(\tilde{\varphi}_n^j\Big)+o_n(1) \text{ in }N^{s_0}(I_n).
\end{equation}
By \eqref{NL50}, \eqref{NL51}, \eqref{NL52} and \eqref{NL53}, we obtain \eqref{NL11}, which concludes the proof.
 \end{proof}

\begin{proof}[Sketch of proof of Lemma \ref{L:Pythagorean}]
 By Theorem \ref{T:NLapprox}, $(u_n(t_n))_n$ is bounded in $H^{s_0}(\R^d)$, and 
 $$u_n(t_n)=\sum_{j=1}^J \tilde{\varphi}_n^j(t_n)+\tilde{w}_{L,n}^J(t_n)+r_n^J(t_n).$$
 We will interpret this expansion as a profile decomposition of $u_n(t_n)$. 
 Using the property of $\tilde{w}_{L,n}^J$ given by Claim \ref{C:modifw}, the property \eqref{lim_rnJ} of $r_n^J$ and the bound of $\varphi_n^j-\tilde{\varphi}_n^j$, $j\in \JJJ_C$ given by Lemma \ref{L:modifprofile1}, we obtain
 $$u_n(t_n,x)=\sum_{j=1}^J \varphi_n^j(t_n,x)+R_{n}^J(x),$$
 where 
 $$\lim_{J\to\infty}\limsup_{n\to\infty} \norm{e^{i\cdot \Delta}R_{n}^J}_{\dot{W}^{s_0}(\R)\cap X_{p_0}(\R)}=0.$$
 We have 
 $$\varphi_{n}^j(t_n,x)=\frac{1}{(\lambda_n^j)^{\frac{2}{p_0}}}\varphi^j\left( \frac{t_n-t_n^j}{(\lambda_n^j)^2},\frac{x-x_n^j}{\lambda_n^j} \right). $$
 Extracting subsequences, we can assume that $\frac{t_n-t_n^j}{(\lambda_n^j)^2}$ has a limit $\sigma^j$ as $n\to\infty$. We define a new linear profile $\psi^j_{Ln}$ by 
 $$\psi_{Ln}^j(t,x)=\frac{1}{(\lambda_n^j)^{\frac{2}{p_0}}}\psi^j_L\left( \frac{t+t_n-t_n^j}{(\lambda_n^j)^2},\frac{x-x_n^j}{\lambda_n^j} \right), $$
 where $\psi_{L}^j$ is the solution of the linear Schr\"odinger equation such that 
 $$\lim_{t\to \sigma^j} \left\|\psi_L^j(t)-\varphi^j(t)\right\|_{\dot{H}^{s_0}}=0.$$
 With these choice of $\psi_L^j$, we see that $\left((\psi^j_{L,n})_n\right)_{j\geq 1}$ is a $\dot{H}^{s_0}$ profile decomposition for the sequence $(u_n(t_n))_n$. The conclusion of the lemma follows from the Pythagorean expansions \eqref{Pythagorean1}, \eqref{PythagoreanHs} and \eqref{PythagoreanL2}.
\end{proof}
\section{Global well-posedness}
\label{S:GWP}

In this section we prove our theorem on global well-posedness, Theorem \ref{T:GWP}.

% that if all solutions of the homogeneous equation \eqref{NLSh} that are bounded in the critical Sobolev space scatters, then all solutions of the equation \eqref{NLS_g} (with a nonlinearity $g$ satisfying Assumption \ref{Assum:profile}) that are bounded in the same space are global. 

We first observe that Property \ref{Proper:bnd} is equivalent to the existence of uniform space-time bound for solutions of equation \eqref{NLSh} that are bounded in critical norm.
\begin{proposition}
\label{P:boundh}
Let $d\geq 2$, $p_0>\frac{4}{d}$, $s_0=\frac{d}{2}-\frac{2}{p_0}$. Assume Property \ref{Proper:bnd}. Then for all $A\in (0,A_0)$, there exists $\FFF(A)>0$ such that for all interval $0\in I$, for all solution $u\in C^0(I,\dot{H}^{s_0})$ of \eqref{NLSh} such that
\begin{equation}
\label{Hs0_bound}
\sup_{t\in I}\|u(t)\|_{\dot{H}^{s_0}}\leq A
\end{equation} 
we have 
\begin{equation}
\label{s0_bound}
\|u\|_{\dot{S}^{s_0}(I)}\leq \FFF(A).
\end{equation} 
\end{proposition}
\begin{proof}[Sketch of proof]
The proof is by contradiction. Let us denote by $\PPP(A)$ the property that there exists $\FFF(A)\in (0,\infty)$ such that \eqref{Hs0_bound} implies \eqref{s0_bound}.
By the small data theory for equation \eqref{NLSh}, $\PPP(A)$ holds for small $A>0$. Assuming that it does not hold for all $A\in (0,A_0)$, we obtain the existence of a critical $A_c\in (0,A_0)$ such hat $\PPP(A)$ holds for $A<A_c$, but $\PPP(A_c)$ does not hold. Thus there exists a sequence of  intervals $I_n=(a_n,b_n) \ni 0$, and of solutions $u_n\in C^0((a_n,b_n),\dot{H}^{s_0})$ such 
$$\sup_{a_n<t<b_n}\|u_n(t)\|_{\dot{H}^{s_0}}\underset{n\to\infty}{\longrightarrow} A_c,$$
and 
$$\lim_{n\to\infty}\|u_n\|_{\dot{S}^{s_0}(a_n,0)}=\lim_{n\to\infty}\|u_n\|_{\dot{S}^{s_0}(0,b_n)}=+\infty.$$
 By a standard compactness argument, using the homogeneous profile decomposition of Subsection \ref{sub:homog_profiles}, with the analog, for the homogeneous equation, of Theorem \ref{T:NLapprox}, we obtain, after extraction of subsequences, that there exists $x_n\in \R^N$, $\lambda_n>0$ and $\varphi_0\in \dot{H}^{s_0}\setminus \{0\}$ such that
 $$\lim_{n\to\infty}\left\| \frac{1}{\lambda_{n}^{\frac{2}{p_0}}}u_n\left(0,\frac{\cdot-x_n}{\lambda_n} \right)-\varphi_0\right\|_{\dot{H}^{s_0}}=0,$$
  and the solution $\varphi$ of \eqref{NLSh} with initial data $\varphi_0$ satisfies that for all $t\in \R$ there exist $x(t)\in \R^N$, $\lambda(t)>0$ such that
  $$ \left\{\frac{1}{\lambda(t)^{\frac{2}{p}}}\varphi\left(t,\frac{\cdot-x(t)}{\lambda(t)} \right),\;t\in I_{\max}(\varphi)\right\}$$
  has compact closure in $\dot{H}^{s_0}$
  and
  $$\sup_{t\in I_{\max}} \|\varphi(t)\|_{\dot{H}^{s_0}}\leq A_c$$
  (see the similar proof of Theorem \ref{T:scatt2} below). Then $\varphi$ is global by Property \ref{Proper:bnd}. Since $\varphi$ is not the zero solution, the preceding compactness property implies $\liminf_{t\to\infty}\|\varphi(t)\|_{L^{\frac{dp_0}{2}}}>0$, a contradiction with the fact that by Property \ref{Proper:bnd}, $\varphi$ must be scattering. The proof is complete.
  \end{proof}

\begin{proof}[Proof of Theorem \ref{T:GWP}]
We argue by contradiction, assuming that there exists a solution $u$ of \eqref{NLS_g} such that $T_{+}(u)<\infty$ and
$$\limsup_{n\to\infty}\|u(t)\|_{\dot{H}^{s_0}}<m \in (0,A_0).$$
Let $\eps_0$ be such that $m+3\eps_0<A_0$. By conservation of mass, we have indeed,
\begin{equation}
\label{boundHs2}
\limsup_{t\rightarrow T^+(u)}\|u(t)\|_{H^{s_0}}<\infty. 
\end{equation}
Let $t_n=T_{+}(u)-1/2^n$. Extracting subsequences, we can assume by Proposition \ref{P:decomposition1} that $(u(t_n))_n$ admits a profile decomposition $\left((\varphi_{L,n}^j)_n\right)_{j\geq 1}$. By the Pythagorean expansion \eqref{Pythagorean1} of the $\dot{H}^{s_0}$ norm, we have
\begin{equation}
\label{glbPytha}
\sum_{j\geq 1}\|\varphi_L^j(0)\|^2_{\dot{H}^{s_0}}\leq m.
\end{equation}
By the Pythagorean expansion of the $L^2$ norm,
\begin{equation}
\label{glbPytha2}
\sum_{j\geq 1}\|\varphi_L^j(0)\|^2_{L^2}<\infty.
\end{equation}

We denote by $\varphi_n^j$ the nonlinear profiles associated to the preceding profile decomposition. We will prove that for every $j\in \JJJ_c$, 
\begin{equation}
 \label{glb90}
 \sup_{0\leq \tau< T_{+}-t_n}\left\|\varphi_n^{j}(\tau)\right\|_{\dot{H}^{s_0}}\leq m+2\eps_0,\quad n\gg_j 1.
\end{equation}
Let $\eps>0$ be a small constant. By \eqref{glbPytha} , \eqref{glbPytha2} and the small data theory for equations \eqref{NLSh} and \eqref{NLS_g} there exists $J_0\geq 1$ such that for $j\geq J_0+1$, $\varphi^j$ is global and 
\begin{equation}
\label{bnd_above_J0}
\eps\geq
\begin{cases}
\|\varphi^j\|_{S^{s_0}(\R)} &\text{ if }j\in \JJJ_{NC}\\
\|\varphi^j\|_{\dot{S}^{s_0}(\R)} &\text{ if }j\in \JJJ_{C}.
\end{cases}
\end{equation} 
 In particular \eqref{glb90} is satisfied for $j\in \JJJ_C$, $j\geq J_0+1$. We next prove by contradiction that \eqref{glb90} holds for $j\in \JJJ_C\cap \llbracket 1,J_0\rrbracket$. If not, by \eqref{glbPytha}, there exists $\tau'_n\in [0,T_{+}-t_n)$ such that
\begin{equation}
  \label{glb91}
 \sup_{j\in\llbracket 1, J_0\rrbracket \cap \JJJ_C}\sup_{0\leq \tau\leq \tau'_n}\left\|\varphi_n^{j}(\tau)\right\|_{\dot{H}^{s_0}}=m+\eps_0 \in (0,A_0).
 \end{equation}
 By the local well-posedness theory for equation \eqref{NLS_g} and the fact that $\lambda_n^j=1$ for $j\in\JJJ_{NC}$, there exists $\tau_0>0$ such that
\begin{equation}
  \label{glb92}
  \limsup_{n\to\infty}\sup_{j\in \llbracket 1,J_0\rrbracket\cap \JJJ_{NC}}\|\varphi_n^j\|_{S^{s_0}(0,\tau_0)}<\infty.
 \end{equation}
Since $m+\eps_0<A_0$, by \eqref{glb91} and Proposition \ref{P:boundh}, we obtain
 $$\forall j\in \llbracket 1,J_0\rrbracket \cap \JJJ_c,\quad \limsup_{n\to\infty} \left\|\varphi_n^j\right\|_{\dot{S}^{s_0}(0,\tau'_n)}<\infty$$
 Note that since $\tau_n'\to 0$ as $n\to\infty$, we have $\tau_0>\tau_n'$ for large $n$. Combining with \eqref{bnd_above_J0} and \eqref{glb92}, we see that the assumptions of Theorem \ref{T:NLapprox} are satisfied on the interval $I_n=[0,\tau'_n)$. By Lemma \ref{L:Pythagorean}, for all sequence $(\sigma_n)_n$ with $0\leq \sigma_n\leq \tau'_n$, for all $J$,
 $$\limsup_n\sum_{j\in \llbracket 1,J\rrbracket \cap \JJJ_C}\left\|\varphi^j_n(\sigma_n)\right\|_{\dot{H}^{s_2}}\leq m.$$
 This clearly contradicts \eqref{glb91}, proving \eqref{glb90}. 
 
 Next, we observe that \eqref{glb90} implies by Proposition \ref{P:boundh}
 $$\forall j\in \llbracket 1,J_0\rrbracket \cap \JJJ_c,\quad \limsup_{n\to\infty} \left\|\varphi_n^j\right\|_{\dot{S}^{s_0}(0,T_{+}-t_n)}<\infty$$
 Combining this information with \eqref{bnd_above_J0} and \eqref{glb92}, we see
 that the Assumptions of Theorem \ref{T:NLapprox} are satisfied on the interval $I_n=[0,T_{+}-t_n)$. By the conclusion of the theorem, we obtain that for large $n$, $u(\cdot+t_n)\in S^{s_0}((0,T_{+}-t_n))$. This implies $u\in S^{s_0}((0,T_{+}))$, contradicting the blow-up criterion for equation \eqref{NLS_g}. The proof is complete.
\end{proof}

\section{General rigidity result}
\label{S:rigidity}
In this section, we consider equation \eqref{NLS_g}, where $g$ satisfies Assumption \ref{Assum:NL} p.~\pageref{Assum:NL}, and is of the form $g(u)=G'(|u|^2)u$ for some $C^1$ function $G$. We recall that with these assumptions, the mass $M(u)$, the energy $E(u)$ and the momentum $P(u)$ are conserved for $H^1\cap H^{s_0}$ solutions of \eqref{NLS_g}, where  as usual $s_0=\frac{d}{2}-\frac{2}{p_0}$. We will also consider the virial functional:
\begin{equation}
 \label{general_Phi}
\Phi(u)=\int |\nabla u|^2+\frac{d}{2}\int (G'(|u|^2)|u|^2-G(|u|^2).
 \end{equation} 
Since the assumptions on $g$ imply $G'(|u|^2)|u|^2+G(|u|^2)\lesssim |u|^{p_0+2}+|u|^{p_1+2}$, one easily checks, using Sobolev inequalities, that $\Phi$ is well-defined if $u\in H^{s_0}\cap H^1$.

We prove the following result:
\begin{proposition}\label{pro1}
With the assumptions above, let $u$ be a solution of \eqref{NLS_g} defined on $[0,\infty)$ such that there exists $x(t)$, $t\in [0,\infty)$ with
\begin{equation}
\label{defK}
K=\{u(t,x+x(t)); t\geq 0\} 
\end{equation} 
has compact closure in $H^{s_0}\cap H^1$. Then
\begin{equation}
\label{eqB12}
\min_{t\geq 0}\left|\Phi(u(t)) - \frac{|P(u)|^2}{M(u)}\right|=0.
\end{equation}
\end{proposition}

%\stefan{In radial setting, we obtain the following result.
%\begin{proposition}
%\label{pro2}
%Let $u$ be a radial solution of \eqref{NLS_g} defined on $[0,\infty)$ such that $K=\{u(t,x); t\geq 0\}$ has compact closure in $H^{s_0}$ ($s_0=\frac{d}{2}-\frac{2}{p_0}$). Then  
%\begin{equation}
%\label{eqB12}
%\min_{t\geq 0}\left|\Phi(u(t))\right|=0.
%\end{equation}
%\end{proposition}}
Define 
\begin{equation}\label{eqB23}
X(t)=2\frac{P(u)}{M(u)}t,
\end{equation}
where $P(u)$ is momentum. The proof of Proposition \ref{pro1} relies on an asymptotic estimate of $x(t)$ and a localized virial argument. We start with two lemmas.
\begin{lemma}\label{lm2}
With the assumptions above, let $u$ be a solution \eqref{NLS_g} such that there exists $x(t)$, $t\geq 0$ such that $K=\{u(t,x+x(t)); t\geq 0\}$ has compact closure in $H^{1}$. Then
\begin{equation}
\label{eqB14}
\lim_{t\rightarrow +\infty}\frac{|x(t)-X(t)|}{t}=0.
\end{equation}
\end{lemma}
\begin{proof}
We can assume that $x$ is continuous (see e.g \cite[Proposition 3.2]{DuHoRo08}). \\
We argue by contradiction, assuming that there exists a sequence $t_n\rightarrow +\infty$, $\varepsilon_0>0$ such that 
\begin{equation}\label{eqB20}
\frac{|x(t_n)-X(t_n)|}{t_n} \geq \varepsilon_0.
\end{equation}
Without loss of generality we may assume $x(0)=0$.\\
For $R>0$, we let
\begin{equation}
\label{eqB21}
t_0(R)=\inf\{t\geq 0; |x(t)-X(t)|\geq R\}.
\end{equation}
Since $x(0)=0=X(0)$ and $x(t)-X(t)$ is continuous, we have $t_0(R)>0$.\\
We define $R_n=|x(t_n)-X(t_n)|$ and $\tilde{t}_n=t_0(R_n)$ so that $t_n\geq \tilde{t}_n$. Thus we have the following properties:
\begin{align}
\forall 0\leq t<\tilde{t}_n&\Rightarrow |x(t)-X(t)|<R_n,\label{eqB22}\\
|x(\tilde{t}_n)-X(\tilde{t}_n)|&=R_n,\label{eqB23 2}\\
\frac{R_n}{\tilde{t}_n}&\geq \varepsilon_0, \label{eqB24}  
\end{align}
where \eqref{eqB24} follows from \eqref{eqB20} and $t_n\geq\tilde{t}_n$. \\
By precompactness of $K$, for any $\varepsilon>0$, there exists $R_0(\varepsilon)>0$ such that for all $t\geq 0$:
\begin{equation}
\label{eqB30}
\int_{|x-x(t)|\geq R_0(\varepsilon)} (|u|^2+|\nabla u|^2)\,dx\leq \varepsilon.
\end{equation} 
Let $\tilde{R}_n=R_n+R_0(\varepsilon)$ (for $\varepsilon$ small enough).\\
Let $\theta \in \mathcal{D}(\R)$ be such that $\theta(x)=x$ for $|x|\leq 1$, $\theta(x)=0$ for $|x|\geq 2$ and $\norm{\theta}_{L^{\infty}}<2$. We write $\varphi(x)=(\theta(x_1),\theta(x_2),\cdot,\cdot,\cdot,\theta(x_d))$. Thus, $\varphi(x)=x$ for $|x|\leq 1$ and $\norm{\varphi}_{L^{\infty}}<2d$.\\
We define
\begin{equation}
\label{eqB31}
z_{\tilde{R_n}}(t)=\int_{\R^d}\tilde{R_n}\varphi\left(\frac{x-X(t)}{\tilde{R}_n}\right)|u(t,x)|^2\,dx.
\end{equation}  
Let $t\in [0,\tilde{t}_n]$. We have $z_{\tilde{R}_n}'(t)=([z_{\tilde{R}_n}'(t)]_1,[z_{\tilde{R}_n}'(t)]_2,\cdot,\cdot,\cdot,[z_{\tilde{R}_n}'(t)]_d)$, where
\begin{align*}
[z_{\tilde{R}_n}'(t)]_j&=-X_j'(t)\int_{\R^d}\theta'\left(\frac{x_j-X_j(t)}{\tilde{R}_n}\right)|u(t,x)|^2\,dx\\
&\quad +2\Im \int_{\R^d}\theta'\left(\frac{x_j-X_j(t)}{\tilde{R}_n}\right) \overline{u}u_j\,dx, \quad (\text{ where } u_j=\partial_j u=\partial_{x_j}u). 
\end{align*}
For $|x_j-X_j(t)|\geq \tilde{R}_n$, we have $|x(t)-X(t)|\geq \tilde{R}_n$ then $|x-x(t)|\geq \tilde{R}_n-R_n$ (by the definition of $R_n$ and triangle inequality). Thus, $|x-x(t)|\geq R_0(\varepsilon)$. For $|x_j-X_j(t)|\leq \tilde{R}_n$, $\theta'\left(\frac{x_j-X_j(t)}{\tilde{R_n}}\right)=1$. By \eqref{eqB30} we deduce, for $t\in [0,\tilde{t}_n]$,
\begin{align}
[z_{\tilde{R}_n}'(t)]_j&=-\frac{2P_j(u)}{M(u)}M(u)+2P_j(u)+ O(\varepsilon)=O(\varepsilon).\label{eqB40}
\end{align}  
Furthermore, we have
\begin{align*}
z_{\tilde{R}_n}(0)&=\int_{\R^d}\tilde{R}_n\varphi\left(\frac{x}{\tilde{R}_n}\right)|u_0(x)|^2\,dx\\
&=\int_{|x|<R_0(\varepsilon)}\tilde{R}_n\varphi\left(\frac{x}{\tilde{R}_n}\right)|u_0(x)|^2\,dx+\int_{|x|>R_0(\varepsilon)}\tilde{R}_n\varphi\left(\frac{x}{\tilde{R}_n}\right)|u_0(x)|^2\,dx.
\end{align*}
Thus, for some constant $C>0$,
\begin{equation}
\label{eqB41}
|z_{\tilde{R}_n}(0)|\leq R_0(\varepsilon)M(u)+C\tilde{R}_n\varepsilon\leq 2R_0(\varepsilon)M(u)+CR_n\varepsilon,
\end{equation}
where we have used for the first bound that $|\varphi(x)|\lesssim |x|$ and for the second bound that $\varphi\in L^{\infty}$.\\
Furthermore,
\begin{align}
z_{\tilde{R}_n}(\tilde{t}_n)&=\int_{|x-x(\tilde{t}_n)|\geq R_0(\varepsilon)}\tilde{R}_n\varphi\left(\frac{x-X(\tilde{t}_n)}{\tilde{R}_n}\right)|u(\tilde{t}_n,x)|^2\,dx\nonumber\\
&\quad +\int_{|x-x(\tilde{t}_n)|\leq R_0(\varepsilon)}\tilde{R}_n\varphi\left(\frac{x-X(\tilde{t}_n)}{\tilde{R}_n}\right)|u(\tilde{t}_n,x)|^2\,dx\nonumber\\
&=I+ II.\label{eqB42}
\end{align}
Using \eqref{eqB30}, we have 
\begin{equation}
\label{eqB50}
|I|\lesssim \tilde{R}_n\varepsilon.
\end{equation}
Furthermore, in the integral defining $II$, we have:
\begin{align*}
|x-X(\tilde{t}_n)|&\leq |x-x(\tilde{t}_n)|+|x(\tilde{t}_n)-X(\tilde{t}_n)|\\
&\leq R_0(\varepsilon)+R_n=\tilde{R}_n,
\end{align*}
where we have used the definition of $\tilde{R}_n$.

We next write
\begin{align*} 
II&= \int_{|x-x(\tilde{t}_n)|\leq R_0(\varepsilon)}(x-X(\tilde{t}_n))|u(\tilde{t}_n,x)|^2\,dx\\
&= (x(\tilde{t}_n)-X(\tilde{t}_n))\int_{\R^d}|u(\tilde{t}_n,x)|^2\,dx
\\&\qquad-(x(\tilde{t}_n)-X(\tilde{t}_n))\int_{|x-x(\tilde{t}_n)|\geq R_0(\varepsilon)}|u(\tilde{t}_n,x)|^2\,dx+\int_{|x-x(\tilde{t}_n)|\leq R_0(\varepsilon)} (x-x(\tilde{t}_n))|u(\tilde{t}_n,x)|^2\,dx.
\end{align*}
By \eqref{eqB23 2}, \eqref{eqB30}, we have
\begin{equation}
\label{eqB51}
|II|\geq R_n(M(u)-C\varepsilon)-R_0(\varepsilon)M(u).
\end{equation}
By \eqref{eqB42}, \eqref{eqB50} and \eqref{eqB51}, we obtain 
\begin{align*}
|z_{\tilde{R}_n}(\tilde{t}_n)|&\geq R_n(M(u)-C\varepsilon)-R_0(\varepsilon)M(u)-\tilde{R}_n\varepsilon.
\end{align*}
Thus, 
\begin{equation}
\label{eqB60}
|z_{\tilde{R}_n}(\tilde{t}_n)|\geq R_n(M(u)-C\varepsilon)-2R_0(\varepsilon)M(u),
\end{equation}
where we chose $\varepsilon \ll M(u)$. \\
By \eqref{eqB40}, \eqref{eqB41} and \eqref{eqB60}, we have
\[
R_n M(u)\lesssim R_0(\varepsilon)M(u)+\varepsilon \tilde{t}_n.
\]
By \eqref{eqB24}, we deduce:
\[
R_n M(u)\lesssim R_0(\varepsilon)M(u)+\varepsilon\frac{R_n}{\varepsilon_0}.
\]
Choosing $\varepsilon\ll \varepsilon_0 M(u)$, we obtain $R_n\lesssim R_0(\varepsilon)$. \\
Letting $n\rightarrow +\infty$, we obtain a contradiction. This completes the proof. 
\end{proof}

The second lemma concerns the derivative of the localized virial functional. We consider
\[
W_R(t)=R\Im\int_{\R^d}\varphi\left(\frac{x-X(t)}{R}\right)\nabla u\overline{u}\,dx,
\]
where $X(t)$ is defined by \eqref{eqB23} and $\varphi$ is as in the proof of Lemma \ref{lm2}. By the relative compactness of $K$ and the continuous embedding of $H^1\cap H^{s_0}$ into $L^{p_0+2}\cap L^{p_1+2}$, we have that for any $\varepsilon>0$, there exists $R_1(\varepsilon)$ such that
\begin{equation}
\label{eqB70}
\int_{|x-x(t)|\geq R_1(\varepsilon)}|\nabla u|^2+|u|^2+|G'(|u|^2)||u|^2+|G(|u|^2)|\,dx\leq \varepsilon.
\end{equation}
By Cauchy-Schwarz inequality, we have
\begin{equation}
\label{eqB71}
|W_R(t)|\leq CR\norm{\nabla u(t)}_2\norm{u(t)}_2\lesssim R.
\end{equation}
For convenience, we denote $\sum_{j=1}^d$ by $\sum_j$, $\partial_{x_j}u=\partial_ju$ by $u_j$, $u_t$ by $\partial_t u$ and $\int_{\R^d}$ by $\int$. We have
\begin{equation*}
W_R'(t)=\underbrace{\sum_j R\Im\int \theta'\left(\frac{x_j-X_j(t)}{R}\right)\frac{-X_j'(t)}{R}u_j\overline{u}\,dx}_{A_1}+\underbrace{\sum_j R\Im\int\theta\left(\frac{x_j-X_j(t)}{R}\right)\partial_t(u_j\overline{u})\,dx}_{A_2}.
\end{equation*}
We will prove the following result on the terms $A_1$ and $A_2$:
\begin{lemma}\label{lm compute A1,A2}
Let $\varepsilon,\tilde{\varepsilon}>0$ be small. There exist $L(\tilde{\varepsilon})$ depending on $\tilde{\varepsilon}$ such that for $R\geq 2R_1(\tilde{\eps})$, $L(\tilde{\varepsilon})<t<\frac{R}{2\tilde{\varepsilon}}$, 
\begin{equation*}
A_1=\frac{-2|P(u)|^2}{M(u)}+O(\varepsilon),
\quad A_2=2\Phi(u)+O(\varepsilon).
\end{equation*}
\end{lemma}
\begin{proof}
Writing $\partial_t(u_j\overline{u})=(u_{tj}\overline{u}+u_j\overline{u}_t)$ and integrating by part, we obtain
\begin{align*}
% A_2&=\sum_j R\Im\int \theta\left(\frac{x_j-X_j(t)}{R}\right)(u_{tj}\overline{u}+u_j\overline{u}_t)\,dx\\
% &=\sum_j R\Im\int -\partial_j\left(\theta\left(\frac{x_j-X_j(t)}{R}\right)\overline{u}\right)u_t+\theta\left(\frac{x_j-X_j(t)}{R}\right)u_j\overline{u}_t\,dx\\
% &=\sum_j R\Im\int -\theta\left(\frac{x_j-X_j(t)}{R}\right)\overline{u}_ju_t-\frac{1}{R}\overline{u}u_t\theta'\left(\frac{x_j-X_j(t)}{R}\right)+\theta\left(\frac{x_j-X_j(t)}{R}\right)u_j\overline{u}_t\,dx\\
A_2&=\sum_j -2R\Im\int\theta\left(\frac{x_j-X_j(t)}{R}\right)\overline{u}_ju_t-\sum_j\Im\int\theta'\left(\frac{x_j-X_j(t)}{R}\right)\overline{u}u_t\\
&=\sum_j C_j+\sum_j D_j=I+II.
\end{align*}
Using the equation \eqref{NLS_g}, we obtain 
\begin{align*}
C_j%&= -2R\Im\int\theta\left(\frac{x_j-X_j(t)}{R}\right)\overline{u}_ju_t\\
%&=2R\Re\int\theta\left(\frac{x_j-X_j(t)}{R}\right)\overline{u}_j iu_t\,dx\\
&=2R\Re\int\theta\left(\frac{x_j-X_j(t)}{R}\right) \overline{u}_j(-\Delta u+G'(|u|^2)u)\,dx\\
&=2R\Re\int\nabla\left(\theta\left(\frac{x_j-X_j(t)}{R}\right)\overline{u}_j\right)\nabla u+\theta\left(\frac{x_j-X_j(t)}{R}\right)\frac{1}{2}\partial_jG(|u|^2)\\
&=2R\int\theta\left(\frac{x_j-X_j(t)}{R}\right)\partial_j\frac{|\nabla u|^2}{2}\,dx+2\int \theta'\left(\frac{x_j-X_j(t)}{R}\right)|u_j|^2\,dx\\
&\quad +2R\int\theta\left(\frac{x_j-X_j(t)}{R}\right)\frac{1}{2}\partial_jG(|u|^2)\,dx.
%&=2R\int\theta\left(\frac{x_j-X_j(t)}{R}\right)\partial_j\left(\frac{|\nabla u|^2}{2}+\frac{1}{2}G(|u|^2)\right)\,dx\\
%&\quad +2R\int\partial_j\theta\left(\frac{x_j-X_j(t)}{R}\right)|u_j|^2\,dx.
\end{align*}
Summing up, we obtain
\begin{align*}
%&=-2R\sum_j \int\partial_j \theta\left(\frac{x_j-X_j(t)}{R}\right)\left(\frac{|\nabla u|^2}{2}+\frac{1}{2}G(|u|^2)\right)\,dx\\
%&\quad +2R\sum_j\int\partial_j\theta\left(\frac{x_j-X_j(t)}{R}\right)|u_j|^2\,dx\\
I&=-2\int\sum_j\theta'\left(\frac{x_j-X_j(t)}{R}\right)\left(\frac{|\nabla u|^2}{2}+\frac{1}{2}G(|u|^2)\right)\,dx+2\sum_j\int\theta'\left(\frac{x_j-X_j(t)}{R}\right)|u_j|^2\,dx\\
&=-d\int|\nabla u|^2+G(|u|^2)\,dx -2\sum_j\int \left(\theta'\left(\frac{x_j-X_j(t)}{R}\right)-1\right)\left(\frac{|\nabla u|^2}{2}+\frac{1}{2}G(|u|^2)\right)\,dx\\
&\quad +2\int|\nabla u|^2\,dx+2\sum_j\int\left(\theta'\left(\frac{x_j-X_j(t)}{R}\right)-1\right)|u_j|^2\,dx.
\end{align*}
Moreover
\begin{align*}
D_j%&=-\Im\int\theta'\left(\frac{x_j-X_j(t)}{R}\right)\overline{u}u_t\\
&=\Re\int\theta'\left(\frac{x_j-X_j(t)}{R}\right)\overline{u}iu_t\,dx\\
&=\Re\int\theta'\left(\frac{x_j-X_j(t)}{R}\right)\overline{u}(-\Delta u+G'(|u|^2)u)\,dx\\
%&=\int \theta'\left(\frac{x_j-X_j(t)}{R}\right) G'(|u|^2)|u|^2\,dx-\int\theta'\left(\frac{x_j-X_j(t)}{R}\right)\Re(\overline{u}\Delta u)\,dx\\
&=\int \theta'\left(\frac{x_j-X_j(t)}{R}\right) G'(|u|^2)|u|^2\,dx-\int\theta'\left(\frac{x_j-X_j(t)}{R}\right)\left(\Delta\frac{|u|^2}{2}-|\nabla u|^2\right)\,dx\\
&=\int\theta'\left(\frac{x_j-X_j(t)}{R}\right) (G'(|u|^2)|u|^2+|\nabla u|^2)\,dx-\frac{1}{2R^2}\int\theta'''\left(\frac{x_j-X_j(t)}{R}\right)|u|^2\,dx.
\end{align*}
where we have used $ \int \theta'\left(\frac{x_j-X_j(t)}{R}\right)\partial_{kk}\frac{|u|^2}{2}=0 $ for all $k\neq j$.
Thus,
\begin{align*}
II
&=d\int|\nabla u|^2+G'(|u|^2)|u|^2\,dx\\
&\qquad +\sum_j \int \left(\theta'\left(\frac{x_j-X_j(t)}{R}\right) -1\right)(|\nabla u|^2+G'(|u|^2)|u|^2)\,dx-\frac{1}{2R^2}\int\theta'''\left(\frac{x_j-X_j(t)}{R}\right)|u|^2\,dx.
\end{align*}
Combining the above, we have
\begin{align}
A_2\notag
% I+II\nonumber\\
% &=-d\int |\nabla u|^2+G(|u|^2)\,dx -\sum_j\int \left(\theta'\left(\frac{x_j-X_j(t)}{R}\right)-1\right)\left(|\nabla u|^2+G(|u|^2)\right)\nonumber\\
% &\quad +2\int|\nabla u|^2\,dx+2\sum_j\int \left(\theta'\left(\frac{x_j-X_j(t)}{R}\right)-1\right)|u_j|^2\nonumber\\
% &\quad +d\int |\nabla u|^2+G'(|u|^2)|u|^2+\sum_j\int\left(\theta'\left(\frac{x_j-X_j(t)}{R}\right)-1\right)(|\nabla u|^2+G'(|u|^2)|u|^2) \,dx\nonumber\\
% &\quad -\sum_j\frac{1}{2R^2}\int\theta'''\left(\frac{x_j-X_j(t)}{R}\right)|u|^2\,dx\nonumber\\
% &=2\int |\nabla u|^2+d\int G'(|u|^2)|u|^2-G(|u|^2)\nonumber\\
% &\quad +\sum_j\int\left(\theta'\left(\frac{x_j-X_j(t)}{R}\right)-1\right)\left(G'(|u|^2)|u|^2-G(|u|^2)+2|u_j|^2\right)\nonumber\\
% &-\sum_j\frac{1}{2R^2}\int\theta'''\left(\frac{x_j-X_j(t)}{R}\right)|u|^2\,dx\nonumber\\
&=2\Phi(u)\\
&\qquad -\sum_j\frac{1}{2R^2}\int\theta'''\left(\frac{x_j-X_j(t)}{R}\right)|u|^2\,dx\label{eqsmall term 1}\\&\qquad +\sum_j\int \left(\theta'\left(\frac{x_j-X_j(t)}{R}\right)-1\right)\left(2|u_j|^2+G'(|u|^2)|u|^2-G(|u|^2)\right).\label{eqsmall term 2}
\end{align}
Applying Lemma \ref{lm2}, for each $\tilde{\varepsilon}$, there exists $L(\tilde{\varepsilon})$ such that
\begin{equation}
\label{eq 1000}
|x(t)-X(t)|\leq \tilde{\varepsilon} t,\quad \forall t \geq L(\tilde{\varepsilon}).
\end{equation} 
Assume $L(\tilde{\varepsilon})< t< \frac{R}{2\tilde{\varepsilon}}$ and $R\geq 2R_1(\varepsilon)$. Then
\begin{equation*}
|\text{\eqref{eqsmall term 1}}|+|\text{\eqref{eqsmall term 2}}|\lesssim \int_{|x-X(t)| \geq R} \frac{1}{2R^2}|u|^2+2|\nabla u|^2+|G'(|u|^2)||u|^2+|G(|u|^2)|\,dx,
\end{equation*}
where we have used the fact that for $|x-X(t)|\leq R$ then $|x_j-X_j(t)|\leq R$, for each $1\leq j\leq d$. Thus, $\theta'\left(\frac{x_j-X_j(t)}{R}\right)=1$ and $\theta'''\left(\frac{x_j-X_j(t)}{R}\right)=0$ for each $1\leq j\leq d$. \\
Moreover, for $|x-X(t)| \geq R$ and $L(\tilde{\varepsilon})< t< \frac{R}{2\tilde{\varepsilon}}$, we have 
$$|x-x(t)|\geq R-|x(t)-X(t)|\geq R-\tilde{\varepsilon} t> R/2\geq R_1(\varepsilon).$$
This implies that
\[
A_2=2\Phi(u)+O(\varepsilon).
\] 
Similarly, for $L(\tilde{\varepsilon})< t< \frac{R}{2\tilde{\varepsilon}}$, $R\geq 2R_1(\varepsilon)$, we have 
\begin{align*}
A_1&=\sum_j -X_j'(t)\Im \int \theta'\left(\frac{x_j-X_j(t)}{R}\right) u_j\overline{u}\,dx\\
&=\sum_j \frac{-2P_j(u)}{M(u)}\Im\int u_j\overline{u}\,dx -\frac{2P_j(u)}{M(u)}\Im\int \left(\theta'\left(\frac{x_j-X_j(t)}{R}\right)-1\right) u_j\overline{u}\,dx\\
&=\sum_j \frac{-2P_j(u)}{M(u)}P_j(u)+O(\varepsilon)=\frac{-2|P(u)|^2}{M(u)}+O(\varepsilon).
\end{align*}
This completes the proof of Lemma \ref{lm compute A1,A2}.
\end{proof}
We are now ready to prove Proposition \ref{pro1}. 
\label{proofprop}
\begin{proof}[Proof of Proposition \ref{pro1}]
Assume that $\inf_{t\geq 0}\left(\Phi(u)-\frac{|P(u)|^2}{M(u)}\right)=\delta>0$. We fix small parameters $\tilde{\eps}\ll \delta$ and $\eps\ll \delta$.

From Lemma \ref{lm compute A1,A2}, we have
\[
W'_R(t)= A_1+A_2=2\left(\Phi(u)-\frac{|P(u)|^2}{M(u)}\right)+O(\varepsilon),
\]
for $L(\tilde{\varepsilon})<t<\frac{R}{2\tilde{\varepsilon}}$, $R\geq 2R_1(\tilde{\eps})$. 

Let $T_1=L(\tilde{\varepsilon})$, $T_2=\frac{R}{2\tilde{\varepsilon}}$, where $R$ is large (and in particular $R>2\tilde{\eps}L(\tilde{\eps})$). We have
\begin{align*}
CR&\geq |W_R(T_2)-W_R(T_1)|=(T_2-T_1)|W_R'(t_0)|,\quad \text{for some } T_2>t_0>T_1\\
&\geq \delta\left(\frac{R}{\tilde{2\varepsilon}}-L(\tilde{\varepsilon})\right),
\end{align*} 
(where we have used $\eps\ll \delta$ for the last inequality).
This gives a contradiction letting $R\to\infty$, since $\tilde{\eps}\ll \delta$. 
\end{proof}

\section{Scattering}
 \label{S:scattering}

This section is dedicated to the proof of Theorems \ref{T:defoc_defoc}, \ref{T:defoc_defoc'}, \ref{T:scatt_intro2} and \ref{T:scatt_intro2'}.
Recall that $M(\varphi)=\int |\varphi|^2$. We consider the following property:
\begin{property}
\label{Assume_C}
There exists $m_c>0$ such that $\forall \varphi\in (H^{s_0}\cap H^1)\setminus\{0\}$, if $M(\varphi)<m_c$ then $\Phi(\varphi)\geq \frac 12\int |\nabla u|^2$ and $E(\varphi)\geq \frac 12\int |\nabla u|^2$.
\end{property}
Property \ref{Assume_C} holds for all $m_c$ when $g$ is defocusing. In the case where $\iota_0>0$, we have the following:
%\thomas{In the next lemma, the lower bound on the energy is necessary in the subcritical case $s_0\leq 1$, where it gives an a priori bound for the $H^1$ norm.}
\begin{lemma}
Let $g$ be a
nonlinearity that satisfies Assumption \ref{Assum:profile} with $\iota_0>0$. Then Property \ref{Assume_C} is true.
\end{lemma}
\begin{proof}
Using that $\iota_0=1$, we have
\begin{multline*}
\Phi(u)=\int |\nabla u|^2+\frac{d}{2}\int (G'(|u|^2)|u|^2-G(|u|^2)\\=\int |\nabla u|^2+\frac{d p_0}{2(2+p_0)}\int |u|^{p_0+2}+\frac{d}{2}\int (G_1'(|u|^2)|u|^2-G_1(|u|^2),
 \end{multline*}
 where $G_1(s)=G(s)-s^{p_0/2}$ is such that $g_1(u)=G_1'(|u|^2)u$. By Assumption \eqref{Assum:profile}, we have $g_1\in \NNN(s_0,p_2,p_1)$, for some $\frac{4}{d}<p_1<p_2<p_0$. Using the definition of $\NNN$, this implies
 \begin{equation}
 \label{boundG1}
 |G_1(|u|^2)|+|G_1'(|u|^2)|u|^2)|\lesssim |u|^{p_2+2}+|u|^{p_1+2}.
 \end{equation}
 The claim that there exists $m_c$ such that $\Phi(u)$ is larger than $\frac{1}{2}\int |\nabla u|^2$ when $0<M(u)<m_c$ follows easily from \eqref{boundG1} and the generalized Gagliardo-Nirenberg inequalities (See (3.25) in \cite{LeNo20})
 $$   \|u\|_{L^{p_j+2}}^{p_j+2}\lesssim \|u\|_{L^2}^{p_j-\theta p_0}\|\nabla u\|^{2(1-\theta)}_{L^2}\|u\|^{\theta(p_0+2)}_{L^{p_0+2)}}.$$
 The lower bound for the energy is obtained by the same proof.
 \end{proof}

We have the following result.
\begin{lemma}
\label{Lm:positive of Phi}
If Property \ref{Assume_C} holds, for each $\varphi\in H^{s_0}\cap H^1$ such that $M(\varphi)<m_c$, we have
$$\Phi(\varphi)-\frac{|P(\varphi)|^2}{M(\varphi)}>0.$$
\end{lemma}

We next prove Theorems \ref{T:defoc_defoc} and \ref{T:scatt_intro2}. The proofs of Theorems  \ref{T:defoc_defoc'} and \ref{T:scatt_intro2'} are similar, but simpler and we omit them. Theorems \ref{T:defoc_defoc} and \ref{T:scatt_intro2} are an immediate consequence of the following result:
\begin{theorem}
\label{thm:main_result_1}
Let $g$ satisfy Assumption \ref{Assum:profile}. Assume that $s_0\geq 1$, Properties \ref{Proper:bnd} and \ref{Assume_C} hold. Let $u$ be a solution of \eqref{NLS_g} with $M(u)<m_c$ and $u$ satisfies \eqref{bound_Hs0}. Then $u$ scatters forward in time.
\end{theorem}

We will prove Theorem \ref{thm:main_result_1} as a consequence of
\begin{theorem}
 \label{T:scatt2}
 Let $g$ satisfy Assumption \ref{Assum:profile}. Assume that $s_0\geq 1$, Properties \ref{Proper:bnd} and \ref{Assume_C} hold. Then for all $A\in (0,A_0)$ there exists $\FFF(A,\eta)>0$ such that for any interval $I$, for any solution $u\in C^0(I,H^{s_0})$ of \eqref{NLS_g} such that
 \begin{equation}
  \label{sct20}
 M(u)\leq m_c-\eta \quad
 \text{and}\quad\sup_{t\in I} \|u(t)\|_{\dot{H}^{s_0}}^2+\eta\|u(t)\|^2_2\leq A^2,
 \end{equation} 
 one has $u\in S^{s_0}(I)$ and $\|u\|_{S^{s_0}(I)}\leq \FFF(A,\eta)$.
\end{theorem}
Theorem \ref{T:scatt2} implies Theorem \ref{thm:main_result_1} by the scattering criterion (Lemma \ref{L:scatt_criterion}) and letting $\eta$ go to zero.
\begin{proof}[Proof of Theorem \ref{T:scatt2}]
 We argue by contradiction, following the compactness/rigidity scheme as in \cite{KeMe06}. We fix $\eta>0$ throughout the argument. In all the proof, we will endow $H^{s_0}$ with the norm defined by
 \begin{equation}
  \label{Hs0eta}
  \|u\|^2_{H^{s_0}}=\|u\|^2_{\dot{H}^{s_0}}+\eta\|u\|^2_2.
 \end{equation}
 
 We will denote by $\PPP(A)$ the property that there exists $\FFF(A)$ such that for any interval $I$, for any solution $u\in C^0(I,H^{s_0})$ such that \eqref{sct20} holds, one has $u\in S^{s_0}$ and $\|u\|_{S^{s_0}(I)}\leq \FFF(A)$.
 
 By the small data theory for \eqref{NLS_g} (see Proposition \ref{P:local wellposed}), if $A>0$ is small and $\|u(t)\|_{H^{s_0}}\leq A$ for some $t\in I_{\max}(u)$, then $u$ is globally defined, scatters and $\|u\|_{S^{s_0}(\R)}\lesssim A$. This implies that $\PPP(A)$ holds for small $A>0$.

 Thus if the conclusion of Theorem \ref{T:scatt2} does not hold, there exists $A_c\in (0,A_0)$ such that for all $A<A_c$, $\PPP(A)$ holds, and $\PPP(A_c)$ does not hold, i.e. there exists a sequence of intervals $((a_n,b_n))_n$, a sequence $(u_n)_n$ of solutions of \eqref{NLS_g} on $(a_n,b_n)$,
 \begin{equation}
  \label{sct32}
  u_n\in C^0((a_n,b_n),H^{s_0}),\quad M(u_n)\leq m_c-\eta,\quad 
  \lim_{n\to\infty}\sup_{a_n<t<b_n} \|u_n(t)\|_{H^{s_0}} =A_c.
 \end{equation} 
 and $\lim_{n\to\infty} \|u_n\|_{S^{s_0}((a_n,b_n))}=\infty$. Time translating $u_n$, we can assume
 \begin{equation}
  \label{sct33}
a_n<0<b_n,\quad 
  \lim_{n\to\infty}\|u_n\|_{S^{s_0}((a_n,0))}=\lim_{n\to\infty}\|u_n\|_{S^{s_0}((0,b_n))}=+\infty.
 \end{equation} 
 We will prove
 \begin{claim}
 \label{Cl:compactness}
  For any sequences $(a_n)_n$, $(b_n)_n$ with $a_n<0<b_n$, for any sequence $(u_n)_n$ of solutions of \eqref{NLS_g} satisfying \eqref{sct32}, \eqref{sct33}, there exist, after extraction of subsequences, a sequence $(x_n)_n\in (\R^d)^{\mathbb{N}}$ and $\varphi\in H^{s_0}$ such that
  \begin{equation}
  \label{convergence}
  \lim_{n\to\infty}\|u_n(\cdot-x_n)-\varphi\|_{H^{s_0}}=0.
  \end{equation} 
 \end{claim}
We first assume the claim and conclude the proof of Theorem \ref{T:scatt2}. By the claim, there exist (after extraction of subsequences) $\varphi\in H^{s_0}$ and $(x_n)_n$ such that \eqref{convergence} holds. Let $u$ be the solution of \eqref{NLS_g} such that $u(0)=\varphi$. Since $\RRR_{\eta}$ is closed in $H^{s_0}$, we have $\varphi\in \RRR_{\eta}$.

We next prove by contradiction
\begin{equation}
\label{lim_an_bn}
\lim_{n\to\infty}a_n=-\infty,\quad \lim_{n\to\infty} b_n=+\infty. 
\end{equation} 
Assume to fix ideas, and after extraction of subsequences $\lim_{n\to\infty} b_n =b\in [0,\infty)$. Using that $\lim_n\|u_n\|_{S^{s_0}([0,b_n))}=\infty$, we must have $T_{+}(u)<\infty$ and $b\geq T_{+}(u)$. By the last assertion of \eqref{sct32}, we obtain
$$\sup_{0\leq t<T_{+}(u)} \|u(t)\|_{H^{s_0}}<\infty.$$
This implies by Theorem \ref{T:GWP} that $T_{+}(u)=+\infty$, a contradiction. Hence \eqref{lim_an_bn}. Next, we see that \eqref{convergence}, perturbation theory for equation \eqref{NLS_g} and the last assertion in \eqref{sct32} implies that for any compact interval $I\subset I_{\max}(u)$,
$$
\sup_{t\in I}\|u(t)\|_{H^{s_0}}\leq A_c.$$
This implies by Theorem \ref{T:GWP} that $u$ is global and 
$$\sup_{t\in \R}\|u(t)\|_{H^{s_0}}\leq A_c.$$
By \eqref{sct33} and stability theory for equation \eqref{NLS_g}, one has
$$ \|u\|_{S^{s_0}((-\infty,0))}=\|u\|_{S^{s_0}((0,+\infty))}=+\infty.$$

If $(t_n)_n$ is any sequence of times, Claim \ref{Cl:compactness} and the preceding properties imply that one can extract subsequence such that $u(t_n,\cdot-x_n)$ converges in $H^{s_0}$ for some sequence $(x_n)_n\in (\R^d)^{\mathbb{N}}$. This is classical that it implies that one can find a function $x(t)$, $t\in \R$ such that $K$ defined by \eqref{defK} has compact closure in $H^{s_0}$. We give a sketch of proof of this fact. Using the compactness and the fact that the solution $u$ is not identically $0$, we first notice that there exists $R_0>0$, $\eta>0$ such that
$$\inf_{t\in \R}\left(\sup_{X\in \R^d}\int_{|x|<R_0}|u(t,x-X)|^2dx\right)\geq \eta.$$
Thus for all $t$, there exists $x(t)\in \R^d$ such that
$$ \int_{|x|<R_0} |u(t,x-x(t))|^2dx \geq \eta/2.$$
For this choice of $x(t)$, one can check that $K$ defined by \eqref{defK} is compact.

By Proposition \ref{pro1}, we have 
\begin{equation}
\label{Eq:bounded of Phi}
\min_{t\geq 0}\left|\Phi(u(t))-\frac{|P(u)|^2}{M(u)}\right|=0.
\end{equation}
%Since $u(t) \in \RRR_{\eta}$, by Theorem \ref{thm3}, $\Phi(u)-\frac{|P(u)|^2}{M(u)}>0$ for all $t\in \R$.
By \eqref{Eq:bounded of Phi}, there exists a sequence of times $(t_n)_n$ such that
\begin{equation} 
\label{Eq:convergences}
\lim_{t_n\rightarrow\infty}\Phi(u(t_n))-\frac{|P(u(t_n))|^2}{M(u(t_n))}=0.
\end{equation}
By the claim, extracting subsequences, there exists $(x_n)_n$ such that $u(t_n,\cdot-x_n)$ convergences to $\varphi_0$ in $H^{s_0}$ (up to extract subsequence). By \eqref{Eq:convergences},
\begin{equation}
\label{Eq:varphi0}
\Phi(\varphi_0)-\frac{|P(\varphi_0)|^2}{M(\varphi_0)}=0.
\end{equation}
Since $\RRR_{\eta}$ is closed in $H^{s_0}$, $\varphi_0\in \RRR_{\eta}$. This implies, by Lemma \ref{Lm:positive of Phi}, that
$$ \Phi(\varphi_0)-\frac{|P(\varphi_0)|^2}{M(\varphi_0)}>0.$$
This contradicts to \eqref{Eq:varphi0}. This completes the proof.

\end{proof}

We are left with proving Claim \ref{Cl:compactness}.
\begin{proof}[Proof of Claim \ref{Cl:compactness}]
 
 \emph{Step 1. Profile decomposition.}
 Extracting subsequences, we can assume that $u_{0,n}=u_n(0)$ has a profile decomposition as in Section \ref{S:profile}:
 $$u_{0,n}=\sum_{j=1}^J \varphi^j_{Ln}(0)+w_n^J,$$
 where $\varphi^j_{Ln}=\frac{1}{\lambda^j_n}\varphi^j_L\left( \frac{t-t_{n}^j}{\lambda_{n}^j},\frac{x-x_{n}^j}{\lambda^j_n} \right)$.
 We denote by $\varphi^j_{n}$ the corresponding nonlinear profiles, and $\tilde{\varphi}^j_n$ the modified nonlinear profiles. 
 
 Our goal is to prove that there is a unique $j_0\geq 1$ such that $\varphi^{j_0}$ is not identically zero, and that $j_0\in \JJJ_c$. We first note that there is at least one $j$ such that $\varphi^j$ is not identically zero. If not, by the small data local well-posedness, we would have 
 $$\lim_{n\to\infty}\|u_n\|_{S^{s_0}(I_n)}=0,$$
 a contradiction with our assumptions.
 
 In the remaining step, we will prove that there is at most one nonzero profile. Arguing by contradiction, we assume that there is at least two nonzero profiles, say $\varphi^1$ and $\varphi^2$. By the small data theory, there exists $\eps_0>0$ such that 
 \begin{equation}
  \label{sct40}
\inf_{t\in I_n}\left\|\tilde{\varphi}^j_n\right\|_{H^{s_0}}\geq \eps_0,\quad j\in \{1,2\}. 
  \end{equation} 
  
\emph{Step 2. Bound of the $H^{s_0}$ norm.}
  We prove that for all $j\geq 1$ we have $I_n\subset I_{\max}(\tilde{\varphi}^j_n)$ and, for large $n$, 
  \begin{equation}
   \label{sct41}
   \sup_{t\in I_n} \|\tilde{\varphi}_n^j(t)\|_{H^{s_0}} \leq \sqrt{A_c^2-\frac 14\eps_0^2}.
  \end{equation} 
  By \eqref{sct32}, and the Pythagorean expansions \eqref{Pythagorean1}, \eqref{PythagoreanL2}, we obtain the bounds, for $J\geq 1$
   \begin{equation*}
    \sum_{j\geq 1} \|\varphi^j_L(0)\|^2_{\dot{H}^{s_0}}\leq A_c^2,\quad \sum_{j\in \JJJ_{NC}} \|\varphi^j_L(0)\|^2_{H^{s_0}}\leq A_c^2
   \end{equation*}
   Fixing a small $\eps>0$, we obtain, by the small data theory for equations \eqref{NLS_g} and \eqref{NLSh} and Lemma \ref{L:modifprofile1}, that there exists $J_0\geq 1$ such that, for $j\geq J_0+1$, $I_{\max}(\varphi^j)=\R$,
\begin{equation}
\label{boundSlargej}
\forall j\geq J_0+1,\quad
\begin{cases}
   \|\varphi^j\|_{\dot{S}^{s_0}(\R)}<\infty &\text{ if }j\in \JJJ_C\\
   \|\varphi^j\|_{S^{s_0}(\R)}<\infty &\text{ if }j\in \JJJ_{NC}.
  \end{cases}
 \end{equation}    
and
\begin{equation}
\label{boundHs2_1}
\forall j\geq J_0+1,\quad
\limsup_{n\to\infty}\sup_{t\in \R}\|\tilde{\varphi}_n^j(t)\|_{H^{s_0}}\leq \eps.
\end{equation} 
We next prove by contradiction that for $j\in \llbracket 1, J_0\rrbracket$, $I_n\subset I_{\max}(\tilde{\varphi}_n^j)$ and  \eqref{sct41} holds. If not, we can assume (inverting time if necessary, and using Theorem \ref{T:GWP}) that for large $n$, there exists $b_n'\in (0,b_n]$ such that $[0,b_n']\subset \bigcap_{1\leq j\leq J_0} I_{\max}(\tilde{\varphi}_n^j)$ and
\begin{equation}
\label{boundHs2_absurd}
\sup_{1\leq j\leq J_0}\sup_{0\leq t\leq b_n'}\left\|\tilde{\varphi}_n^j(t)\right\|_{H^{s_0}}^2=A_c^2-\frac{1}{2}\eps_0^2. 
\end{equation} 
This implies  that for large $n$
$$ \forall j\in \JJJ_{C}\cap \llbracket 1,J_0\rrbracket,\quad  \sup_{0\leq t\in b_n'}\|\varphi_n^j(t)\|_{\dot{H}^{s_0}}< A_c<A_0.$$
Thus by Proposition \ref{P:boundh}, we obtain a constant $C>0$ such that for large $n$,
\begin{equation}
\label{boundJC}
\sup_{j\in \JJJ_C\cap \llbracket 1,J_0\rrbracket}\|\varphi^j_n\|_{\dot{S}^{s_0}([0,b_n'])}\leq C. 
\end{equation} 
Going back to \eqref{boundHs2_absurd}, we see also that for large $n$
$$\forall j\in \llbracket 1,J_0\rrbracket \cap \JJJ_{NC},\quad \sup_{0\leq t\leq b_n'} \|\varphi_n^j(t)\|_{H^{s_0}}\leq \sqrt{A_c^2-\frac{1}{4}\eps_0^2}.$$
Also, using the Pythagorean expansion of the mass we see that 
$$\forall j\in \JJJ_{NC}, \; M(\varphi^j)\leq m_c-\eta.$$
Using that $\PPP(A)$ holds for $A=\sqrt{A_c^2-\frac{1}{4}\eps_0^2}$ we obtain that there exists a constant $C>0$ such that for large $n$
\begin{equation}
\label{boundJNC}
\sup_{j\in \JJJ_{NC}\cap \llbracket 1,J_0\rrbracket}\|\varphi^j_n\|_{S^{s_0}([0,b_n'])}\leq C. 
\end{equation} 
Combining \eqref{boundSlargej}, \eqref{boundJC} and \eqref{boundJNC}, we obtain that the assumptions of Theorem \ref{T:NLapprox} are satisfied on $[0,b_n']$. Using the Pythagorean expansion of Lemma \ref{L:Pythagorean} together with the limit in \eqref{sct32}, we obtain
$$\limsup_{n\to\infty} \sup_{0\leq t\leq b_n'} \sum_{j=1}^{J_0}\left\|\tilde{\varphi}_n^j(t)\right\|_{H^{s_0}}^2\leq A_c^2.$$
By \eqref{sct40}, we deduce
$$ \forall j\in \llbracket 1,J_0\rrbracket, \quad \limsup_{n\to\infty}\sup_{0\leq t\leq b_n'} \left\|\tilde{\varphi}_n^j(t)\right\|_{H^{s_0}}^2\leq A_c^2-\eps_0^2,$$
contradicting \eqref{boundHs2_absurd}. This proves that \eqref{sct41} holds for all $j\geq 1$, for large $n$.

\medskip

\emph{Step 3. Uniqueness of the nonzero profile.}

In this step we still assume that $\varphi^1$ and $\varphi^2$ are nonzero profiles.
Using \eqref{sct32} and \eqref{sct41}, and arguing as in Step 2, we see that the assumptions of Theorem \ref{T:NLapprox} are satisfied on $[a_n,b_n]$. This proves that $u_n$ scatters for large $n$, contradicting \eqref{sct33}. This concludes the proof that there is only one nonzero profile.

\medskip

\emph{Step 4. End of the proof.}

We assume that $\varphi^1$ is the only nonzero profile. By the same argument as before, we obtain that for large $n$, $I_n\subset I_{\max}(\tilde{\varphi}^1_n)$ and
$$\lim_{n\to\infty}\sup_{t\in I_n} \|\tilde{\varphi}^1_n\|_{H^{s_0}}\leq A_c.$$
If $1\in \JJJ_{C}$, we obtain by Proposition \ref{P:boundh} that $\limsup_{n\to\infty}\|\varphi^1_n\|_{\dot{S}^{s_0}(I_n)}<\infty$. Thus the assumptions of Theorem \ref{T:NLapprox} are satisfied on $I_n$, a contradiction with \eqref{sct33}. Thus $1\in \JJJ_{NC}$. By the same argument, we obtain $\limsup_{n\to\infty} \sup_{t\in I_n} \|\tilde{w}_{Ln}^1(t)\|_{H^{s_0}}=0$. Indeed, if not, we would have by the conservation of the $H^{s_0}$ norm for the linear Schr\"odinger equation (and after extraction of a subsequence) $\lim_{n\to\infty} \sup_{t\in I_n} \|\tilde{w}_{Ln}^1(t)\|_{H^{s_0}}=\eps_0>0$, and the same strategy as in Steps 2,3 would yield that $u_n$ scatters for large $n$, a contradiction. We have proved
$$u_n(0,x)=\varphi^1_L(-t_{1,n},x-x_{1,n})+o(1)\text{ in }H^{s_0}.$$
By \eqref{sct33}, $t_{1,n}$ must be bounded, and we can assume $t_{1,n}=0$ for all $n$, i.e.
\begin{equation}
\label{simple_expansion}
u_n(0,x)=\varphi^1(0,x-x_{1,n})+o(1)\text{ in }H^{s_0},
\end{equation} 
which concludes the proof of the claim.
\end{proof}

\appendix

\section{Equivalence of Sobolev norms}
\label{A:equivalence}
In this appendix we prove \eqref{equivalence}. We recall Mikhlin multiplier theorem \cite{Mi56}, \cite[Theorem 2.5]{Ho60}: if $m\in C^{\infty}\left( \R^d\setminus \{0\} \right)$ is such that $|\xi|^{\alpha}\partial_{\xi}^{\alpha}m(\xi)$ is bounded for all multi-indices $\alpha$ such that $|\alpha|\leq 1+\frac{n}{2}$, then $m(D)$ is a bounded operator from $L^p$ to $L^p$ for $1<p<\infty$. 

Using Mikhlin multiplier theorem with $m(\xi)=|\xi|^s(1+|\xi|^2)^{-s/2}$ then with $m(\xi)=(1+|\xi|^2)^{-s/2}$, we see that 
$$ \|u\|_{L^p}+\left\| |\nabla|^s u\right\|_{L^p}\lesssim \left\| \langle \nabla\rangle^su\right\|_{L^p}.$$
By Mikhlin theorem with $m(\xi)=\frac{(1+|\xi|^2)^{s/2}}{1+|\xi|^s}$, we obtain 
$$ \left\| \langle \nabla\rangle^su\right\|_{L^p}\lesssim \left\| (1+|\nabla|^s)u\right\|_{L^p}\lesssim \|u\|_{L^p}+\left\| |\nabla|^s u\right\|_{L^p},$$
which concludes the proof of \eqref{equivalence}.

\bibliographystyle{abbrv}
\bibliography{paper}
  
\end{document}